\documentclass[11pt]{article}
\usepackage{amsthm, amsmath, amssymb, amsfonts, url, booktabs, tikz, setspace, fancyhdr, bm}
\usepackage{geometry}
\usepackage{hyperref, enumerate}
\usepackage[shortlabels]{enumitem}
\usepackage[babel]{microtype}
\usepackage[english]{babel}
\usepackage[capitalise]{cleveref}
\usepackage{comment}
\usepackage{bbm}
\usepackage{csquotes}
\usepackage{mathabx}
\usepackage{tikz}
\usepackage{graphicx}
\usepackage{float}
\usepackage{xcolor}
\usepackage{mathtools}
\usepackage{mathrsfs}

\usetikzlibrary{positioning, arrows.meta, shapes.geometric}

\counterwithin{figure}{section}

\newtheorem{theorem}{Theorem}[section]
\newtheorem{prop}[theorem]{Proposition}
\newtheorem{lemma}[theorem]{Lemma}

\newtheorem{conj}[theorem]{Conjecture}
\newtheorem{claim}[theorem]{Claim}

\usetikzlibrary{decorations.pathmorphing}

\theoremstyle{definition}
\newtheorem{defn}[theorem]{Definition}
\newtheorem*{defn-non}{Definition}

\newtheorem{ques}[theorem]{Question}

\newtheorem{rmk}[theorem]{Remark}

\newlist{Case}{enumerate}{2}
\setlist[Case, 1]{%
    label           =   {\bfseries Case \arabic*.},
    labelindent=1em ,labelwidth=1.3cm, labelsep*=1em, leftmargin =!
}
\setlist[Case, 2]{%
    label           =   {\bfseries Subcase \arabic{Casei}.\arabic*.},
    labelindent=-1em ,labelwidth=1.3cm, labelsep*=1em, leftmargin =!
}

\newenvironment{poc}{\begin{proof}[Proof of the claim]}{\end{proof}}

\usepackage{todonotes}

\newcommand{\F}{\mathbb{F}}

\newcommand{\ii}{\mathrm{i}}

\newcommand{\T}{\mathbb T}
\newcommand{\R}{\mathbb R}
\newcommand{\Z}{\mathbb Z}

\title{Kusner's conjecture: Exact values and linear bounds} 
\author{
Hong-Jun Ge\thanks{School of Mathematical Sciences, University of Science and Technology of China,
Hefei, China. Emails: gehj22@mail.ustc.edu.cn and zy19700816@mail.ustc.edu.cn.}
\and 
Zixiang Xu\thanks{School of Mathematical Sciences, Zhejiang University, Hangzhou, China. Email: zixiangxu@zju.edu.cn. }
\and 
Yang Zhou\footnotemark[1]
}

\date{}
\begin{document}
\maketitle

\begin{abstract}
In 1983, Kusner conjectured that the 
largest equilateral set in \(\mathbb{R}^{n}\) with metric \(\ell_{p}\) has cardinality \(n+1\) when \(1<p<\infty\) and \(2n\) when \(p=1.\) This conjecture was proved only in the isolated cases $p=2$ and $p=4$, and was disproved when \(1<p<2\). The best general upper bound
\(O_p(n^{\frac{2p+2}{2p-1}})\) is due to the celebrated work of Alon and
Pudl\'ak~[GAFA, 2003]. Our main contributions include:
\begin{enumerate}
    \item[\textup{(1)}] We prove Kusner's conjecture for every dimension \(n\ge 1\) when $2\le p\le 4$. More generally, for every integer $k\ge 0$ and every $p\in[4k+2,4k+4]$, every equilateral set in $\mathbb{R}^{n}$ with metric $\ell_p$ has cardinality at most $(2k+1)n+1$. On the complementary intervals $p\in(4k,4k+2)$ with $p\geq 1$, we obtain the almost linear bound $O_p(n\log n)$.
    \item[\textup{(2)}] We also consider the analogous problem on the torus $\mathbb{T}^n$, recently initiated by Alon, where the cyclic distance makes the problem substantially more delicate than in \(\mathbb R^n\). We prove the almost linear bound $O_p(n\log n)$ for $1\le p\le 2$ and $O_p(n^{\frac{3}{2}-\frac{1}{p}})$ for every fixed real $p>2$, improving Alon's bounds $O_p(n^{2+\frac{2}{\lfloor p\rfloor}})$ for all finite $p\ge 1$.
\end{enumerate}
\end{abstract}

\section{Introduction}
A subset \(A\) of a metric space is called \emph{equilateral} if all distances
between distinct points of \(A\) are equal.
Equilateral sets are among the most basic and extensively studied configurations
in metric and discrete geometry.  Besides their classical role in Euclidean and
normed spaces, they are closely connected to coding theory, spherical codes, and
other problems in algebraic and extremal combinatorics~\cite{2018Inven,2016SIDMABukh,2021Annals,1973JoA,1966Vanlint}.  In this paper, we study
the maximum cardinality of an equilateral set in two families of spaces: the
normed spaces \((\mathbb R^n,\|\cdot \|_p)\) and their toroidal metric analogues
\((\mathbb T^n,d_p)\).

\subsection{Kusner's conjecture}
The study of equilateral sets in normed spaces is a classical topic in discrete geometry and Banach space theory. 
For a normed space $X$, let $e(X)$ denote the maximum cardinality of an equilateral set in $X$.
For \(1\le p<\infty\), we write
$\ell_p^n := (\mathbb R^n,\|\cdot\|_p)$,
where \(\mathbb R^n\) is equipped with the \(\ell_p\)-norm
$\|\boldsymbol{x}\|_p
:=
\left(\sum_{i=1}^n |x_i|^p\right)^{1/p}.$
The Euclidean case is completely understood: every equilateral set in $(\mathbb{R}^{n},\|\cdot\|_2)$ has at most $n+1$ points, with equality precisely for the vertex sets of regular simplices. Hence
\(
e(\ell_2^n)=n+1.
\)
For general \(\ell_{p}\)-norms, however, the situation is much more delicate. 
A theorem of Petty~\cite{1971Petty} shows that every equilateral set in an $n$-dimensional normed space has at most $2^n$ points, with equality only for $\ell_\infty^n$. 
On the other hand, for every $1<p<\infty$, the standard basis vectors together with a suitable multiple of the all-ones vector form an equilateral set of size $n+1$, while in $(\mathbb{R}^{n},\|\cdot\|_1)$ the set $\{\pm \boldsymbol{e}_i:i\in[n]\}$ is an equilateral set of size $2n$. 
Thus $n+1$ is the natural lower bound for finite $p>1$, and $2n$ is the natural lower bound for $p=1$.

In his celebrated list of open problems, Guy recorded the following conjectures of Kusner~\cite{1983openProblem}, which is the most famous question pertaining the above results.
\begin{conj}[Kusner's conjecture]
    \(e(\ell_1^n)=2n\) and \(e(\ell_p^n)=n+1\) for all \(1<p<\infty.\)
\end{conj}
These questions have driven much of the subsequent work on equilateral sets in $(\mathbb{R}^{n},\|\cdot\|_p)$. The exact picture is currently known only in a few cases. 
The Euclidean case $p=2$ is classical. 
Swanepoel~\cite{2004ple2} proved the quartic case
\(
e(\ell_4^n)=n+1\) for all \(n\ge 1\) via polynomial methods, also see another proof in~\cite{2014Canada}. For $e(\ell_1^n)$, the conjectured value $2n$ is known only in small dimensions: it holds for $n\le 4$, with the cases $n=3$ and $n=4$ established by Bandelt, Chepoi and Laurent~\cite{1998DCG} and Koolen, Laurent and Schrijver~\cite{2000DCC}, respectively. 

A major turning point was Swanepoel's discovery that Kusner's conjecture is false throughout the whole range $1<p<2$ in sufficiently high dimension~\cite{2004ple2}. 
More precisely, for every fixed $1<p<2$, Swanepoel~\cite{2004ple2} constructed equilateral sets in $(\mathbb{R}^{n},\|\cdot\|_p)$ of cardinality at least $(1+\varepsilon_p)n$ for some $\varepsilon_p>0$. 
Thus the difficult part of Kusner's conjecture is not the sub-Euclidean regime $1<p<2$, which is now known to be false, but rather the super-Euclidean regime $p>2$.

As for upper bounds, the first nontrivial polynomial estimate goes back to Smyth~\cite{2001Smyth}, who proved
\(
e(\ell_p^n)\le C_{p}n^{\frac{p+1}{p-1}}
\)
for every fixed $1<p<\infty$, see Smyth's survey~\cite{2013Smyth}. 
This was substantially sharpened by Alon and Pudl\'ak~\cite{2003GAFA}, who established
\(
e(\ell_p^n)\le c_{p}n^{\frac{2p+2}{2p-1}}
\)
for every fixed \(p\ge 1\)
together with the stronger bound
\(
e(\ell_p^n)\le c_{p}n\log n
\)
for every odd integer $p\ge 1$. For even integers $p$, Swanepoel~\cite{2004ple2} obtained linear bounds:
\[
e(\ell_p^n)\le
\begin{cases}
\left(\frac{p}{2}-1\right)\cdot n+1, & p\equiv 0 \pmod 4,\\[1ex]
\frac{p}{2}\cdot n+1, & p\equiv 2 \pmod 4.
\end{cases}
\]
Later, Konyagin~\cite{2011Konyagin} gave a different approach which yields
\(
e(\ell_p^n)\le C_{p}n\log n\)
for all \(1<p<2.\)
This recovers the best currently known order of magnitude for $e(\ell_1^n)$ and extends the logarithmic bound from odd integers to the entire sub-Euclidean range. 
More recently, Chen, Gui, Tang, and Xiong~\cite{2022EUJC} obtained an improvement in the regime where $p$ is significantly larger than $n$.

Summarizing the current status, Kusner's conjecture for $1<p<\infty$ had previously exhibited a striking asymmetry. 
On the one hand, Swanepoel~\cite{2004ple2} showed that the conjecture fails for every fixed $1<p<2$ in sufficiently high dimension, while the exact simplex bound is known for the isolated cases $p=2$ and $p=4$. 
On the other hand, throughout the entire super Euclidean regime $p>2$, apart from these two exceptional points, essentially no exact general theorem was known. 
Thus, before the present work, even the most basic question, whether the equilateral dimension should remain linear in $n$, was open for every fixed $p>2$ except those even integers \(p.\)

Our main result changes this picture substantially. 
It shows that for a whole half of the parameters $p\ge 2$, the cardinality of an equilateral set is always bounded linearly in the dimension. 
\begin{theorem}\label{thm:KusnerConj}
Let $k$ and $n$ be integers with $k\ge 0$ and $n\ge 1$, and let $p\in[4k+2,4k+4]$.
If $A$ is an equilateral set in $(\mathbb R^n,\|\cdot\|_p)$, then
\[
|A|\le (2k+1)n+1.
\]
\end{theorem}
Theorem~\ref{thm:KusnerConj} has two immediate consequences. First, it gives a
linear upper bound for all
\[
p\in [2,4]\cup[6,8]\cup[10,12]\cup\cdots,
\]
a set of exponents of density one half in \([2,\infty)\). Thus Kusner's
conjecture is confirmed up to a constant factor on a large subset of the
super Euclidean range. Second, in the first interval \(2\le p\le4\), the bound is
exactly \(n+1\), and hence Kusner's conjecture holds there for every dimension
\(n\). To the best of our knowledge, this is the first verification of the
conjectured simplex bound on a continuum of exponents.

On the complementary intervals \((4k,4k+2)\), the situation is more delicate. Our method no longer directly gives a linear bound, but we can still provide an upper bound that is optimal up to a logarithmic factor. In fact, the argument in this range is not merely a variation of the previous one: it requires a  different idea, together with several additional ingredients and a more refined construction.
\begin{theorem}\label{thm:log-upper-4k-4k+2}
Let $k$ and $n$ be integers with $k\ge 0$ and $n\ge 1$, and let $p\in(4k,4k+2)$ with $p\ge 1$. If
$A$ is an equilateral set in $(\mathbb R^n,\|\cdot\|_p)$, then
\[
|A|\le C_{p}n\log{(2n)}
\]
for some constant $C_{p}>0$ depending only on $p$.
\end{theorem}

\subsection{Largest equilateral sets on the torus}
We next turn to the analogous problem on the torus, recently studied by Alon~\cite{2026AlonTorus}. 
One motivation for this toroidal version comes from
the work of Bilyk, Nagel, and Ruohoniemi on tensor product energies on
\(\mathbb T^n\), which arise naturally in discrepancy theory and
quasi-Monte Carlo integration \cite{BilykNagelRuohoniemi}.  
A key role in their work is played by Latin hypercube configurations with a
single distance vector.  In finite dimension, the single-distance-vector
property is equivalent to the same point set being equilateral with respect to
toroidal \(\ell_p\)-metric \(d_p\), \(1\le p\le\infty\).  Thus, for a
fixed \(p\), the toroidal equilateral problem can be viewed as a natural
relaxation of this more structured setting.
Let
\(
\T:=\R/\Z,
\)
which we identify with the interval $[0,1)$ endowed with the cyclic distance
\[
d_{\T}(u,v):=\min\{|u-v|,\ 1-|u-v|\}.
\]
For $\boldsymbol{x}=(x_1,\dots,x_n),\boldsymbol{y}=(y_1,\dots,y_n)\in \T^n$ and $1\le p<\infty$, define
\[
d_p(\boldsymbol{x},\boldsymbol{y})
:=
\left(\sum_{i=1}^n d_{\T}(x_i,y_i)^p\right)^{1/p}
=
\left(\sum_{i=1}^n \min\{|x_i-y_i|,\ 1-|x_i-y_i|\}^p\right)^{1/p},
\]
and for $p=\infty$, define
\(
d_{\infty}(\boldsymbol{x},\boldsymbol{y})
:=
\max_{1\le i\le n} d_{\T}(x_i,y_i).
\)
Thus $(\T^n,d_p)$ is the $n$-dimensional torus equipped with the $\ell_p$-product metric.
Similarly, we write
\(
e_p(\T^n):=e\bigl((\T^n,d_p)\bigr).
\)

While equilateral sets in normed spaces have been studied for decades, the
corresponding problem on the torus is much more recent. Alon~\cite{2026AlonTorus}
initiated the systematic study of the numbers \(e_p(\T^n)\) and proved a general
polynomial upper bound together with natural linear lower bounds.
\begin{theorem}[Alon~\cite{2026AlonTorus}]
\label{thm:Alon_torus_lp}
For every fixed real $p\ge 1$, there exists a constant $c(p)>0$ such that for every integer $n\ge 1$,
\[
2n\le e_p(\T^n)\le c(p)n^{2+2/\lfloor p\rfloor}.
\]
Moreover, if $2n+1$ is a prime, then
\[
e_p(\T^n)\ge 2n+1.
\]
\end{theorem}
In the same work, Alon also determined the exact value in the extremal case $p=\infty$, proving that
\(
e_{\infty}(\T^n)=3^n.
\)
For finite $p$, however, the correct order of magnitude remained widely open. 
The lower bounds above strongly suggest that $e_p(\T^n)$ should be linear in $n$, and Alon explicitly conjectured that this is indeed the case, in particular that
\(
e_1(\T^n)=\Theta(n)
\)
for all sufficiently large $n$.

Our next result gives the upper bounds that substantially improve Alon's
general polynomial estimate for all finite \(p\ge 1\).

\begin{theorem}\label{thm:torus-nlogn-ple2}
Let \(p\ge 1\). Then there exists a constant \(C_p>0\) such that for every
integer \(n\ge1\),
\[
e_p(\T^n)\le
\begin{cases}
C_p n\log(2n),&1\le p\le 2,\\[1mm]
C_p n^{\frac{3}{2}-\frac{1}{p}},&p>2.
\end{cases}
\]
\end{theorem}
In particular, for \(1\le p\le2\), this brings the toroidal analogue of
Kusner-type problems to within a logarithmic factor of the conjectured linear
behavior.  For \(p>2\), it gives a substantially smaller exponent than Alon's
bound \(O_p(n^{2+2/\lfloor p\rfloor})\).

We also give two constructions which strengthen the known lower bounds.  The
first extends Alon's construction from prime values of \(2n+1\) to odd prime
powers.

\begin{theorem}\label{thm:odd-prime-power-lower-bound}
Let $q=r^{m}$ be an odd prime power
for some odd prime $r$ and some integer $m\ge 1$.
If
\(
n=\frac{q-1}{2},
\)
then there exists a set \(X\subseteq \mathbb T^n\) of cardinality
\(2n+1\)
such that \(X\) is equilateral in \((\mathbb T^n,d_p)\) for every
\(1\le p\le \infty\). 
\end{theorem}

The second construction shows that, for suitable choices \(p=p_q\) depending
on a prime \(q\), the lower bound can be improved 
along an infinite sequence of dimensions.

\begin{theorem}\label{thm:lower-bound-4n-plus-1}
There exist infinitely many primes \(q\) with the following property: there is a real number
\(
p_q>\frac{q-1}{4}
\)
such that for every integer \(m\ge1\), if
\(
n=\frac{q^m-1}{4},
\)
then
\(
e_{p_q}(\T^n)\ge 4n+1.
\)
\end{theorem}

\section{Kusner's conjecture for \(p\in [4k+2,4k+4]\)}
We start with two pieces of notation that will be used throughout this section.
The first is a convenient way to write distance matrices.  For
$\boldsymbol{x},\boldsymbol{y}\in \mathbb R^n$ and $q>0$, write
\(
\ell_q(\boldsymbol{x},\boldsymbol{y})
:=
(\sum_{r=1}^n |x_r-y_r|^q)^{1/q}.
\)
Thus
\(
\ell_p(\boldsymbol{x},\boldsymbol{y})
=
\|\boldsymbol{x}-\boldsymbol{y}\|_p,
\)
which is the distance induced by the
\(\ell_p\)-norm on \(\mathbb R^n\).
The second is the standard notion of conditional negative definiteness.  This
notion is useful here because, after double centering, a conditionally negative
definite kernel becomes a positive semidefinite matrix, which is exactly the
linear algebraic mechanism behind our rank argument.

\begin{defn}
Let $X$ be a set. A symmetric kernel $K:X\times X\to \mathbb R$ is called
\emph{conditionally negative definite} on $X$ if for any $t\ge 2$ points
$x_1,\dots,x_t\in X$ and coefficients $a_1,\dots,a_t\in\mathbb R$ with
\(
\sum_{i=1}^t a_i=0,
\)
we have
\(
\sum_{i,j=1}^t a_i a_j K(x_i,x_j)\le 0.
\)
Moreover, $K$ is called \emph{strictly conditionally negative definite} on $X$ if
\(
\sum_{i,j=1}^t a_i a_j K(x_i,x_j)<0
\)
whenever $\sum_{i=1}^t a_i=0$ and $(a_1,\dots,a_t)\ne (0,\dots,0)$.
\end{defn}
\subsection{Useful tools}
We collect the auxiliary results needed for the proof. The main input is a
one-dimensional kernel decomposition. Its role is to correct the kernel
\(|x-y|^p\), on a prescribed finite set of real numbers, by a quadratic term
and finitely many oscillatory terms so that the remaining kernel is
conditionally negative definite. The crucial point is the sign control on the
oscillatory coefficients: at most \(k\) of them are positive. This is what
ultimately leads to the factor \((2k+1)n\) in the rank bound.

\begin{theorem}\label{thm:cnd}
Let \(p\in[4k+2,4k+4]\) for some integer \(k\ge 0\), and set
\(
d:=\lfloor p/2\rfloor.
\)
Let \(A=\{x_1,\dots,x_N\}\subseteq \mathbb R\) be a finite set. Then there exist
constants \(\theta_0=\theta_0(p,A)>0\) and \(C_p(A)>0\) such that, for every
\(0<\theta<\theta_0\), there exist a symmetric function
\(E_\theta:A\times A\to\mathbb R\) and real coefficients
\(
b_0=b_0(\theta), b_1=b_1(\theta),\dots,b_{d-1}=b_{d-1}(\theta)
\)
with
\(
|E_\theta(x,y)|\le C_p(A)\theta^2
\)
for \(x,y\in A\), such that the kernel
\[
K_\theta(x,y)
:=
b_0|x-y|^2
+\sum_{h=1}^{d-1} b_h\bigl(1-\cos(h\theta(x-y))\bigr)
-|x-y|^p
+E_\theta(x,y)
\]
is conditionally negative definite on \(A\). Moreover, among the coefficients
\(b_1,\dots,b_{d-1}\), at most \(k\) are positive.
\end{theorem}

The remaining tools are elementary linear algebraic facts. The first one explains
why conditionally negative definite kernels are useful for us: after double
centering, they become positive semidefinite matrices.

\begin{lemma}\label{lem:cnd posit}
Let \(X=\{x_1,x_2,\dots,x_m\}\) be a finite set, and let
\(K:X\times X\to\mathbb R\) be a symmetric kernel that is conditionally negative
definite on \(X\). Then the matrix
\[
-C_m\bigl(K(x_i,x_j)\bigr)_{i,j=1}^mC_m
\]
is positive semidefinite, where \(C_m=I_m-m^{-1}J_m\).
\end{lemma}

The second lemma is the rank-forcing step. It says that a positive semidefinite
matrix, after adding the centered identity \(C_m\) and a sufficiently small
symmetric error, must still have rank at least \(m-1\).

\begin{lemma}\label{lem:rank2}
Let \(C_m:=I_m-m^{-1}J_m\). Assume that
\[
G=NN^{\mathsf T}+C_m+E,
\]
where \(E\) is a symmetric matrix satisfying \(\lambda_{\min}(E)>-1\). Then
\[
\operatorname{rank}(G)\ge m-1.
\]
\end{lemma}

The last lemma is the low-rank input for the oscillatory terms. After centering, each basic cosine distance matrix is controlled by the
centered cosine Gram matrix, which has rank at most \(2\).

\begin{lemma}\label{lem:rank d}
Let \(u_1,\dots,u_m\in\mathbb R\), let \(\theta>0\) and \(r\ge1\), and define
\[
D:=\bigl(1-\cos(r\theta(u_i-u_j))\bigr)_{i,j=1}^m.
\]
Then the matrix \(-C_mDC_m\) is positive semidefinite and has rank at most \(2\),
where \(C_m:=I_m-m^{-1}J_m\).
\end{lemma}

\subsection{Sketch of the proof}

We give a short overview of the proof before entering the details. The argument is a rank comparison. Starting from an equilateral set
\(
A=\{\boldsymbol{x}_1,\dots,\boldsymbol{x}_m\}\subseteq\mathbb R^n
\)
with \(p\in[4k+2,4k+4]\), we normalize so that
\(
\ell_p(\boldsymbol{x}_i,\boldsymbol{x}_j)=1
\)
for all \(i\ne j\). Our goal is to prove
\(
m\le (2k+1)n+1.
\)
The strategy is to build a matrix identity whose right hand side has rank at least \(m-1\), while its left hand side has rank at most \((2k+1)n\).

The key input is the one-dimensional kernel decomposition in Theorem~\ref{thm:cnd}. Applied to the finite set of all coordinate values appearing in \(A\), it expresses the kernel \(|x-y|^p\), up to a small error, as a sum of a quadratic term, finitely many oscillatory terms of the form \(1-\cos(h\theta(x-y))\), and a conditionally negative definite remainder. The important point is not merely that such a decomposition exists, but that among the oscillatory coefficients \(b_1,\dots,b_{d-1}\), at most \(k\) are positive. After summing over all coordinates, the full distance matrix of \(A\) is decomposed into a Euclidean quadratic part, a controlled number of positive-coefficient oscillatory parts, a conditionally negative definite part, and a small error.

The next step is to double-center the resulting matrix identity by
\[
C_m=I_m-m^{-1}J_m.
\]
This is where the equilateral condition enters in a very clean way. Since the matrix
\(
\Lambda=\bigl(\ell_p(\boldsymbol{x}_i,\boldsymbol{x}_j)^p\bigr)_{i,j=1}^m
\)
is equal to \(J_m-I_m\), double centering gives
\(
C_m\Lambda C_m=-C_m.
\)
Thus the equilateral condition contributes exactly the centered identity \(C_m\), with the correct sign after rearrangement.

The conditionally negative definite part is favorable after this centering operation. Indeed, Lemma~\ref{lem:cnd posit} says that if \(K\) is conditionally negative definite, then its centered matrix \(-C_m(K(x_i,x_j))C_m\) is positive semidefinite. The oscillatory terms with negative coefficients are also favorable: by Lemma~\ref{lem:rank d}, each centered cosine kernel is positive semidefinite. Therefore all these favorable terms can be moved to the same side and absorbed into a positive semidefinite matrix. In schematic form, the centered identity becomes
\[
2b_0\widetilde G_2+\sum_{t=1}^n G_\theta^+(t)
=
NN^{\mathsf T}+C_m+E.
\]
Here \(\widetilde G_2\) is the centered Gram matrix coming from the quadratic term \(|x-y|^2\), the matrix \(G_\theta^+(t)\) collects the positive-coefficient oscillatory contributions from the \(t\)-th coordinate, \(NN^{\mathsf T}\) absorbs the positive semidefinite terms, and \(E\) is the small centered error.

The smallness of \(E\) is used only at this point. By taking \(\theta>0\) sufficiently small, the matrix \(E\) satisfies
\(
\lambda_{\min}(E)>-1.
\)
Lemma~\ref{lem:rank2} then applies to the right hand side and gives the lower bound
\[
\operatorname{rank}\bigl(NN^{\mathsf T}+C_m+E\bigr)\ge m-1.
\]

It remains to bound the rank of the left hand side from above. The matrix \(\widetilde G_2\) is a centered Gram matrix of \(m\) vectors in \(\mathbb R^n\), so
\(
\operatorname{rank}(\widetilde G_2)\le n.
\)
For each fixed coordinate \(t\), the matrix \(G_\theta^+(t)\) is a sum of at most \(k\) centered cosine kernels, and each has rank at most \(2\) by Lemma~\ref{lem:rank d}. Hence
\(
\operatorname{rank}(G_\theta^+(t))\le 2k.
\)
Consequently,
\[
\operatorname{rank}\Bigl(2b_0\widetilde G_2+\sum_{t=1}^n G_\theta^+(t)\Bigr)
\le n+2kn=(2k+1)n.
\]
Comparing the lower and upper rank bounds gives
\(
m-1\le (2k+1)n.
\)
\subsection{Proof of Theorem~\ref{thm:KusnerConj}}

Set \(d:=\lfloor p/2\rfloor\). Let
\(
A=\{\boldsymbol{x}_1,\boldsymbol{x}_2,\dots,\boldsymbol{x}_m\}\subseteq\mathbb R^n
\)
be an equilateral set in \((\mathbb R^n,\|\cdot\|_p)\). After scaling, we may assume that
\(
\ell_p(\boldsymbol{x}_i,\boldsymbol{x}_j)=1
\)
for all \(i\ne j\). After translating the whole set, we may further assume that
\(
\boldsymbol{x}_1=\boldsymbol{0}.
\)
Write
\(
\boldsymbol{x}_i=(x_i^{(1)},\dots,x_i^{(n)})
\)
for \(1\le i\le m\). Since \(\boldsymbol{x}_1=\boldsymbol{0}\) and \(\|\boldsymbol{x}_i-\boldsymbol{x}_1\|_p=1\), every coordinate of every point lies in \([-1,1]\). Thus
\(
\boldsymbol{x}_i\in[-1,1]^n
\)
for all \(1\le i\le m\).

We now apply the one-dimensional kernel decomposition to all coordinate values that occur in \(A\). For each coordinate \(t\in[n]\), define
\[
A_t:=\{x_i^{(t)}:1\le i\le m\}\subseteq[-1,1],
\]
and set
\(
B:=\bigcup_{t=1}^n A_t.
\)
It is clear that \(B\) is finite. Choose \(\theta>0\) sufficiently small so that
\(
\theta<\theta_0(p,B)
\)
and
\(
mnC_p(B)\theta^2<1,
\)
where \(\theta_0(p,B)\) and \(C_p(B)\) are given by Theorem~\ref{thm:cnd}. Applying Theorem~\ref{thm:cnd} to \(B\), we obtain a symmetric function \(E_\theta:B\times B\to\mathbb R\) and real coefficients
\(
b_0,b_1,\dots,b_{d-1}
\)
such that
\(
|E_\theta(x,y)|\le C_p(B)\theta^2
\)
for \(x,y\in B\), the kernel
\[
K(x,y)
:=
b_0|x-y|^2+\sum_{r=1}^{d-1}b_r\bigl(1-\cos(r\theta(x-y))\bigr)-|x-y|^p+E_\theta(x,y)
\]
is conditionally negative definite on \(B\), and among \(b_1,\dots,b_{d-1}\) at most \(k\) are positive.

For every pair of points \(\boldsymbol{x},\boldsymbol{y}\in A\), the definition of \(K\) gives
\begin{equation}\label{eq:psi-distance-corrected}
\sum_{t=1}^n\bigl(K(x^{(t)},y^{(t)})-E_\theta(x^{(t)},y^{(t)})\bigr)
+\ell_p(\boldsymbol{x},\boldsymbol{y})^p
=
b_0\|\boldsymbol{x}-\boldsymbol{y}\|^2_2
+\sum_{t=1}^n\sum_{r=1}^{d-1}
b_r\bigl(1-\cos(r\theta(x^{(t)}-y^{(t)}))\bigr).
\end{equation}
We now turn this scalar identity into a matrix identity. Define
\(
M_K:=\sum_{t=1}^n\bigl(K(x_i^{(t)},x_j^{(t)})\bigr)_{i,j=1}^m\) and
\(
M_E:=\sum_{t=1}^n\bigl(E_\theta(x_i^{(t)},x_j^{(t)})\bigr)_{i,j=1}^m,
\)
and let
\(
\Lambda:=\bigl(\ell_p(\boldsymbol{x}_i,\boldsymbol{x}_j)^p\bigr)_{i,j=1}^m.
\)
Since \(A\) is equilateral and \(\ell_p(\boldsymbol{x}_i,\boldsymbol{x}_j)=1\) for \(i\ne j\), we have
\(
\Lambda=J_m-I_m.
\)

We separate the oscillatory terms according to the signs of their coefficients. Put
\[
R_+:=\{r\in\{1,\dots,d-1\}:b_r>0\},\]
and \[
R_-:=\{r\in\{1,\dots,d-1\}:b_r<0\}
\]
respectively. By our selection and Theorem~\ref{thm:cnd}, \(|R_+|\le k\). For each \(t\in[n]\), define
\[
D_\theta^+(t):=
\sum_{r\in R_+} b_r
\bigl(1-\cos(r\theta(x_i^{(t)}-x_j^{(t)}))\bigr)_{i,j=1}^m
\]
and
\[
D_\theta^-(t):=
\sum_{r\in R_-} (-b_r)
\bigl(1-\cos(r\theta(x_i^{(t)}-x_j^{(t)}))\bigr)_{i,j=1}^m.
\]
Thus both \(D_\theta^+(t)\) and \(D_\theta^-(t)\) are nonnegative linear combinations of the basic cosine-distance matrices. Set
\(
G_\theta^+(t):=-C_mD_\theta^+(t)C_m\) and
\(
G_\theta^-(t):=-C_mD_\theta^-(t)C_m.
\)
By Lemma~\ref{lem:rank d}, each of \(G_\theta^+(t)\) and \(G_\theta^-(t)\) is positive semidefinite. Moreover,
\[
\operatorname{rank}(G_\theta^+(t))
\le
\sum_{r\in R_+}
\operatorname{rank}\Bigl(
-C_m\bigl(1-\cos(r\theta(x_i^{(t)}-x_j^{(t)}))\bigr)_{i,j=1}^mC_m
\Bigr)
\le 2|R_+|\le 2k.
\]

Recall that \(
C_m=I_m-m^{-1}J_m.\) Next let
\(
D_2:=\bigl(\|\boldsymbol{x}_i-\boldsymbol{x}_j\|_2^2\bigr)_{i,j=1}^m
\) and \(
\widetilde G_2:=-\frac12 C_mD_2C_m.
\)
Notice that
\(\widetilde G_2\) is the Gram matrix of the centered vectors
\(
\boldsymbol{x}_i-\overline{\boldsymbol{x}}\), where
\(
\overline{\boldsymbol{x}}:=\frac1m\sum_{j=1}^m\boldsymbol{x}_j.
\)
Since these vectors lie in \(\mathbb R^n\), we have
\[
\operatorname{rank}(\widetilde G_2)\le n.
\]

By \eqref{eq:psi-distance-corrected}, we have
\[
M_K-M_E+\Lambda
=
b_0D_2+\sum_{t=1}^nD_\theta^+(t)-\sum_{t=1}^nD_\theta^-(t).
\]
Since \(C_mJ_m=J_mC_m=0\) and \(C_m^2=C_m\), we have
\[
C_m\Lambda C_m=C_m(J_m-I_m)C_m=-C_m.
\]
Multiplying the preceding matrix identity by \(-C_m\) on the left and by \(C_m\) on the right, we obtain
\[
-C_mM_KC_m+C_mM_EC_m+C_m
=
2b_0\widetilde G_2+\sum_{t=1}^nG_\theta^+(t)-\sum_{t=1}^nG_\theta^-(t),
\]
which implies that
\begin{equation}\label{eq:main-rank-identity-corrected}
2b_0\widetilde G_2+\sum_{t=1}^nG_\theta^+(t)
=
\Bigl(-C_mM_KC_m+\sum_{t=1}^nG_\theta^-(t)\Bigr)+C_m+C_mM_EC_m.
\end{equation}

We then estimate the rank of the right-hand side from below. First, observe that the coordinatewise sum of \(K\) is still conditionally negative definite on \(A\). Indeed, let \(\lambda_1,\dots,\lambda_m\in\mathbb R\) satisfy \(\sum_{i=1}^m\lambda_i=0\). Then \[ \sum_{i,j=1}^m \lambda_i\lambda_j \sum_{t=1}^n K(x_i^{(t)},x_j^{(t)}) = \sum_{t=1}^n \sum_{i,j=1}^m \lambda_i\lambda_j K(x_i^{(t)},x_j^{(t)}). \] For each fixed \(t\), all points \(x_i^{(t)}\) belong to \(B=\bigcup_{t=1}^n A_t\). Since \(K\) is conditionally negative definite on \(B\), the inner sum is non-positive. If some of the values \(x_i^{(t)}\) coincide, we simply combine the corresponding coefficients: for \(z\in B\), set \( \mu_z:=\sum_{i:x_i^{(t)}=z}\lambda_i. \) Then \(\sum_z\mu_z=0\), and \[ \sum_{i,j=1}^m \lambda_i\lambda_j K(x_i^{(t)},x_j^{(t)}) = \sum_{z,z'\in B}\mu_z\mu_{z'}K(z,z')\le0. \] Thus every coordinate contribution is conditionally negative definite, and their sum is conditionally negative definite as well. Hence the kernel \[ (\boldsymbol{x}_i,\boldsymbol{x}_j)\mapsto \sum_{t=1}^nK(x_i^{(t)},x_j^{(t)}) \] is conditionally negative definite on \(A\). By Lemma~\ref{lem:cnd posit}, \(-C_mM_KC_m\) is positive semidefinite. Since each \(G_\theta^-(t)\) is also positive semidefinite, the matrix
\[
-C_mM_KC_m+\sum_{t=1}^nG_\theta^-(t)
\]
is positive semidefinite. Since every real positive semidefinite matrix admits a Gram factorization, we may write 
\[ -C_mM_KC_m+\sum_{t=1}^nG_\theta^-(t)=NN^{\mathsf T} \]
for some real matrix \(N\).

It remains to check that the error term is small enough.
Recall that for a real symmetric matrix \(M\), \(\|M\|_{\textup{OP}}\) denotes the largest
absolute value of its eigenvalues.  Each entry of \(M_E\) has absolute value at
most \(nC_p(B)\theta^2\).  Hence, by the Gershgorin circle theorem~\cite{1931Ger},
\[
\|M_E\|_\textup{OP}
\le
\max_{1\le i\le m}\sum_{j=1}^m |(M_E)_{ij}|
\le mnC_p(B)\theta^2<1.
\]
Since \(C_m\) is an orthogonal projection,
\[
\|C_mM_EC_m\|_\textup{OP}\le \|M_E\|_\textup{OP}<1.
\]
The matrix \(C_mM_EC_m\) is symmetric, and hence
\[
\lambda_{\min}(C_mM_EC_m)>-1.
\]
Thus the right hand side of \eqref{eq:main-rank-identity-corrected} has the form
\[
NN^{\mathsf T}+C_m+E
\]
with \(E=C_mM_EC_m\) and \(\lambda_{\min}(E)>-1\).  Then Lemma~\ref{lem:rank2} gives
\[
\operatorname{rank}\Bigl(
\Bigl(-C_mM_KC_m+\sum_{t=1}^nG_\theta^-(t)\Bigr)+C_m+C_mM_EC_m
\Bigr)\ge m-1.
\]
Therefore, by \eqref{eq:main-rank-identity-corrected},
\[
m-1
\le
\operatorname{rank}\Bigl(2b_0\widetilde G_2+\sum_{t=1}^nG_\theta^+(t)\Bigr).
\]

On the other hand, the rank of the left hand side is controlled by the Euclidean part and the positive oscillatory parts:
\[
\operatorname{rank}\Bigl(2b_0\widetilde G_2+\sum_{t=1}^nG_\theta^+(t)\Bigr)
\le
\operatorname{rank}(\widetilde G_2)+\sum_{t=1}^n\operatorname{rank}(G_\theta^+(t))
\le n+2kn=(2k+1)n.
\]
Combining the two inequalities gives
\(
m-1\le (2k+1)n,
\)
and hence
\[
m\le (2k+1)n+1.
\]
This completes the proof.

\subsection{Proofs of the tools}\label{subsec:proofs-tools}
We now prove the auxiliary results stated above. 
The main task is to prove the one-dimensional decomposition, Theorem~\ref{thm:cnd}. 
The proof has two stages. 
First we show that, on any finite set of real numbers, 
the kernel \(-|x-y|^p\) can be corrected by adding a suitable even polynomial in \(x-y\), 
so that the resulting kernel is strictly conditionally negative definite.
Then we approximate this correcting polynomial by a finite cosine sum with coefficients of prescribed signs.

\subsubsection{Proof of Theorem~\ref{thm:cnd}}
We begin with the polynomial correction. 
In the range \(p\in(4k+2,4k+4)\), one has \(d=\lfloor p/2\rfloor=2k+1\).
The oddness of \(d\) is the key point:
it is what allows us to add a suitable even polynomial in \(x-y\) to \(-|x-y|^p\), so that the resulting kernel becomes strictly conditionally negative definite.

\begin{lemma}\label{lem cnd}
Let \(p\in(4k+2,4k+4)\) for some integer \(k\ge0\), and set
\(
d:=\lfloor p/2\rfloor=2k+1.
\)
Let \(A=\{x_1,\dots,x_m\}\subseteq\mathbb R\) be a finite set. Then there exist real constants
\(
C_2(A,p),C_4(A,p),\dots,C_{2d}(A,p)
\)
such that the kernel
\[
K(x,y):=\sum_{r=1}^d C_{2r}(A,p)(x-y)^{2r}-|x-y|^p
\]
is strictly conditionally negative definite on \(A\).
\end{lemma}

Before giving the proof, we keep the notation
\(A=\{x_1,\dots,x_m\}\) from the statement and introduce the following filtration by vanishing moments. 
For an even function \(f:\mathbb R\to\mathbb R\), define
\[
B_f(\boldsymbol{a},\boldsymbol{b})
:=
\sum_{i,j=1}^m a_i b_j f(x_i-x_j).\]
In particular, we write
\[
Q_f(\boldsymbol{a}):=B_f(\boldsymbol{a},\boldsymbol{a}).
\]
For each integer \(r\ge0\), define the \(r\)-th moment of \(\boldsymbol{a}=(a_1,\dots,a_m)\in\mathbb R^m\) by
\(
M_r(\boldsymbol{a}):=\sum_{i=1}^m a_i x_i^r,
\)
and set
\[
V_r:=\{\boldsymbol{a}\in\mathbb R^m:M_0(\boldsymbol{a})=M_1(\boldsymbol{a})=\cdots=M_r(\boldsymbol{a})=0\}.
\]
Thus \(V_r\) consists of coefficient vectors whose moments up to order \(r\) vanish.
In particular,
\(
V_0=\{\boldsymbol{a}\in\mathbb R^m:\sum_{i=1}^m a_i=0\}.
\)
These spaces form the decreasing filtration
\[
V_0\supseteq V_1\supseteq \cdots\supseteq V_d.
\]
The role of the correction polynomial is to make the resulting quadratic form negative definite on \(V_0\). 
We achieve this by decreasing the moment order one step at a time, passing from \(V_d\) to \(V_{d-1}\) and eventually to \(V_0\).

The first ingredient is the following elementary identity.  It shows that, on
\(V_{r-1}\), the quadratic form induced by the kernel \((x-y)^{2r}\) depends only
on the \(r\)-th moment.
\begin{lemma}\label{lem mom0}
	Let $r\ge 1$. If $\boldsymbol{a},\boldsymbol{b}\in V_{r-1}$, then
	\[
	B_{t^{2r}}(\boldsymbol{a},\boldsymbol{b})
	=
	(-1)^r\binom{2r}{r}M_r(\boldsymbol{a})M_r(\boldsymbol{b}).
	\]
	In particular:
	\begin{enumerate}
		\item[\textup{(1)}] if $\boldsymbol{a}\in V_{r-1}$ and $\boldsymbol{b}\in V_r$, then
		\(
		B_{t^{2r}}(\boldsymbol{a},\boldsymbol{b})=0;
		\)
		
		\item[\textup{(2)}] if $0\ne \boldsymbol{a}\in V_{r-1}\setminus V_r$, then
		\(
		Q_{(-1)^{r+1}t^{2r}}(\boldsymbol{a})
		=
		-\binom{2r}{r}\,M_r(\boldsymbol{a})^2<0.
		\)
	\end{enumerate}
\end{lemma}
\begin{proof}[Proof of Lemma~\ref{lem mom0}]
	Expanding $(x_i-x_j)^{2r}$ gives
	\[
	(x_i-x_j)^{2r}
	=
	\sum_{\ell=0}^{2r}(-1)^{2r-\ell}\binom{2r}{\ell}x_i^\ell x_j^{2r-\ell}.
	\]
	Therefore
	\[
	B_{t^{2r}}(\boldsymbol{a},\boldsymbol{b})
	=
	\sum_{\ell=0}^{2r}(-1)^{2r-\ell}\binom{2r}{\ell}
	M_\ell(\boldsymbol{a})M_{2r-\ell}(\boldsymbol{b}).
	\]
	Since $\boldsymbol{a},\boldsymbol{b}\in V_{r-1}$, we have
	\(
	M_\ell(\boldsymbol{a})=M_\ell(\boldsymbol{b})=0\)
	for every \(\ell<r.\)
	Hence every term vanishes except the middle term $\ell=r$, and so
	\[
	B_{t^{2r}}(\boldsymbol{a},\boldsymbol{b})
	=
	(-1)^r\binom{2r}{r}M_r(\boldsymbol{a})M_r(\boldsymbol{b}).
	\]
	The two stated consequences now follow directly from this formula. If \(\boldsymbol{b}\in V_r\), then \(M_r(\boldsymbol{b})=0\), and hence \( B_{t^{2r}}(\boldsymbol{a},\boldsymbol{b})=0 \)
    for every \(\boldsymbol{a}\in V_{r-1}\). 
    On the other hand, if \(0\ne\boldsymbol{a}\in V_{r-1}\setminus V_r\), then \(M_r(\boldsymbol{a})\ne0\). Therefore \[ Q_{(-1)^{r+1}t^{2r}}(\boldsymbol{a}) = (-1)^{r+1}B_{t^{2r}}(\boldsymbol{a},\boldsymbol{a}) = -\binom{2r}{r}M_r(\boldsymbol{a})^2<0. \]
\end{proof}

The second ingredient is the positivity of $|x-y|^p$ on $V_d$ when
$d=\lfloor p/2\rfloor$ is odd. Recall that \(Q_f(\boldsymbol{a}):=B_f(\boldsymbol{a},\boldsymbol{a}).\)

\begin{lemma}\label{lem odd d}
	Let $d\ge 0$ be an integer, and let $p$ satisfy
	\(
	2d<p<2d+2.
	\)
	If $d$ is odd, then the quadratic form $Q_{|t|^p}$ is positive
	definite on $V_d$, that is,
	\(
	Q_{|t|^p}(\boldsymbol{a})>0\) for every 
	\(\boldsymbol{0}\ne \boldsymbol{a}\in V_d.
	\)
\end{lemma}
\begin{proof}[Proof of Lemma~\ref{lem odd d}]
	Define
	\[
	R_d(u):=\cos u-\sum_{\ell=0}^d(-1)^\ell\frac{u^{2\ell}}{(2\ell)!}.
	\]
	Since
	\(
	\cos u=\sum_{\ell=0}^{\infty}(-1)^\ell\frac{u^{2\ell}}{(2\ell)!},
	\)
	we have $R_d(u)=O(u^{2d+2})$ as $u\to 0$.
	Also, $R_d(u)=O(u^{2d})$ as $|u|\to\infty$. Therefore, the integral
	\[
	J_{p,d}:=\int_0^\infty \frac{R_d(u)}{u^{1+p}}\,du
	\]
	converges absolutely for $2d<p<2d+2$.
	
	Integrating by parts \(2d\) times, we claim that all boundary terms vanish.
Indeed, at the \((j+1)\)-st integration by parts, where \(0\le j\le 2d-1\),
the boundary term is a constant multiple of
\(
R_d^{(j)}(u)u^{j-p}.
\)
As \(u\to0\), we have
\[
R_d^{(j)}(u)=O(u^{2d+2-j}),
\]
and hence
\[
R_d^{(j)}(u)u^{j-p}=O(u^{2d+2-p})=o(1),
\]
because \(p<2d+2\).  As \(u\to\infty\), we have
\[
R_d^{(j)}(u)=O(u^{2d-j}),
\]
and hence
\[
R_d^{(j)}(u)u^{j-p}=O(u^{2d-p})=o(1),
\]
because \(p>2d\). Thus all boundary terms vanish. We therefore obtain
\[
J_{p,d}
=
\frac{(-1)^d}{p(p-1)\cdots(p-2d+1)}
\int_0^\infty \frac{\cos u-1}{u^{1+(p-2d)}}\,du.
\]
	As $\cos u-1\le 0$, we have
	\[
	(-1)^{d+1}J_{p,d}>0.
	\]
	Set
	\(
	c_{p,d}:=\frac{1}{|J_{p,d}|}>0.
	\)
	Then by the change of variables $u=t|x|$, for every $x\in\mathbb R$, we have
	\begin{equation}\label{eq xp}
		|x|^p
		=
		c_{p,d}(-1)^{d+1}
		\int_0^\infty
		\frac{
			R_d(t|x|)
		}{t^{1+p}}\,dt
		=
		c_{p,d}(-1)^{d+1}
		\int_0^\infty
		\frac{
			\cos(tx)-\sum_{\ell=0}^d(-1)^\ell (tx)^{2\ell}/(2\ell)!
		}{t^{1+p}}\,dt.
	\end{equation}
	
	For $\boldsymbol{a}=(a_1,\dots,a_m)\in V_d$, and \(\ell=0,1,\ldots,d,\) we have
	\[
	M_\ell(\boldsymbol{a})=\sum_{i=1}^m a_i x_i^\ell=0.
	\]
	
	Applying \eqref{eq xp} to $x_i-x_j$, summing over $i,j$, and using
	Lemma~\ref{lem mom0} to eliminate all polynomial terms, we obtain
	\[
	Q_{|t|^p}(\boldsymbol{a})
	=
	c_{p,d}(-1)^{d+1}
	\int_0^\infty
	\frac{
		\sum_{i,j=1}^m a_i a_j \cos(t(x_i-x_j))
	}{t^{1+p}}\,dt.
	\]
	By Euler's formula $e^{\ii \theta}=\cos \theta+\ii \sin \theta$, we have
	\[
	\sum_{i,j=1}^m a_i a_j e^{\ii t(x_i-x_j)}
	=
	\left(\sum_{i=1}^m a_i e^{\ii tx_i}\right)
	\left(\sum_{j=1}^m a_j e^{-\ii tx_j}\right)
	=
	\left|\sum_{i=1}^m a_i e^{\ii tx_i}\right|^2.
	\]
	Taking real parts gives
	\[
	\sum_{i,j=1}^m a_i a_j \cos(t(x_i-x_j))
	=
	\left|\sum_{i=1}^m a_i e^{\ii tx_i}\right|^2.
	\]
	It follows that
	\begin{equation}\label{eq Qp}
		Q_{|t|^p}(\boldsymbol{a})
		=
		c_{p,d}(-1)^{d+1}
		\int_0^\infty
		\frac{
			\left|\sum_{i=1}^m a_i e^{\ii tx_i}\right|^2
		}{t^{1+p}}\,dt.
	\end{equation}
	Because $d$ is odd, we have $(-1)^{d+1}=1$, so the right hand side is nonnegative.
	
	It remains to prove strict positivity. Let
	\[
	F_{\boldsymbol{a}}(t):=\sum_{i=1}^m a_i e^{\ii tx_i}.
	\]
	This is an entire function. If \(F_{\boldsymbol{a}}\equiv0\), then all derivatives at \(0\) vanish, so
\[
\sum_{i=1}^m a_i x_i^r=0
\]
for every \(r\ge0\). In particular, this holds for \(0\le r\le m-1\). Since \(x_1,\dots,x_m\) are distinct, the Vandermonde matrix is invertible, and hence \(a_1=\cdots=a_m=0\). Thus \(F_{\boldsymbol{a}}\not\equiv0\) whenever \(\boldsymbol{a}\ne\boldsymbol{0}\). The integrand in \eqref{eq Qp} is therefore not identically zero, and the integral is strictly positive. This proves the lemma.
\end{proof}

We are now ready to prove Lemma~\ref{lem cnd}.

\begin{proof}[Proof of Lemma~\ref{lem cnd}]
	If \(m=1\), the conclusion is trivial, so assume \(m\ge2\). Set
\(
s:=\min\{d,m-1\}.
\)
We first show that \(V_r\) has codimension \(1\) inside \(V_{r-1}\).
\begin{claim}\label{claim:CodimOne}
    For each
$r\in\{1,\dots,s\}$, 
\(
\dim(V_{r-1}/V_r)=1.
\)
\end{claim}
\begin{poc}
Notice that the
linear functionals \(M_0,\dots,M_s\) are linearly independent on \(\mathbb R^m\).
To see this, suppose that
\(
\sum_{r=0}^s c_rM_r=0.
\)
Then the polynomial
\(
q(x):=\sum_{r=0}^s c_rx^r
\)
vanishes at the distinct real numbers \(x_1,\dots,x_m\). Since \(s\le m-1\), this
forces \(q\equiv0\), and hence \(c_0=\cdots=c_s=0\). Therefore each time we pass
from \(V_{r-1}\) to \(V_r\), we impose exactly one new independent linear
condition, namely \(M_r(\boldsymbol a)=0\). Thus
\(
\dim(V_{r-1}/V_r)=1
\)
for every \(1\le r\le s\).
\end{poc}
Also, if \(s<m-1\), then necessarily \(s=d\), while if
\(s=m-1\le d\), then \(V_s=\{\boldsymbol0\}\).
	
	We shall construct the correcting coefficients by descending induction. For \(1\le r\le s\), define
\[
K_r(x,y):=\sum_{j=r}^s C_{2j}(A,p)(x-y)^{2j}-|x-y|^p.
\]
We will choose the coefficients so that
\begin{equation}\label{eq:negative_on_Vr_minus_1}
Q_{K_r}(\boldsymbol{a})<0
\end{equation}
for every \(0\ne\boldsymbol{a}\in V_{r-1}\).
	
	Start with
\(
K_{s+1}(x,y):=-|x-y|^p.
\)
We claim that \(Q_{K_{s+1}}\) is negative definite on \(V_s\). Indeed, if \(s=d\), this follows from Lemma~\ref{lem odd d}, since \(d=2k+1\) is odd. If \(s=m-1<d\), then \(V_s=\{\boldsymbol{0}\}\).

	Fix $r\in\{1,\dots,s\}$, and assume that the coefficients
	$C_{2(r+1)}(A,p),\dots,C_{2s}(A,p)$ have already been chosen so that
	\(
	Q_{K_{r+1}}(\boldsymbol{a})<0\) for all
	\(\boldsymbol{0}\ne \boldsymbol{a}\in V_r.
	\)
	Thus the bilinear form \(B_{K_{r+1}}\) is nondegenerate on \(V_r\): if
\(B_{K_{r+1}}(\boldsymbol u,\boldsymbol v)=0\) for every
\(\boldsymbol v\in V_r\), then in particular
\(Q_{K_{r+1}}(\boldsymbol u)=0\), forcing \(\boldsymbol u=0\).
	
	Consider the subspace
	\[
	W_r:=\{\boldsymbol{w}\in V_{r-1}: B_{K_{r+1}}(\boldsymbol{w},\boldsymbol{v})=0\ \text{for all }\boldsymbol{v}\in V_r\}.
	\]
	Since $B_{K_{r+1}}|_{V_r}$ is nondegenerate, we have
	\(
	W_r\cap V_r=\{\boldsymbol{0}\}.
	\)
	Consider the linear map
\[
\Phi:V_{r-1}\longrightarrow V_r^*
\]
defined as follows: for each \(\boldsymbol w\in V_{r-1}\), the image
\(\Phi(\boldsymbol w)\) is the linear functional on \(V_r\) given by
\[
\Phi(\boldsymbol w)(\boldsymbol v)
=
B_{K_{r+1}}(\boldsymbol w,\boldsymbol v)
\]
for every \(\boldsymbol v\in V_r\). Since \(B_{K_{r+1}}\) is nondegenerate on
\(V_r\), the restriction of \(\Phi\) to \(V_r\) is an isomorphism from \(V_r\)
onto \(V_r^*\). In particular,
\[
\dim(\operatorname{im}\Phi)=\dim V_r.
\]
	Moreover,
	\(
	\ker(\Phi)=W_r.
	\)
	Hence, by rank-nullity and Claim~\ref{claim:CodimOne},
	\[
	\dim W_r=\dim V_{r-1}-\dim V_r=1,
	\]
	and therefore
	\(
	V_{r-1}=W_r\oplus V_r.
	\)
	
	Choose a nonzero vector $\boldsymbol{w}_r\in W_r$. Since $W_r\cap V_r=\{\boldsymbol{0}\}$, we have
	$\boldsymbol{w}_r\notin V_r$, and so $M_r(\boldsymbol{w}_r)\ne 0$. By
	Lemma~\ref{lem mom0},
	\begin{equation}\label{eq Qneg}
		Q_{(-1)^{r+1}t^{2r}}(\boldsymbol{w}_r)
		=
		-\binom{2r}{r}M_r(\boldsymbol{w}_r)^2<0.
	\end{equation}
	Also, by the same lemma, for every $\boldsymbol{w}\in W_r\subseteq V_{r-1}$ and every $\boldsymbol{v}\in V_r$,
	\(
	B_{t^{2r}}(\boldsymbol{w},\boldsymbol{v})=0.
	\)
	
	Since $\boldsymbol{w}_r$ is fixed, by \eqref{eq Qneg}, we can choose $\lambda_r>0$ so large that
	\[
	Q_{K_{r+1}}(\boldsymbol{w}_r)+\lambda_r\,Q_{(-1)^{r+1}t^{2r}}(\boldsymbol{w}_r)<0.
	\]
	Set
	\(
	C_{2r}(A,p):=\lambda_r(-1)^{r+1},
	\)
	and define
	\[
	K_r(x,y):=K_{r+1}(x,y)+\lambda_r(-1)^{r+1}(x-y)^{2r}.
	\]
	
	We claim that $Q_{K_r}$ is negative definite on $V_{r-1}$.
	Let $\boldsymbol{0}\ne \boldsymbol{u}\in V_{r-1}$. Write uniquely
	\[
\boldsymbol{u}=\boldsymbol{w}+\boldsymbol{v}
	\]
    for \(\boldsymbol{w}\in W_r,\ \boldsymbol{v}\in V_r.\)
	Since $W_r$ is $B_{K_{r+1}}$-orthogonal to $V_r$, and
	$B_{t^{2r}}(W_r,V_r)=0$, we have
	\[
	Q_{K_r}(\boldsymbol{u})=Q_{K_r}(\boldsymbol{w})+Q_{K_r}(\boldsymbol{v}).
	\]
	Now:
	\begin{itemize}
		\item if $\boldsymbol{v}\ne \boldsymbol{0}$, then
		\(
		Q_{K_r}(\boldsymbol{v})=Q_{K_{r+1}}(\boldsymbol{v})<0,
		\)
		because $\boldsymbol{v}\in V_r$ and $Q_{t^{2r}}(\boldsymbol{v})=0$ by
		Lemma~\ref{lem mom0};
		
		\item if $\boldsymbol{w}\ne \boldsymbol{0}$, then $\boldsymbol{w}=\alpha \boldsymbol{w}_r$ for some $\alpha\ne 0$, and hence
		\[
		Q_{K_r}(\boldsymbol{w})
		=
		\alpha^2\Bigl(
		Q_{K_{r+1}}(\boldsymbol{w}_r)+\lambda_r\,Q_{(-1)^{r+1}t^{2r}}(\boldsymbol{w}_r)
		\Bigr)<0.
		\]
	\end{itemize}
	Therefore $Q_{K_r}(\boldsymbol{u})<0$ whenever $\boldsymbol{u}\ne \boldsymbol{0}$, and so
	\eqref{eq:negative_on_Vr_minus_1} holds. This completes the descending induction.
	
	At the end, we obtain coefficients $C_{2r}(A,p)$ for $1\le r\le s$ such that
	\(
	Q_{K_1}(\boldsymbol{a})<0\) for all
	\(\boldsymbol{0}\ne \boldsymbol{a}\in V_0.
	\)
	Finally, set
	\(
	C_{2r}(A,p):=0\) for
	\(s<r\le d\).
	Then
	\[
	K(x,y):=\sum_{r=1}^d C_{2r}(A,p)\,(x-y)^{2r}-|x-y|^p
	\]
	satisfies
	\(
	Q_K(\boldsymbol{a})<0\)
	for all \(\boldsymbol{0}\ne \boldsymbol{a}\in V_0.\)
	Since
	\[
	V_0=\Bigl\{\boldsymbol{a}\in\mathbb R^m:\sum_{i=1}^m a_i=0\Bigr\},
	\]
	this exactly says that $K$ is strictly conditionally negative definite on $A$.
\end{proof}

We now pass from the polynomial correction to the trigonometric correction required in Theorem~\ref{thm:cnd}.

\begin{proof}[Proof of Theorem~\ref{thm:cnd}]
If \(|A|=1\), the conclusion is trivial. We assume \(|A|\ge2\).

We first handle the case \(d=1\). If \(p=2\), take \(b_0=1\) and \(E_\theta\equiv0\). If \(2<p<4\), then Lemma~\ref{lem cnd} gives a coefficient \(a_2\) such that the kernel
\[
a_2(x-y)^2-|x-y|^p
\]
is conditionally negative definite on \(A\). In this case take \(b_0=a_2\) and \(E_\theta\equiv0\). Since there are no oscillatory terms when \(d=1\), the sign assertion is vacuous. Thus we may assume \(d\ge2\).

We first choose an even polynomial correction. If \(p\in(4k+2,4k+4)\), then Lemma~\ref{lem cnd} gives real coefficients \(a_2,a_4,\dots,a_{2d}\) such that the kernel
\[
K_{\mathrm{poly}}(x,y)
:=
\sum_{r=1}^d a_{2r}(x-y)^{2r}-|x-y|^p
\]
is conditionally negative definite on \(A\). If \(p\in\{4k+2,4k+4\}\), then \(p=2d\). In this case take
\[
(a_2,a_4,\dots,a_{2d-2},a_{2d})=(0,0,\dots,0,1),
\]
so that
\[
K_{\mathrm{poly}}(x,y)=(x-y)^{2d}-|x-y|^{2d}=0.
\]
Thus in all cases \(K_{\mathrm{poly}}\) is conditionally negative definite on \(A\).

We then approximate the polynomial part by a finite cosine sum. For \(\theta>0\), define \(b_1(\theta),\dots,b_{d-1}(\theta)\) as the unique solution of
\begin{equation}\label{eq bi}
\sum_{h=1}^{d-1} b_h(\theta)h^{2r}
=
(-1)^{r+1}\frac{(2r)!}{\theta^{2r}}a_{2r}
\end{equation}
for \( r=2,\ldots,d.\)
The coefficient matrix \((h^{2r})_{2\le r\le d,1\le h\le d-1}\) is invertible. Indeed, after writing \(\lambda_h=h^2\), it becomes a Vandermonde matrix in the distinct numbers \(\lambda_1,\dots,\lambda_{d-1}\), up to multiplication of columns by positive factors. Define
\begin{equation}\label{eq b0}
b_0(\theta):=a_2-\frac{\theta^2}{2}\sum_{h=1}^{d-1}b_h(\theta)h^2
\end{equation}
and set
\[
T_\theta(t):=b_0(\theta)t^2+\sum_{h=1}^{d-1}b_h(\theta)(1-\cos(h\theta t)).
\]

Let
\[
\Delta:=A-A=\{x-y:x,y\in A\}.
\]
Since \(\Delta\) is finite, all estimates below are uniform for \(t\in\Delta\). For each \( h\in\{1,\dots,d-1\}\), Taylor's formula gives
\[
1-\cos(h\theta t)
=
\sum_{r=1}^d (-1)^{r+1}\frac{(h\theta t)^{2r}}{(2r)!}
+R_h(\theta,t),
\]
where \(R_h(\theta,t)=O(\theta^{2d+2})\) uniformly for \(t\in\Delta\).

Let
\[
s:=\max\{r\in\{1,\dots,d\}:a_{2r}\ne0\}.
\]
If \(s=1\), then the right hand side of \eqref{eq bi} is zero for every \(r=2,\dots,d\), and hence \(b_h(\theta)=0\) for all \(h\). In this case \(T_\theta(t)=a_2t^2\), and the approximation below is exact. We may therefore assume \(s\ge2\). Since \(a_{2r}=0\) for all \(r>s\), equation \eqref{eq bi} gives
\(
b_h(\theta)=O(\theta^{-2s})
\)
for \(1\le h\le d-1\). Hence, uniformly for \(t\in\Delta\),
\[
\sum_{h=1}^{d-1}b_h(\theta)R_h(\theta,t)=O(\theta^{2d+2-2s})=O(\theta^2).
\]
By \eqref{eq bi} and \eqref{eq b0}, the Taylor coefficients of \(T_\theta(t)\) up to degree \(2d\) agree with those of \(\sum_{r=1}^d a_{2r}t^{2r}\). Therefore, defining
\[
E_\theta(x,y):=\sum_{r=1}^d a_{2r}(x-y)^{2r}-T_\theta(x-y),
\]
we obtain
\[
|E_\theta(x,y)|\le C_p(A)\theta^2
\]
for all \(x,y\in A\), after increasing \(C_p(A)\) if necessary. Since \(T_\theta\) is even, \(E_\theta\) is symmetric.

Now define
\[
K_\theta(x,y)
:=
b_0(\theta)|x-y|^2+\sum_{h=1}^{d-1}b_h(\theta)(1-\cos(h\theta(x-y)))-|x-y|^p+E_\theta(x,y).
\]
By the definition of \(E_\theta\), we have
\[
K_\theta(x,y)=K_{\mathrm{poly}}(x,y)
\]
for all \(x,y\in A\). Hence \(K_\theta\) is conditionally negative definite on \(A\).

It remains to control the signs of the coefficients \(b_1(\theta),\dots,b_{d-1}(\theta)\). If \(s=1\), then all these coefficients are zero, so there is nothing to prove. Assume \(s\ge2\). Write \eqref{eq bi} in matrix form as
\[
M\boldsymbol b(\theta)=\boldsymbol u(\theta),
\]
where
\(
\boldsymbol b(\theta):=(b_1(\theta),\dots,b_{d-1}(\theta))^{\mathsf T}
\)
and
\[
\boldsymbol u(\theta):=
\left((-1)^{r+1}\frac{(2r)!}{\theta^{2r}}a_{2r}\right)_{r=2}^d.
\]
Since \(a_{2r}=0\) for \(r>s\) and \(a_{2s}\ne0\), we have
\[
\boldsymbol u(\theta)
=
(-1)^{s+1}(2s)!a_{2s}\theta^{-2s}\boldsymbol e_{s-1}
+O(\theta^{-2s+2}),
\]
where \(\boldsymbol e_{s-1}\) denotes the \((s-1)\)-st standard basis vector of \(\mathbb R^{d-1}\). Therefore
\begin{equation}\label{eq bim}
b_h(\theta)
=
(-1)^{s+1}(2s)!a_{2s}(M^{-1})_{h,s-1}\theta^{-2s}
+O(\theta^{-2s+2}).
\end{equation}

We now determine the sign pattern of \(M^{-1}\). For \(1\le h\le d-1\), put
\(
\lambda_h:=h^2.
\)
Recall that the rows of \(M\) are indexed by \(r=2,\dots,d\), while the columns
are indexed by \(h=1,\dots,d-1\), and
\[
M_{r,h}=h^{2r}=\lambda_h^r.
\]
Equivalently, after relabeling the row \(r=q+1\), where \(1\le q\le d-1\), the
corresponding entry is
\[
h^{2r}=h^{2(q+1)}=\lambda_h^{q+1}=\lambda_h^{q-1}\lambda_h^2.
\]
Thus, after this relabeling of the rows,
\[
M=VD,
\]
where
\(
V=(\lambda_h^{q-1})_{1\le q,h\le d-1}\)
and
\(
D=\operatorname{diag}(\lambda_1^2,\dots,\lambda_{d-1}^2).
\)

Since \(D\) has positive diagonal entries, \(M^{-1}=D^{-1}V^{-1}\) has the same
sign pattern as \(V^{-1}\). Write \(V^{-1}=(\alpha_{h,q})_{1\le h,q\le d-1}\).
The \(h\)-th Lagrange interpolation polynomial for the nodes
\(\lambda_1,\dots,\lambda_{d-1}\) is
\[
L_h(X):=
\prod_{\substack{1\le j\le d-1\\ j\ne h}}
\frac{X-\lambda_j}{\lambda_h-\lambda_j}.
\]
Since the rows of \(V\) correspond to the monomials
\(1,X,\dots,X^{d-2}\), we may write
\[
L_h(X)
=
\alpha_{h,1}+\alpha_{h,2}X+\cdots+\alpha_{h,d-1}X^{d-2}.
\]
Thus \(\alpha_{h,q}\) is the coefficient of \(X^{q-1}\) in \(L_h(X)\). The
coefficient of \(X^{q-1}\) in
\[
\prod_{\substack{1\le j\le d-1\\ j\ne h}}(X-\lambda_j)
\]
has sign \((-1)^{d-1-q}\), because all \(\lambda_j\) are positive. On the other
hand,
\[
\operatorname{sgn}\left(\prod_{j\ne h}(\lambda_h-\lambda_j)\right)
=
(-1)^{d-1-h}.
\]
Therefore
\[
\operatorname{sgn}(\alpha_{h,q})
=
(-1)^{d-1-q}(-1)^{d-1-h}
=
(-1)^{h+q}.
\]
It follows that
\[
\operatorname{sgn}((M^{-1})_{h,q})
=
\operatorname{sgn}((V^{-1})_{h,q})
=
(-1)^{h+q}.
\]

Applying this with \(q=s-1\) in \eqref{eq bim}, we get, for all sufficiently small \(\theta\),
\[
\operatorname{sgn}(b_h(\theta))
=
\operatorname{sgn}(a_{2s})(-1)^{s+1}(-1)^{h+s-1}
=
\operatorname{sgn}(a_{2s})(-1)^h.
\]
Thus the nonzero coefficients \(b_1(\theta),\dots,b_{d-1}(\theta)\) have alternating signs.

If \(p\in[4k+2,4k+4)\), then \(d=2k+1\), so there are \(d-1=2k\) oscillatory coefficients, and an alternating sign pattern gives at most \(k\) positive ones. If \(p=4k+4\), then \(d=2k+2\), \(s=d\), and \(a_{2s}=a_{2d}=1\). Hence positive signs can occur only for even \(h\), namely among \(h=2,4,\dots,2k\), so again there are at most \(k\) positive coefficients.

Finally, choose \(\theta_0=\theta_0(p,A)>0\) sufficiently small so that both the error estimate above and the sign conclusions hold for every \(0<\theta<\theta_0\). This proves the theorem.
\end{proof}

\subsubsection{Proofs of remaining lemmas}
We finish this subsection by proving the three elementary algebraic lemmas used in the rank argument.
\begin{proof}[Proof of Lemma~\ref{lem:cnd posit}]
Let
\[
M:=\bigl(K(x_i,x_j)\bigr)_{i,j=1}^m.
\]
For any \(\boldsymbol{v}\in\mathbb R^m\), set
\(
\boldsymbol{u}:=C_m\boldsymbol{v}.
\)
Since \(C_m\) is the orthogonal projection onto
\[
\boldsymbol{1}^{\perp}:=\{\boldsymbol{w}\in\mathbb R^m:\langle \boldsymbol{w},\boldsymbol{1}\rangle=0\},
\]
we have \(\boldsymbol{u}\in\boldsymbol{1}^{\perp}\). Therefore, by the conditional negative definiteness of \(K\),
\[
\boldsymbol{v}^{\mathsf T}(-C_mMC_m)\boldsymbol{v}
=
-\boldsymbol{u}^{\mathsf T}M\boldsymbol{u}
\ge 0.
\]
Thus \(-C_mMC_m\) is positive semidefinite.
\end{proof}

\begin{proof}[Proof of Lemma~\ref{lem:rank2}]
We show that \(G\) is positive definite on the \((m-1)\)-dimensional subspace \(\boldsymbol{1}^{\perp}\). Let
\(
\boldsymbol{v}\in\boldsymbol{1}^{\perp}\setminus\{\boldsymbol{0}\}.
\)
Since \(C_m\) is the orthogonal projection onto \(\boldsymbol{1}^{\perp}\), we have \(C_m\boldsymbol{v}=\boldsymbol{v}\). Hence
\[
\boldsymbol{v}^{\mathsf T}G\boldsymbol{v}
=
\boldsymbol{v}^{\mathsf T}NN^{\mathsf T}\boldsymbol{v}
+\boldsymbol{v}^{\mathsf T}C_m\boldsymbol{v}
+\boldsymbol{v}^{\mathsf T}E\boldsymbol{v}
=
\|N^{\mathsf T}\boldsymbol{v}\|_2^2+\|\boldsymbol{v}\|_2^2+\boldsymbol{v}^{\mathsf T}E\boldsymbol{v}.
\]
Since \(\lambda_{\min}(E)>-1\), we have
\[
\boldsymbol{v}^{\mathsf T}E\boldsymbol{v}
\ge
\lambda_{\min}(E)\|\boldsymbol{v}\|_2^2.
\]
It follows that
\[
\boldsymbol{v}^{\mathsf T}G\boldsymbol{v}
\ge
\|N^{\mathsf T}\boldsymbol{v}\|_2^2+
(1+\lambda_{\min}(E))\|\boldsymbol{v}\|_2^2
>0.
\]
Thus \(G\) is positive definite on \(\boldsymbol{1}^{\perp}\). Consequently, \(G\) has rank at least \(\dim\boldsymbol{1}^{\perp}=m-1\), as desired.
\end{proof}

\begin{proof}[Proof of Lemma~\ref{lem:rank d}]
Let
\[
H:=\bigl(\cos(r\theta(u_i-u_j))\bigr)_{i,j=1}^m.
\]
Since
\(
D=J_m-H
\)
and \(C_mJ_m=J_mC_m=0\), we have
\[
-C_mDC_m=C_mHC_m.
\]
It is therefore enough to show that \(H\) is positive semidefinite and has rank at most \(2\).

Using
\(
\cos(\alpha-\beta)=\cos\alpha\cos\beta+\sin\alpha\sin\beta,
\)
we can write
\[
H=\boldsymbol{c}\boldsymbol{c}^{\mathsf T}+\boldsymbol{s}\boldsymbol{s}^{\mathsf T},
\]
where
\(
\boldsymbol{c}:=(\cos(r\theta u_1),\dots,\cos(r\theta u_m))^{\mathsf T}
\)
and
\(
\boldsymbol{s}:=(\sin(r\theta u_1),\dots,\sin(r\theta u_m))^{\mathsf T}.
\)
Hence \(H\) is positive semidefinite and \(\operatorname{rank}(H)\le2\). Since \(C_mHC_m\) is a compression of \(H\), it is also positive semidefinite and has rank at most \(2\). Therefore \(-C_mDC_m\) is positive semidefinite and has rank at most \(2\).
\end{proof}

\subsection{Affine independence and nonregular examples for \(2<p<4\)}
We close this section with two simple observations about the equality case in the range \(2<p<4\). The first one says that every largest equilateral set must be affinely independent. Thus, in this range, the extremal configuration is always an \(n\)-simplex in the affine sense.

\begin{prop}\label{prop:equality-affine-independent}
Let \(2<p<4\), and let
\(
A=\{\boldsymbol{x}_1,\dots,\boldsymbol{x}_m\}\subseteq \mathbb R^n
\)
be an equilateral set in \((\mathbb R^n,\|\cdot\|_p)\). If \(m=n+1\), then \(A\) is affinely independent. Equivalently, \(A\) is the vertex set of an \(n\)-simplex.
\end{prop}

\begin{proof}[Proof of Proposition~\ref{prop:equality-affine-independent}]
In the proof of Theorem~\ref{thm:KusnerConj}, when \(2<p<4\), we have \(k=0\). Thus there are no oscillatory correction terms and no error term. The centered identity reduces to
\[
2b_0\widetilde G_2=-C_mM_KC_m+C_m,
\]
where
\(
\widetilde G_2:=-\frac12 C_mD_2C_m
\)
is the centered Gram matrix of \(\boldsymbol{x}_1,\dots,\boldsymbol{x}_m\), and \(M_K\) is the matrix associated with the conditionally negative definite kernel \(K\).

By Lemma~\ref{lem:cnd posit}, the matrix \(-C_mM_KC_m\) is positive semidefinite. Hence we may write
\[
-C_mM_KC_m=NN^{\mathsf T}
\]
for some real matrix \(N\). Therefore
\[
2b_0\widetilde G_2=NN^{\mathsf T}+C_m.
\]
By Lemma~\ref{lem:rank2}, applied with \(E=0\), we get
\(
\operatorname{rank}(2b_0\widetilde G_2)\ge m-1.
\)
Since \(m=n+1\ge2\), this also implies \(b_0\ne0\), and hence
\[
\operatorname{rank}(\widetilde G_2)\ge m-1.
\]
On the other hand, \(\widetilde G_2\) is the Gram matrix of \(m\) centered vectors in \(\mathbb R^n\), so
\(
\operatorname{rank}(\widetilde G_2)\le n.
\)
Since \(m=n+1\), we have
\[
\operatorname{rank}(\widetilde G_2)=n=m-1.
\]
But the rank of the centered Gram matrix is exactly the dimension of the affine span of \(A\). Hence the affine span of \(A\) has dimension \(m-1\), which means that \(\boldsymbol{x}_1,\dots,\boldsymbol{x}_m\) are affinely independent.
\end{proof}

The preceding proposition should not be confused with Euclidean regularity. It only says that a largest equilateral set is an affine simplex. In general, this simplex need not be regular with respect to the Euclidean metric.
\begin{prop}\label{prop:nonregular-largest-general-n}
Let $2<p<4$ and $n\ge 2$.
Then there exists an equilateral set of size $n+1$ in $(\mathbb R^n,\|\cdot\|_p)$
which is not a regular simplex in the Euclidean sense.
\end{prop}
\begin{proof}[Proof of Proposition~\ref{prop:nonregular-largest-general-n}]
Let \( \boldsymbol{x}_0:=\frac12\boldsymbol{e}_1 \) and \( \boldsymbol{x}_1:=-\frac12\boldsymbol{e}_1. \) Next, let $\alpha>0$ be the unique real number satisfying \begin{equation}\label{eq:ConstructionP} \bigl(\alpha+2^{-1/p}\bigr)^p+(n-2)\alpha^p=1-2^{-p}. \end{equation} Such an $\alpha$ exists and is unique, because the function \[ h(t):=\bigl(t+2^{-1/p}\bigr)^p+(n-2)t^p \] is continuous and strictly increasing when \(t\ge 0\), with \(h(0)=2^{-1}<1-2^{-p}, \) since \(p>1\), and \(h(t)\to\infty\) as \(t\to\infty\). For each \(i=1,\ldots,n-1\), define \[ \boldsymbol{x}_{i+1} := \alpha\sum_{j=2}^{n}\boldsymbol{e}_j + 2^{-1/p}\boldsymbol{e}_{i+1}. \] We claim that \( A:=\{\boldsymbol{x}_0,\boldsymbol{x}_1,\dots,\boldsymbol{x}_n\} \) is equilateral in \((\mathbb R^n,\|\cdot\|_p)\) with common distance \(1\). First, \( \|\boldsymbol{x}_0-\boldsymbol{x}_1\|_p^p=1. \) Next, if \(i\neq j\) with \(1\le i,j\le n-1\), then \(\boldsymbol{x}_{i+1}-\boldsymbol{x}_{j+1}\) has exactly two nonzero coordinates, namely \(2^{-1/p}\) and \(-2^{-1/p}\). Hence \[ \|\boldsymbol{x}_{i+1}-\boldsymbol{x}_{j+1}\|_p^p = 2\cdot 2^{-1} = 1. \] Finally, for each \(i=1,\ldots,n-1\), \[ \boldsymbol{x}_0-\boldsymbol{x}_{i+1} = \frac12\boldsymbol{e}_1 - \alpha\sum_{j=2}^{n}\boldsymbol{e}_j - 2^{-1/p}\boldsymbol{e}_{i+1}, \] and hence, by~\eqref{eq:ConstructionP}, \[ \|\boldsymbol{x}_0-\boldsymbol{x}_{i+1}\|_p^p = 2^{-p} + \bigl(\alpha+2^{-1/p}\bigr)^p + (n-2)\alpha^p = 1. \] The same computation applies to \(\boldsymbol{x}_1\), and gives \( \|\boldsymbol{x}_1-\boldsymbol{x}_{i+1}\|_p^p=1. \) Therefore \(A\) is equilateral in \((\mathbb R^n,\|\cdot\|_p)\). It remains to show that \(A\) is not a regular simplex in the Euclidean sense. We have \( \|\boldsymbol{x}_0-\boldsymbol{x}_1\|_2^2=1. \)

If \(n\ge3\), then for distinct \(i,j\in\{1,\dots,n-1\}\), \[ \|\boldsymbol{x}_{i+1}-\boldsymbol{x}_{j+1}\|_2^2 = 2\cdot 2^{-2/p} = 2^{1-2/p} \ne 1. \] Thus \(A\) is not Euclidean regular when \(n\ge3\). It remains to consider the case \(n=2\). In this case \eqref{eq:ConstructionP} gives \( \alpha+2^{-1/p} = (1-2^{-p})^{1/p}. \) Therefore \[ \|\boldsymbol{x}_0-\boldsymbol{x}_2\|_2^2 = \frac14+\bigl(\alpha+2^{-1/p}\bigr)^2 = \frac14+(1-2^{-p})^{2/p}. \] Since \(p>2\), we have \(2/p<1\), and hence \( (1-2^{-p})^{2/p}>1-2^{-p}>\frac34. \) It follows that \[ \|\boldsymbol{x}_0-\boldsymbol{x}_2\|_2^2>1 = \|\boldsymbol{x}_0-\boldsymbol{x}_1\|_2^2. \] Thus \(A\) is not Euclidean regular also when \(n=2\). This completes the proof.
\end{proof}

\section{An almost linear bound for \(e(\ell_{p}^{n})\) when \(p\in(4k,4k+2)\)}
The proof of Theorem~\ref{thm:log-upper-4k-4k+2} requires a different argument from that used in the preceding section. 
We begin with a brief sketch, highlighting where the previous rank argument breaks down and explaining the mechanism that replaces it.

\subsection{Sketch of the Proof of Theorem~\ref{thm:log-upper-4k-4k+2}}
Let \(A\subseteq\mathbb R^n\) be an equilateral set with common \(\ell_p\)-distance
\(1\), where \(p\in(4k,4k+2)\).  After translating \(A\), we may assume that one
point of \(A\) is the origin.  This normalization places all coordinates in a bounded
interval and gives a uniform coordinatewise \(p\)-mass bound, which will later make
the approximation errors add up correctly over the \(n\) coordinates.

The finite-rank argument from the preceding section cannot be closed in the
same way in the range \(p\in(4k,4k+2)\). The obstruction is a parity-driven
sign reversal at the last level of the moment filtration. In the interior of
the range treated in the preceding section, namely \(p\in(4k+2,4k+4)\), one has
\(d=\lfloor p/2\rfloor=2k+1\), and Lemma~\ref{lem odd d} gives the favorable
positive definiteness of \(Q_{|t|^p}\) on \(V_d\). This sign was the key input
which allowed the signed kernel \(-|x-y|^p\) to be corrected, up to a controlled
error term, by a quadratic term and finitely many oscillatory kernels. After
double centering, the conditionally negative contribution lay on the same side
as the \(p\)-distance term and could be absorbed into the
positive-semidefinite side, leaving only a controlled number of unfavorable
finite-rank modes to be counted.

In the complementary range \(p\in(4k,4k+2)\), however, we have
\(d=\lfloor p/2\rfloor=2k\), and the corresponding quadratic form
\(Q_{|t|^p}\) is negative definite on \(V_d\) rather than positive definite.
Thus the sign needed for the preceding absorption mechanism is reversed. The
previous finite-rank correction scheme therefore no longer produces a
conditionally negative definite correction of the signed kernel \(-|x-y|^p\)
in a way compatible with the same rank count. This reversal suggests the
appropriate replacement: the kernel to be corrected is now \(|x-y|^p\) itself,
rather than \(-|x-y|^p\).

The replacement is a zero-anchored Hilbert-space representation. Write
\(p=4k+\beta\), with \(0<\beta<2\). Since the kernel
\( |u-v|^\beta
\)
is conditionally negative definite on the real line, it admits a Hilbert-space
realization by squared distances. We use this conditionally negative kernel
after taking \(2k\)-fold zero-anchored primitives. Integrating by parts then
reduces the exponent \(p\) to the conditionally negative exponent \(\beta\).
This gives a one-dimensional representation in which the main sign-definite
part of the \(p\)-distance quadratic form becomes a squared norm in a Hilbert
space, up to finite-dimensional boundary correction terms. The anchoring at the
origin is essential: it ensures that the boundary terms produced by the
integrations remain finite-dimensional.

Applying this representation coordinate by coordinate, and then performing the
centering step, gives the analogue of the matrix identity used in the preceding
section. The centered simplex matrix is balanced against two contributions: a
finite-dimensional correction term of rank at most \(4kn\), coming from the
boundary terms, and a Gram matrix coming from the Hilbert-space component.

The crucial difference is where this Hilbert-space term appears in the rank
identity. In the previous argument, the conditionally negative contribution was
on the same side as the \(p\)-distance term. Hence, after double centering, it
could be absorbed into the positive-semidefinite side, while the rank estimate
was applied to the opposite side. In the present range, the sign is reversed:
the Hilbert-space Gram term appears on the side to which the rank estimate must
be applied. Thus it is no longer harmless. Although its Gram matrix has finite
rank for the finite set under consideration, that rank may be as large as the
size of the set itself, so inserting it directly would destroy the rank count.

The rest of the proof is therefore an approximation argument. The resulting Hilbert map
has a scaling structure, which allows its image to be approximated by dyadic
shells. Using \(L\) shells gives an \(O_p(L)\)-dimensional approximation in one
coordinate, and hence an \(O_p(nL)\)-dimensional approximation after taking the
direct sum over all coordinates. A coordinatewise \(p\)-mass estimate controls
the accumulation of the approximation errors over the \(n\) coordinates.
Choosing \(L=O_{p}(\log{n})\), and then applying a suitable projection lower
bound to the resulting centered Gram identity, yields
\(
|A|\le C_p n\log(2n).
\)

\subsection{A zero-anchored primitive and an exact bridge}
Fix \(p\in(4k,4k+2)\) with \(p\ge1\), and write
\(
\beta:=p-4k\in(0,2).
\)
The zero anchored primitive below is needed only when \(k\ge1\). Thus, for the
rest of this lemma, assume \(k\ge1\) and set
\(
n_0:=2k-1.
\)
The case \(k=0\), namely \(1\le p<2\), will be handled separately in
Proposition~\ref{prop:zero-anchored-exact-bridge-log}.
The next function should be viewed as a compactly supported \(2k\)-fold primitive of the point mass \(\delta_x\), normalized so that all lower-order boundary terms are anchored at the origin. This anchoring is what later produces only finitely many correction terms. For a real number \(u\), write \( u_+:=\max\{u,0\}.\)
\begin{lemma}\label{lem:zero-anchored-primitive-log}
For each $x\in[-1,1]$, define
\[
\psi_x^{(0)}(t):=
\begin{cases}
\dfrac{(x-t)_+^{\,n_0}}{n_0!}
-
\displaystyle\sum_{r=0}^{n_0}
\frac{x^r}{r!}\,
\frac{(-t)_+^{\,n_0-r}}{(n_0-r)!},
& x\ge 0,\\[3ex]
\dfrac{(t-x)_+^{\,n_0}}{n_0!}
-
\displaystyle\sum_{r=0}^{n_0}
\frac{(-x)^r}{r!}\,
\frac{t_+^{\,n_0-r}}{(n_0-r)!},
& x\le 0.
\end{cases}
\]
Then the following hold.

\begin{enumerate}
\item[\textnormal{(i)}]
\(
\operatorname{supp}\psi_x^{(0)}
\subseteq [\min\{x,0\},\max\{x,0\}].
\)

\item[\textnormal{(ii)}]
In the sense of distributions,
\[
\bigl(\psi_x^{(0)}\bigr)^{(2k)}
=
\delta_x-\sum_{r=0}^{2k-1}\frac{(-1)^r x^r}{r!}\,\delta_0^{(r)}.
\]

\item[\textnormal{(iii)}]
For every $\alpha>-1$ and every integer $r\in\{0,1,\dots,2k-1\}$,
\[
\int_{\mathbb R}\psi_x^{(0)}(t)\,\operatorname{sgn}(t)^r|t|^\alpha\,dt
=
\frac{\Gamma(\alpha+1)}{\Gamma(\alpha+2k+1)}
\operatorname{sgn}(x)^r|x|^{\alpha+2k}.
\]
In particular,
\(
\int_{\mathbb R}\psi_x^{(0)}(t)\,dt=\frac{|x|^{2k}}{(2k)!},
\)
and
\[
\int_{\mathbb R}\psi_x^{(0)}(t)\,|t|^\beta\,dt
=
\frac{\Gamma(\beta+1)}{\Gamma(p-2k+1)}|x|^{p-2k}.
\]
\end{enumerate}
\end{lemma}
\begin{proof}[Proof of Lemma~\ref{lem:zero-anchored-primitive-log}]
    Suppose first that $x\ge 0$. For $t<0$, we have
\[
(x-t)^{n_0}
=
\sum_{r=0}^{n_0}\binom{n_0}{r}x^r(-t)^{n_0-r},
\]
hence the two terms in the definition of $\psi_x^{(0)}$ cancel identically on
$(-\infty,0)$. Also, for $t>x$, the first term vanishes and the second term is
already zero. Therefore
\(
\operatorname{supp}\psi_x^{(0)}\subseteq [0,x].
\)
Suppose now that $x\le 0$. For $t>0$, we have
\[
(t-x)^{n_0}
=
\sum_{r=0}^{n_0}\binom{n_0}{r}(-x)^r t^{n_0-r},
\]
so the two terms in the definition of $\psi_x^{(0)}$ cancel identically on
$(0,\infty)$. Also, for $t<x$, the first term vanishes and the second term is
already zero. Hence
\(
\operatorname{supp}\psi_x^{(0)}\subseteq [x,0].
\)
This proves \textnormal{(i)}.

We next prove \textnormal{(ii)}. If $x\ge 0$, then
\(
\frac{d^{2k}}{dt^{2k}}\frac{(x-t)_+^{n_0}}{n_0!}=\delta_x.
\)
Also for $0\le r\le n_0$, we have
\[
\frac{d^{2k}}{dt^{2k}}
\left(
\frac{(-t)_+^{n_0-r}}{(n_0-r)!}
\right)
=
(-1)^r\delta_0^{(r)}.
\]
Therefore
\[
\bigl(\psi_x^{(0)}\bigr)^{(2k)}
=
\delta_x-\sum_{r=0}^{2k-1}\frac{(-1)^r x^r}{r!}\delta_0^{(r)}.
\]
If $x\le 0$, then
\(
\frac{d^{2k}}{dt^{2k}}\frac{(t-x)_+^{n_0}}{n_0!}=\delta_x,
\)
and
\(
\frac{d^{2k}}{dt^{2k}}
\left(
\frac{t_+^{n_0-r}}{(n_0-r)!}
\right)
=
\delta_0^{(r)}.
\)
Since $(-x)^r=(-1)^r x^r$, we again obtain
\[
\bigl(\psi_x^{(0)}\bigr)^{(2k)}
=
\delta_x-\sum_{r=0}^{2k-1}\frac{(-1)^r x^r}{r!}\delta_0^{(r)}.
\]
This proves \textnormal{(ii)}.

Finally, we prove \textnormal{(iii)}. If $x=0$, then $\psi_0^{(0)}\equiv 0$, and the
claim is trivial. Assume first that $x>0$. By \textnormal{(i)}, $\psi_x^{(0)}$ is supported
on $[0,x]$, where $\operatorname{sgn}(t)^r=1$. Thus
\[
\int_{\mathbb R}\psi_x^{(0)}(t)\,\operatorname{sgn}(t)^r |t|^\alpha\,dt
=
\int_0^x \frac{(x-t)^{2k-1}}{(2k-1)!}\, t^\alpha\,dt.
\]
After the change of variables $t=xu$, this becomes
\[
\frac{x^{\alpha+2k}}{(2k-1)!}\int_0^1 (1-u)^{2k-1}u^\alpha\,du
=
\frac{x^{\alpha+2k}}{(2k-1)!}B(\alpha+1,2k)
=
\frac{\Gamma(\alpha+1)}{\Gamma(\alpha+2k+1)}x^{\alpha+2k}.
\]

If $x<0$, then $\psi_x^{(0)}$ is supported on $[x,0]$, where
$\operatorname{sgn}(t)^r=\operatorname{sgn}(x)^r$. Hence
\[
\int_{\mathbb R}\psi_x^{(0)}(t)\,\operatorname{sgn}(t)^r |t|^\alpha\,dt
=
\operatorname{sgn}(x)^r
\int_x^0 \frac{(t-x)^{2k-1}}{(2k-1)!}\, |t|^\alpha\,dt.
\]
Writing $x=-a$ with $a>0$, and changing variables $t=-au$, we obtain
\[
\operatorname{sgn}(x)^r
\frac{a^{\alpha+2k}}{(2k-1)!}
\int_0^1 (1-u)^{2k-1}u^\alpha\,du
=
\frac{\Gamma(\alpha+1)}{\Gamma(\alpha+2k+1)}
\operatorname{sgn}(x)^r |x|^{\alpha+2k}.
\]
This proves \textnormal{(iii)}.
\end{proof}

We now package the one-dimensional identity into a Hilbert space bridge. We use
Hilbert spaces only in a concrete linear algebraic way. Namely, whenever a
symmetric bilinear form is positive semidefinite, we may quotient out its null
space and complete the resulting inner product space, which produces a real
Hilbert space. In such a space, finite Gram matrices are positive semidefinite
and orthogonal projections onto finite dimensional subspaces are available. In
what follows, this construction is used to represent the conditionally negative
definite kernel \(|x-y|^\beta\), with \(0<\beta<2\), as a squared norm.

Although in this section we only need the case \(k\ge1\), we state the bridge
also for \(k=0\), where it reduces to the standard Hilbert space embedding of
conditionally negative definite kernels. In particular, it will be useful in the
proof of Lemma~\ref{lem:real-line-part-ple2}.

\begin{prop}\label{prop:zero-anchored-exact-bridge-log}
Let $p\in(4k,4k+2)$ for some integer $k\geq 0$, and set
	\(
	\beta:=p-4k\in(0,2).
	\)
	If $p\geq 1$,
then there exist:
\begin{itemize}
\item a real Hilbert space $\mathcal H_\beta$,
\item a continuous map $\Gamma_0:[-1,1]\to \mathcal H_\beta$,
\item a map $\boldsymbol w_0:[-1,1]\to \mathbb R^{4k}$,
\item symmetric positive semidefinite matrices
\(
M_{+,0},M_{-,0}\in\mathbb R^{4k\times 4k},
\)
\item linear operators $T_a:\mathcal H_\beta\to\mathcal H_\beta$ for $0<a\le 1$,
\end{itemize}
such that the following hold.

\medskip
\noindent
\textnormal{(i)}
For every finite set $X=\{x_1,\dots,x_m\}\subseteq [-1,1]$ and every
$\boldsymbol\lambda=(\lambda_1,\dots,\lambda_m)\in \boldsymbol{1}^{\perp}$,
\[
-\frac12\sum_{i,j=1}^m \lambda_i\lambda_j |x_i-x_j|^p
=
c_{p,k}\left\|\sum_{i=1}^m \lambda_i\Gamma_0(x_i)\right\|_{\mathcal H_\beta}^2 +
\left\|\sum_{i=1}^m \lambda_i M_{+,0}^{1/2}\boldsymbol w_0(x_i)\right\|_2^2
-
\left\|\sum_{i=1}^m \lambda_i M_{-,0}^{1/2}\boldsymbol w_0(x_i)\right\|_2^2,
\]
where
\(
c_{p,k}:=\frac{\Gamma(p+1)}{\Gamma(\beta+1)}.
\)

\medskip
\noindent
\textnormal{(ii)}
For all $0<a\le 1$ and all $x\in[-1,1]$,
\(
\Gamma_0(ax)=T_a\Gamma_0(x),
\)
and for \(\xi\in\mathcal H_\beta,\)
\[
\|T_a\xi\|_{\mathcal H_\beta}=a^{p/2}\|\xi\|_{\mathcal H_\beta}.
\]
\end{prop}
\begin{proof}[Proof of Proposition~\ref{prop:zero-anchored-exact-bridge-log}]
If \(k=0\), then \(1\le p<2\). Since \(|x-y|^p\) is conditionally negative definite on \(\R\), there exist a real Hilbert space \(\mathcal H_p\) and a map \(\Gamma_0:\R\to\mathcal H_p\) such that
\[
-\frac12\sum_{i,j}\lambda_i\lambda_j|x_i-x_j|^p
=
\left\|\sum_i\lambda_i\Gamma_0(x_i)\right\|^2.
\]
We take \(\boldsymbol w_0\) to be the map into \(\mathbb R^0\), and \(M_{+,0},M_{-,0}\) to be the corresponding \(0\times0\) matrices. The second statement follows from homogeneity.

   Assume now that \(k\ge1\). We use signed measures only as a compact notation for combining ordinary integrals and point masses: \(f(t)\,dt\) contributes \(\int \varphi(t)f(t)\,dt\), while \(\delta_y\) contributes \(\varphi(y)\). Let
\[
\mathcal M_0:=
\Bigl\{
\mu:\mu\ \text{is a compactly supported signed Borel measure on }\mathbb R,\
\mu(\mathbb R)=0
\Bigr\}.
\]
Define a symmetric bilinear form on $\mathcal M_0$ by
\[
\langle \mu,\nu\rangle_\beta
:=
-\frac12\iint_{\mathbb R^2}|u-v|^\beta\,d\mu(u)\,d\nu(v).
\]
Since $0<\beta<2$, the kernel $(u,v)\rightarrow |u-v|^\beta$ is conditionally negative
definite on $\mathbb R$, and therefore $\langle\cdot,\cdot\rangle_\beta$ is positive
semidefinite on $\mathcal M_0$. Quotienting by the null space and completing, we
obtain a real Hilbert space $\mathcal H_\beta$.

For $x\in[-1,1]$, let $\psi_x^{(0)}$ be as in
Lemma~\ref{lem:zero-anchored-primitive-log}, and define
\[
\nu_x^{(0)}:=\psi_x^{(0)}(t)\,dt-\frac{|x|^{2k}}{(2k)!}\delta_0.
\]
By Lemma~\ref{lem:zero-anchored-primitive-log}\textnormal{(iii)}, $\nu_x^{(0)}\in \mathcal M_0$.
Define
\(
\Gamma_0(x):=[\nu_x^{(0)}]\in \mathcal H_\beta.
\)
Now let
\(
X=\{x_1,\dots,x_m\}\subseteq [-1,1],
\)
\(
\boldsymbol\lambda=(\lambda_1,\dots,\lambda_m)\in \boldsymbol{1}^{\perp},
\)
and define
\[
Q_{p,X}(\boldsymbol\lambda):=
-\frac12\sum_{i,j=1}^m \lambda_i\lambda_j|x_i-x_j|^p.
\]
Set
\(
F_\lambda^{(0)}(t):=\sum_{i=1}^m \lambda_i\psi_{x_i}^{(0)}(t),\)
\(
\mu_\lambda:=\sum_{i=1}^m \lambda_i\delta_{x_i},
\)
\(
M_r(\boldsymbol\lambda):=\sum_{i=1}^m \lambda_i x_i^r
\)
for \(1\le r\le 2k\),
\[
S_r(\boldsymbol\lambda):=\sum_{i=1}^m \lambda_i\,\operatorname{sgn}(x_i)^r|x_i|^{p-r}
\]
for \(1\le r\le 2k.\)
By Lemma~\ref{lem:zero-anchored-primitive-log}\textnormal{(ii)},
\[
(F_\lambda^{(0)})^{(2k)}
=
\mu_\lambda-\sum_{r=0}^{2k-1}\frac{(-1)^r M_r(\boldsymbol\lambda)}{r!}\delta_0^{(r)}.
\]
Since $\boldsymbol\lambda\in\boldsymbol{1}^{\perp}$, we have $M_0(\boldsymbol\lambda)=0$, and therefore
\begin{equation}\label{eq:primitive-distribution-log}
(F_\lambda^{(0)})^{(2k)}
=
\mu_\lambda-\sum_{r=1}^{2k-1}\frac{(-1)^r M_r(\boldsymbol\lambda)}{r!}\delta_0^{(r)}.
\end{equation}
Define also
\[
\nu_\lambda^{(0)}:=F_\lambda^{(0)}(t)\,dt-\frac{M_{2k}(\boldsymbol\lambda)}{(2k)!}\delta_0.
\]
By Lemma~\ref{lem:zero-anchored-primitive-log}\textnormal{(iii)}, $\nu_\lambda^{(0)}\in\mathcal M_0$, and
\(
\sum_{i=1}^m \lambda_i\Gamma_0(x_i)=[\nu_\lambda^{(0)}].
\)
Let \(\|\cdot\|_{\mathcal H_\beta}\) denote the norm in the Hilbert
space \(\mathcal H_\beta\), induced by the inner product
\[
\langle \mu,\nu\rangle_\beta
:=
-\frac12\iint_{\mathbb R^2}|u-v|^\beta\,d\mu(u)\,d\nu(v).
\]
Since
\(
\sum_{i=1}^m\lambda_i\Gamma_0(x_i)=[\nu_\lambda^{(0)}]\in\mathcal H_\beta,
\)
we have, by the definition of this Hilbert-space norm,
\begin{equation}\label{eq:nu-energy-log}
\left\|\sum_{i=1}^m \lambda_i\Gamma_0(x_i)\right\|_{\mathcal H_\beta}^2
=
-\frac12\iint |u-v|^\beta\,d\nu_\lambda^{(0)}(u)\,d\nu_\lambda^{(0)}(v).
\end{equation}

We now compute $Q_{p,X}(\boldsymbol\lambda)$.

\paragraph{Bulk--bulk term.}
Set
\[
Q^{\mathrm{bulk}}
:=
-\frac12\left\langle
(F_\lambda^{(0)})^{(2k)}\otimes (F_\lambda^{(0)})^{(2k)},
\,|x-y|^p
\right\rangle,
\]
where the tensor product is understood in the usual distributional sense, more precisely,
for distributions \(S,T\) on \(\mathbb R\), the distribution \(S\otimes T\) acts on a
function \(\Phi(x,y)\) by first applying \(T\) in the \(y\)-variable and then \(S\)
in the \(x\)-variable.
Since
\[
\partial_x^{2k}\partial_y^{2k}\left(-\frac12|x-y|^p\right)
=
c_{p,k}\left(-\frac12|x-y|^\beta\right),
\]
integration by parts yields
\[
Q^{\mathrm{bulk}}
=
-\frac{c_{p,k}}2
\iint_{\mathbb R^2}
F_\lambda^{(0)}(u)F_\lambda^{(0)}(v)|u-v|^\beta\,du\,dv.
\]
On the other hand, Lemma~\ref{lem:zero-anchored-primitive-log}\textnormal{(iii)} gives
\[
\int_{\mathbb R}F_\lambda^{(0)}(t)\,dt
=
\frac{M_{2k}(\boldsymbol\lambda)}{(2k)!}\]
and
\[
\int_{\mathbb R}|t|^\beta F_\lambda^{(0)}(t)\,dt
=
\frac{\Gamma(\beta+1)}{\Gamma(p-2k+1)}\,S_{2k}(\boldsymbol\lambda).
\]
Expanding \eqref{eq:nu-energy-log}, and using
\(c_{p,k}:=\frac{\Gamma(p+1)}{\Gamma(\beta+1)},\) we obtain
\begin{equation}\label{eq:bulk-identity-log}
Q^{\mathrm{bulk}}
=
c_{p,k}
\left\|\sum_{i=1}^m \lambda_i \Gamma_0(x_i)\right\|_{\mathcal H_\beta}^2
-
\binom{p}{2k}M_{2k}(\boldsymbol\lambda)S_{2k}(\boldsymbol\lambda).
\end{equation}

\paragraph{Bulk--boundary term.}
For $1\le r\le 2k-1$, define
\(
h_r(x):=\operatorname{sgn}(x)^r |x|^{p-r}.
\) From now on, we use the notation \[ (p)_r:=p(p-1)\cdots(p-r+1) \] for the falling factorial.
Moreover, we use \(\langle \delta_0^{(r)},\cdot\rangle_y\) to denote that the distribution \(\delta_0^{(r)}\) acts on the \(y\)-variable, with \(x\) fixed. Therefore, for \(x\neq0\),
\[
\left\langle \delta_0^{(r)},\,|x-y|^p\right\rangle_y = (-1)^r \partial_y^r |x-y|^p\big|_{y=0}.
\]
Since 
\[
\partial_y^r |x-y|^p\big|_{y=0} = (-1)^r (p)_r \operatorname{sgn}(x)^r |x|^{p-r},
\]
we obtain \[ \left\langle \delta_0^{(r)},\,|x-y|^p\right\rangle_y = (p)_r h_r(x). \]
Using \eqref{eq:primitive-distribution-log}, the total bulk--boundary contribution is
\[
Q^{\mathrm{cross}}
=
-\sum_{r=1}^{2k-1}\frac{(-1)^r M_r(\boldsymbol\lambda)}{r!}\,
\left\langle (F_\lambda^{(0)})^{(2k)},\,(p)_r h_r\right\rangle.
\]
Here the factor $2$ coming from the two cross terms is cancelled by the
prefactor $-\frac12$ in $Q_{p,X}(\boldsymbol\lambda)$.

Notice that $p-r>2k$, so $h_r\in C^{2k}(\mathbb R)$, and
\[
h_r^{(2k)}(x)
=
(p-r)_{2k}\,\operatorname{sgn}(x)^r |x|^{p-r-2k}.
\]
Hence, by integration by parts and
Lemma~\ref{lem:zero-anchored-primitive-log}\textnormal{(iii)} with
$\alpha=p-r-2k>-1$,
\[
\left\langle (F_\lambda^{(0)})^{(2k)},\,h_r\right\rangle
=
S_r(\boldsymbol\lambda).
\]
Since $\frac{(p)_r}{r!}=\binom{p}{r}$, summing over $r=1,\dots,2k-1$ gives
\begin{equation}\label{eq:cross-identity-log}
Q^{\mathrm{cross}}
=
-\sum_{r=1}^{2k-1}(-1)^r\binom{p}{r}M_r(\boldsymbol\lambda)S_r(\boldsymbol\lambda).
\end{equation}

\paragraph{Boundary--boundary term.}
Every boundary--boundary term is of the form
\[
\left\langle \delta_0^{(r)}\otimes \delta_0^{(s)},\,|x-y|^p\right\rangle
=
(-1)^{r+s}\partial_x^r\partial_y^s |x-y|^p\big|_{(0,0)}.
\]
Since $r+s$ is an integer and
\(
r+s\le 4k-2<p,
\)
the derivative of order $r+s$ of $|t|^p$ exists at $0$ and equals $0$. Therefore all
boundary--boundary terms vanish.

Combining \eqref{eq:bulk-identity-log} and \eqref{eq:cross-identity-log}, we conclude that
\begin{equation}\label{eq:exact-bridge-finite-part-log}
Q_{p,X}(\boldsymbol\lambda)
=
c_{p,k}
\left\|\sum_{i=1}^m \lambda_i \Gamma_0(x_i)\right\|_{\mathcal H_\beta}^2
-
\sum_{r=1}^{2k}(-1)^r\binom{p}{r}M_r(\boldsymbol\lambda)S_r(\boldsymbol\lambda).
\end{equation}

We now encode the finite-dimensional correction. Define
\(
u_r(x):=x^r
\)
for \(1\le r\le 2k\),
\(
v_r(x):=\operatorname{sgn}(x)^r |x|^{p-r}\)
for \(1\le r\le 2k.\)
Set
\[
\boldsymbol{w}_0(x):=
\bigl(
u_1(x),\dots,u_{2k}(x),v_1(x),\dots,v_{2k}(x)
\bigr)\in\mathbb R^{4k}.
\]
Define a symmetric matrix $A_{p,k}\in\mathbb R^{4k\times 4k}$ by
\[
A_{p,k}(r,2k+r)=A_{p,k}(2k+r,r)
=
-\frac12(-1)^r\binom{p}{r}
\]
for \(1\le r\le 2k\),
and all other entries equal to $0$. Notice that if
\(
\mathbf m_{w_0}(\boldsymbol\lambda):=\sum_{i=1}^m \lambda_i \boldsymbol w_0(x_i),
\)
then by construction
\[
\mathbf m_{w_0}(\boldsymbol\lambda)^\top A_{p,k}\,\mathbf m_{w_0}(\boldsymbol\lambda)
=
-
\sum_{r=1}^{2k}(-1)^r\binom{p}{r}M_r(\boldsymbol\lambda)S_r(\boldsymbol\lambda).
\]
Thus \eqref{eq:exact-bridge-finite-part-log} becomes
\[
Q_{p,X}(\boldsymbol\lambda)
=
c_{p,k}
\left\|\sum_{i=1}^m \lambda_i \Gamma_0(x_i)\right\|_{\mathcal H_\beta}^2
+
\mathbf m_{w_0}(\boldsymbol\lambda)^\top A_{p,k}\,\mathbf m_{w_0}(\boldsymbol\lambda).
\]
Take a spectral decomposition
\(
A_{p,k}=M_{+,0}-M_{-,0}
\)
with \(
M_{+,0},M_{-,0}\succeq 0.
\)
Then
\[
Q_{p,X}(\boldsymbol\lambda)
=
c_{p,k}
\left\|\sum_{i=1}^m \lambda_i \Gamma_0(x_i)\right\|_{\mathcal H_\beta}^2 +
\left\|\sum_{i=1}^m \lambda_i M_{+,0}^{1/2}\boldsymbol w_0(x_i)\right\|_2^2
-
\left\|\sum_{i=1}^m \lambda_i M_{-,0}^{1/2}\boldsymbol w_0(x_i)\right\|_2^2.
\]
This proves \textnormal{(i)}.

We now prove \textnormal{(ii)}.  We first define the scaling operation at the
level of measures.  For \(0<a\le1\) and \(\mu\in\mathcal M_0\), let
\(\widetilde T_a\mu\) be the signed measure determined by
\[
\int_{\mathbb R}\varphi(t)\,d(\widetilde T_a\mu)(t)
=
a^{2k}\int_{\mathbb R}\varphi(au)\,d\mu(u)
\]
for every test function \(\varphi\).  In words, we scale the underlying variable by
\(u\mapsto au\), and multiply the mass by \(a^{2k}\).  This operation preserves
compact support and total mass zero, so \(\widetilde T_a\mu\in\mathcal M_0\).

For \(\mu,\nu\in\mathcal M_0\), the definition gives
\[
\begin{aligned}
\langle \widetilde T_a\mu,\widetilde T_a\nu\rangle_\beta
&=
-\frac12\iint_{\mathbb R^2}|u-v|^\beta\,
d(\widetilde T_a\mu)(u)\,d(\widetilde T_a\nu)(v)\\
&=
a^{4k}\left(
-\frac12\iint_{\mathbb R^2}|au-av|^\beta\,d\mu(u)\,d\nu(v)
\right)\\
&=
a^{4k+\beta}\langle\mu,\nu\rangle_\beta
=
a^p\langle\mu,\nu\rangle_\beta .
\end{aligned}
\]
In particular, if \(\mu\) belongs to the null space of the semidefinite form
\(\langle\cdot,\cdot\rangle_\beta\), then so does \(\widetilde T_a\mu\).  Hence
\(\widetilde T_a\) descends to the quotient space.  On the quotient, and then on
its Hilbert-space completion, we denote the induced operator by \(T_a\).  The
identity above gives
\[
\|T_a\xi\|_{\mathcal H_\beta}=a^{p/2}\|\xi\|_{\mathcal H_\beta}
\]
for every \(\xi\in\mathcal H_\beta\).

It remains to verify that this operator has the required action on the vectors
\(\Gamma_0(x)\).  By direct inspection of the definition of
\(\psi_x^{(0)}\), for every \(x\in[-1,1]\) and \(0<a\le1\),
\[
\psi_{ax}^{(0)}(t)=a^{2k-1}\psi_x^{(0)}(t/a).
\]
Equivalently, for every test function \(\varphi\),
\[
\int_{\mathbb R}\varphi(t)\psi_{ax}^{(0)}(t)\,dt
=
a^{2k}\int_{\mathbb R}\varphi(au)\psi_x^{(0)}(u)\,du.
\]
Moreover,
\(
\frac{|ax|^{2k}}{(2k)!}\delta_0
\)
is obtained from
\(
\frac{|x|^{2k}}{(2k)!}\delta_0
\)
by the same scaling rule, because the mass is multiplied by \(a^{2k}\) and the
point \(0\) remains fixed.  Therefore
\(
\nu_{ax}^{(0)}=\widetilde T_a\nu_x^{(0)}.
\)
Passing to equivalence classes in \(\mathcal H_\beta\), we obtain
\(
\Gamma_0(ax)=T_a\Gamma_0(x).
\)

Finally, we prove that \(\Gamma_0\) is continuous.  Applying
Proposition~\ref{prop:zero-anchored-exact-bridge-log}\textnormal{(i)} to the
two-point set \(\{x,y\}\) with \(\boldsymbol\lambda=(1,-1)\), we obtain
\[
|x-y|^p
=
c_{p,k}\|\Gamma_0(x)-\Gamma_0(y)\|_{\mathcal H_\beta}^2
+
\|M_{+,0}^{1/2}(\boldsymbol w_0(x)-\boldsymbol w_0(y))\|_2^2
-
\|M_{-,0}^{1/2}(\boldsymbol w_0(x)-\boldsymbol w_0(y))\|_2^2.
\]
Thus
\[
c_{p,k}\|\Gamma_0(x)-\Gamma_0(y)\|_{\mathcal H_\beta}^2
\le
|x-y|^p
+
\left|
\|M_{+,0}^{1/2}(\boldsymbol w_0(x)-\boldsymbol w_0(y))\|_2^2
-
\|M_{-,0}^{1/2}(\boldsymbol w_0(x)-\boldsymbol w_0(y))\|_2^2
\right|.
\]
Since \(\boldsymbol w_0\) is continuous on \([-1,1]\), the finite-dimensional correction tends
to \(0\) as \(y\to x\).  Hence
\[
\|\Gamma_0(x)-\Gamma_0(y)\|_{\mathcal H_\beta}\to0
\]
as \(y\to x\).  This completes the proof.

\end{proof}
\subsection{Projection and dyadic shell approximation}

We next record a projection lower bound.

\begin{prop}\label{prop:projection-logn}
Let $B=\{x_1,\dots,x_m\}$ be a finite set, let $H$ be a real Hilbert space, and let
\(
\phi:B\to H
\)
be a map. Define
\(
R
:=
\bigl(\langle \phi(x_i),\phi(x_j)\rangle_H\bigr)_{i,j=1}^m,\)
and
\(
C_m:=I_m-\frac{1}{m}J_{m}.
\)
Assume that $a>0$ and that there exists an $m\times m$ matrix $X$ with rank $r$
such that
\(
\frac{1}{a}\,C_m+X=R.
\)
Then for every orthogonal projection $P:H\to S$ onto an $s$-dimensional subspace
$S\subseteq H$, one has
\begin{equation}\label{eq:projection-lower-log}
\max_{x\in B} \|\phi(x)-P\phi(x)\|_H^2
\ge
\frac{m-r-1-s}{am}.
\end{equation}
\end{prop}
\begin{proof}[Proof of Proposition~\ref{prop:projection-logn}]
    Let
\(
Q:=I_H-P,
\)
so that $Q$ is the orthogonal projection onto $S^\perp$. For each $i\in[m]$, write
\(
\phi_i:=\phi(x_i).
\)
Define
\(
R_P
:=
\bigl(\langle P\phi_i,P\phi_j\rangle_H\bigr)_{i,j=1}^m,
\) and
\(
R_Q
:=
\bigl(\langle Q\phi_i,Q\phi_j\rangle_H\bigr)_{i,j=1}^m.
\)
Since $P$ and $Q$ are complementary orthogonal projections, we have
\(
R=R_P+R_Q,\)
therefore,
\(
R_Q=\frac1a\,C_m+X-R_P.
\)
Next, we set
\(
M:=R_P-X.
\)
Then
\[
\operatorname{rank}(M)
\le
\operatorname{rank}(R_P)+\operatorname{rank}(X)
\le s+r.
\]
Consider the subspace
\[
U:=\ker(M)\cap \boldsymbol{1}^\perp.
\]
Since $\dim(\boldsymbol{1}^\perp)=m-1$ and
\(
\dim\ker(M)=m-\operatorname{rank}(M)\ge m-(s+r),
\)
we obtain
\[
\dim U\ge m-1-(s+r).
\]
If $\boldsymbol u\in U$, then $M\boldsymbol u=0$ and $\boldsymbol u\in\boldsymbol{1}^\perp$, so $C_m \boldsymbol u=\boldsymbol u$. Therefore
\[
R_Q\boldsymbol u=\left(\frac1a\,C_m-M\right)\boldsymbol u=\frac1a\,\boldsymbol u.
\]
Thus every vector in $U$ is an eigenvector of $R_Q$ with eigenvalue $\frac{1}{a}$.
Since $R_Q$ is a Gram matrix, it is positive semidefinite. Hence
\[
\operatorname{Tr}(R_Q)\ge \frac{\dim U}{a}\ge \frac{m-1-(s+r)}{a}.
\]
On the other hand,
\(
\operatorname{Tr}(R_Q)
=
\sum_{i=1}^m \|Q\phi_i\|_H^2.
\)
Therefore
\[
\max_{1\le i\le m}\|Q\phi_i\|_H^2
\ge
\frac1m\sum_{i=1}^m \|Q\phi_i\|_H^2
=
\frac{\operatorname{tr}(R_Q)}{m}
\ge
\frac{m-r-1-s}{am}.
\]
Since $Q=I_H-P$, this is exactly \eqref{eq:projection-lower-log}.
\end{proof}

The next lemma approximates $\Gamma_0$ on dyadic shells.

\begin{lemma}\label{lem:dyadic-shell-approx-log}
Fix $\delta>0$. Then for every integer $L\ge 1$, there exists a subspace
\(
\mathcal V_L\subseteq \mathcal H_\beta
\)
such that
\(
\dim(\mathcal V_L)\le C_{p,k,\delta}\,L
\)
and for all \(x\in [-1,1]\),
\begin{equation}\label{eq:dyadic-shell-bound-log}
\operatorname{dist}\bigl(\Gamma_0(x),\mathcal V_L\bigr)^2
\le
2^p\delta^2 |x|^p + C_{p,k}\,2^{-Lp},
\end{equation}
where
\(
\operatorname{dist}(\xi,V):=\inf_{z\in V}\|\xi-z\|_{\mathcal H_\beta}.
\)
\end{lemma}
\begin{proof}[Proof of Lemma~\ref{lem:dyadic-shell-approx-log}]
    Consider the compact intervals
\(
I_+:=[1/2,1]\) and
\(
I_-:=[-1,-1/2].
\)
Since $\Gamma_0$ is continuous, the sets $\Gamma_0(I_+)$ and $\Gamma_0(I_-)$ are
compact in $\mathcal H_\beta$. Hence there exist finite sets
\(
U_+=\{u_1^+,\dots,u_{N_+}^+\}\subseteq I_+\)
and
\(
U_-=\{u_1^-,\dots,u_{N_-}^-\}\subseteq I_-
\)
such that
\(
\operatorname{dist}\bigl(\Gamma_0(x),V_+\bigr)\le \delta\) for all \(x\in I_+\), and
\(
\operatorname{dist}\bigl(\Gamma_0(x),V_-\bigr)\le \delta\) for all \(x\in I_{-},\)
where
\(
V_+:=\operatorname{span}\{\Gamma_0(u_1^+),\dots,\Gamma_0(u_{N_+}^+)\}
\)
and
\(
V_-:=\operatorname{span}\{\Gamma_0(u_1^-),\dots,\Gamma_0(u_{N_-}^-)\}.
\)

For each $j=0,1,\dots,L-1$, define
\(
V_{j,+}:=T_{2^{-j}}V_+\)
and
\(
V_{j,-}:=T_{2^{-j}}V_-.
\)
Now set
\[
\mathcal V_L:=\sum_{j=0}^{L-1}(V_{j,+}+V_{j,-}).
\]
Then
\(
\dim(\mathcal V_L)\le (N_++N_-)L=:C_{p,k,\delta}L
\) for some \(C_{p,k,\delta}>0.\) Let $x\in[-1,1]$. If $x\in(2^{-j-1},2^{-j}]$ for some $0\le j\le L-1$, write
\(
x=2^{-j}u,\) \(u\in[1/2,1].
\)
Choose $v\in V_+$ with
\(
\|\Gamma_0(u)-v\|_{\mathcal H_\beta}\le \delta.
\)
Then $T_{2^{-j}}v\in V_{j,+}\subseteq \mathcal V_L$, and by the Proposition~\ref{prop:zero-anchored-exact-bridge-log}~(ii),
\[
\operatorname{dist}\bigl(\Gamma_0(x),\mathcal V_L\bigr)^2
\le
\|T_{2^{-j}}(\Gamma_0(u)-v)\|_{\mathcal H_\beta}^2
=
2^{-jp}\|\Gamma_0(u)-v\|_{\mathcal H_\beta}^2
\le
2^{-jp}\delta^2.
\]
Since $x\in(2^{-j-1},2^{-j}]$, we have $2^{-jp}\le 2^p |x|^p$, and therefore
\[
\operatorname{dist}\bigl(\Gamma_0(x),\mathcal V_L\bigr)^2
\le
2^p\delta^2 |x|^p.
\]

If $x\in[-2^{-j},-2^{-j-1})$ for some $0\le j\le L-1$, write
\(
x=2^{-j}u,\) \( u\in[-1,-1/2].
\)
Choose $v\in V_-$ with
\(
\|\Gamma_0(u)-v\|_{\mathcal H_\beta}\le \delta.
\)
Then $T_{2^{-j}}v\in V_{j,-}\subseteq \mathcal V_L$, and similarly
\[
\operatorname{dist}\bigl(\Gamma_0(x),\mathcal V_L\bigr)^2
\le
2^{-jp}\delta^2
\le
2^p\delta^2|x|^p.
\]
Finally, if $|x|\le 2^{-L}$, then $0\in\mathcal V_L$, so
\[
\operatorname{dist}\bigl(\Gamma_0(x),\mathcal V_L\bigr)^2
\le
\|\Gamma_0(x)\|_{\mathcal H_\beta}^2.
\]
If $x\neq 0$, write $x=|x|\sigma$ with $\sigma\in\{-1,1\}$. Then
\[
\|\Gamma_0(x)\|_{\mathcal H_\beta}
=
\|T_{|x|}\Gamma_0(\sigma)\|_{\mathcal H_\beta}
=
|x|^{p/2}\|\Gamma_0(\sigma)\|_{\mathcal H_\beta}.
\]
Therefore
\[
\|\Gamma_0(x)\|_{\mathcal H_\beta}^2
\le
C_{p,k}|x|^p
\le
C_{p,k}2^{-Lp}.
\]
The same is trivial when $x=0$. This proves
\eqref{eq:dyadic-shell-bound-log}.
\end{proof}

\begin{lemma}\label{lem:centered-gram-log}
Let
\(
A=\{\boldsymbol{x}_1,\dots,\boldsymbol{x}_m\}
\)
be an equilateral set in $(\mathbb R^n,\|\cdot\|_p)$, and normalize the common distance to be $1$. After translation,
assume
\(
\boldsymbol{x}_1=\boldsymbol{0}.
\)
For each coordinate $t\in[n]$ and $s\in[-1,1]$, define
\[
\Theta_t(s):=
\bigl(\sqrt{c_{p,k}}\,\Gamma_0(s),\,M_{+,0}^{1/2}\boldsymbol w_0(s)\bigr)
\in
\mathcal H_\beta\oplus\mathbb R^{4k},
\]
and
\[
U_t(s):=M_{-,0}^{1/2}\boldsymbol w_0(s)\in\mathbb R^{4k}.
\]
For each $\boldsymbol{x}_i\in A$, define
\(
\Theta_i
:=
\bigoplus_{t=1}^n \Theta_t(x_{i,t}),
\)
\(
U_i
:=
\bigoplus_{t=1}^n U_t(x_{i,t}).
\)
Let
\(
\overline{\Theta}:=\frac1m\sum_{i=1}^m \Theta_i,
\)
\(
\overline{U}:=\frac1m\sum_{i=1}^m U_i,
\)
and set
\(
\widetilde\Theta_i:=\Theta_i-\overline{\Theta},\)
\(
\widetilde U_i:=U_i-\overline{U}.
\)
Let
\(
R:=\bigl(\langle \widetilde\Theta_i,\widetilde\Theta_j\rangle\bigr)_{i,j=1}^m\) and
\(
X:=\bigl(\langle \widetilde U_i,\widetilde U_j\rangle\bigr)_{i,j=1}^m.
\)
Then $X$ is positive semidefinite, $\operatorname{rank}(X)\le 4kn$ and
\(R=\frac12 C_m+X\), where
\(
C_m:=I_m-\frac{1}{m}J_m.
\)
\end{lemma}
\begin{proof}[Proof of Lemma~\ref{lem:centered-gram-log}]
Since each $\widetilde U_i$ belongs to $(\mathbb R^{4k})^n$, the matrix $X$ is a
Gram matrix, hence positive semidefinite, and
\(
\operatorname{rank}(X)\le 4kn.
\)

Let $\boldsymbol{\lambda}=(\lambda_1,\dots,\lambda_m)\in \boldsymbol{1}^\perp$. Here and below, we write
\[
\|\boldsymbol{\lambda}\|_2^2:=\sum_{i=1}^m\lambda_i^2.
\] Since centering does
not change the sums
\(
\sum_{i=1}^m \lambda_i\Theta_i\)
and
\(
\sum_{i=1}^m \lambda_iU_i,
\)
we can apply Proposition~\ref{prop:zero-anchored-exact-bridge-log}
coordinatewise and sum over all coordinates. Because $A$ is equilateral of common
distance $1$,
\[
-\frac12\sum_{i,j=1}^m \lambda_i\lambda_j \|\boldsymbol{x}_i-\boldsymbol{x}_j\|_p^p
=
\frac12\|\boldsymbol{\lambda}\|_2^2.
\]
On the other hand, the Proposition~\ref{prop:zero-anchored-exact-bridge-log}~(i) gives
\[
\frac12\|\boldsymbol{\lambda}\|_2^2
=
\left\|\sum_{i=1}^m \lambda_i\widetilde\Theta_i\right\|^2
-
\left\|\sum_{i=1}^m \lambda_i\widetilde U_i\right\|^2
=
\boldsymbol{\lambda}^\top (R-X)\boldsymbol{\lambda}.
\]
Since also for \(\boldsymbol{\lambda}\in\boldsymbol{1}^\perp\), we have
\[
\boldsymbol{\lambda}^\top \left(\frac12 C_m\right)\boldsymbol{\lambda}=\frac12\|\boldsymbol{\lambda}\|_2^2
\]
which further yields that
\[
\boldsymbol{\lambda}^\top\left(R-X-\frac12 C_m\right)\boldsymbol{\lambda}=0
\]
for all \(\boldsymbol{\lambda}\in \boldsymbol{1}^\perp.\) Now the matrix
\(
A_0:=R-X-\frac12 C_m
\)
is symmetric and satisfies $A_0\boldsymbol{1}=0$, because $R$, $X$, and $C_m$ all
annihilate $\boldsymbol{1}$. By polarization, $A_0$ vanishes on
$\boldsymbol{1}^\perp\times \boldsymbol{1}^\perp$, and hence $A_0=0$. This proves the lemma.
    
\end{proof}
\subsection{Proof of the logarithmic upper bound}

We are now ready to prove Theorem~\ref{thm:log-upper-4k-4k+2}.

\begin{proof}[Proof of Theorem~\ref{thm:log-upper-4k-4k+2}]
Let
\(
A=\{\boldsymbol{x}_1,\dots,\boldsymbol{x}_m\}
\)
be an equilateral set in $(\mathbb R^n,\|\cdot\|_p)$. By scaling, we may assume that the common distance is $1$.
After translation, we may assume that
\(
\boldsymbol{x}_1=\boldsymbol{0}.
\)
Then for every $\boldsymbol{x}_i=(x_{i,1},\dots,x_{i,n})\in A$ with \(x_{i,t}\in[-1,1]\), we have
\begin{equation}\label{eq:coordinate-sum-bound-log}
\sum_{t=1}^n |x_{i,t}|^p\le 1.
\end{equation}
Define $\Theta_i,U_i,\widetilde\Theta_i,\widetilde U_i,R,X$ exactly as in
Lemma~\ref{lem:centered-gram-log}. By that lemma, we can see \(R=\frac12 C_m+X\) with
\begin{equation}\label{eq:gram-decomposition-log}
\operatorname{rank}(X)\le 4kn.
\end{equation}
Fix $\delta>0$ so that
\begin{equation}\label{eq:delta-choice-log}
c_{p,k}2^p\delta^2\le \frac18.
\end{equation}
For an integer $L\ge 1$, let $\mathcal V_L$ be given by
Lemma~\ref{lem:dyadic-shell-approx-log}, and define
\[
Z_0(L):=
\mathcal V_L \oplus M_{+,0}^{1/2}\mathbb R^{4k}
\subseteq
\mathcal H_\beta\oplus\mathbb R^{4k}.
\]
Then
\(
\dim Z_0(L)\le C'_{p,k}L
\)
for some constant $C'_{p,k}>0$. Set
\(
Z(L):=\bigoplus_{t=1}^n Z_0(L).
\)
Hence
\begin{equation}\label{eq:dim-ZL-log}
\dim Z(L)\le C'_{p,k}\,nL.
\end{equation}
Fix $i\in[m]$. Since $M_{+,0}^{1/2}\mathbb R^{4k}\subseteq Z_0(L)$, the distance from the
$t$-th coordinate block of $\Theta_i$ to $Z_0(L)$ is controlled only by the Hilbert
component. Thus Lemma~\ref{lem:dyadic-shell-approx-log} yields
\[
\operatorname{dist}(\Theta_i,Z(L))^2
\le
\sum_{t=1}^n
c_{p,k}\,\operatorname{dist}\bigl(\Gamma_0(x_{i,t}),\mathcal V_L\bigr)^2\le
\sum_{t=1}^n
\Bigl(
c_{p,k}2^p\delta^2 |x_{i,t}|^p + c_{p,k}C_{p,k}2^{-Lp}
\Bigr).
\]
Using \eqref{eq:coordinate-sum-bound-log} and \eqref{eq:delta-choice-log}, we obtain
\[
\operatorname{dist}(\Theta_i,Z(L))^2
\le
\frac18 + C''_{p,k}\,n2^{-Lp}
\]
for some constant $C''_{p,k}>0$. Choose
\(
L:=\max\left\{1,\left\lceil \frac1p\log_2(8C''_{p,k}n)\right\rceil\right\}.
\)
Then
\(
C''_{p,k}n2^{-Lp}\le \frac18,
\)
and hence
\begin{equation}\label{eq:theta-near-ZL-log}
\max_{1\le i\le m}\operatorname{dist}(\Theta_i,Z(L))^2\le \frac14.
\end{equation}
Define
\(
W:=\operatorname{span}(\overline\Theta,Z(L)).
\)
Then
\(
\dim W\le \dim Z(L)+1.
\)
Moreover, for each $i$,
\[
\operatorname{dist}(\widetilde\Theta_i,W)\le \operatorname{dist}(\Theta_i,Z(L)).
\]
Indeed, if $z\in Z(L)$, then $z-\overline\Theta\in W$, and
\(
\widetilde\Theta_i-(z-\overline\Theta)=\Theta_i-z.
\)
Now apply Proposition~\ref{prop:projection-logn} to the centered family
\(
\{\widetilde\Theta_1,\dots,\widetilde\Theta_m\},
\)
with
\(
a=2,
\) \(
r\le 4kn,
\)
and
\(
s=\dim W.
\)
Using \eqref{eq:gram-decomposition-log}, we get
\[
\max_{1\le i\le m}\operatorname{dist}(\widetilde\Theta_i,W)^2
\ge
\frac{m-4kn-1-\dim W}{2m}.
\]
Combining this with \eqref{eq:theta-near-ZL-log}, we obtain
\[
\frac{m-4kn-1-\dim W}{2m}\le \frac14.
\]
Since $\dim W\le \dim Z(L)+1$, this gives
\[
\frac{m-4kn-\dim Z(L)-2}{2m}\le \frac14.
\]
Therefore
\(
m\le 2\bigl(4kn+\dim Z(L)+2\bigr).
\)
Using \eqref{eq:dim-ZL-log}, we conclude that
\[
m\le 2\bigl(4kn+C'_{p,k}nL+2\bigr).
\]
Since
\(
L=\max\left\{1,\left\lceil \frac1p\log_2(8C''_{p,k}n)\right\rceil\right\},
\)
we have
\[
m\le C_{p,k}n\log(2n)
\]
for some \(C_{p,k}>0\). This finishes the proof.

\end{proof}

\section{Equilateral sets in \((\T^{n},d_p)\): improved upper bounds}
In this section we prove Theorem~\ref{thm:torus-nlogn-ple2}.  Compared with the
Euclidean problem, the toroidal problem has an additional obstruction: in one
dimension the cyclic distance
\(
d_{\mathbb T}(x,y)=\min\{|x-y|,1-|x-y|\}
\)
is only piecewise a real-line distance.  It switches branches at the antipodal
point \(y=x+1/2\), where \(d_{\mathbb T}(x,y)^p\) has a corner.  This loss of
smoothness produces an additional wrap-around contribution in the associated
kernel, so the sign-controlled kernel decompositions used in \(\mathbb R^n\) do
not apply directly.  Our proof first shows that the small-distance case reduces
to the Euclidean setting, so that it remains to treat equilateral sets with
large common distance.  We then reduce the argument to a one-dimensional
estimate for the anchored torus kernel.  This kernel is split into a real-line
part and a genuinely toroidal correction, the latter capturing precisely the
corner contribution near the antipodal point.  The real-line part is controlled
by the Hilbert-space approximation developed earlier, while the toroidal
correction is reduced to one-sided corner kernels and estimated by low-rank
approximation with small positive trace error.  After summing over the
coordinates, this gives the required rank comparison and proves the theorem.

\subsection{Equilateral sets in \((\mathbb{R}^{n},\|\cdot\|_p)\) and \((\T^{n},d_p)\)}\label{sec: equil R T}

We begin with a simple embedding observation showing that the torus problem is at
least as hard as the corresponding Euclidean problem.
We also explain why on the torus it is enough to treat the regime of
large common distance.
\begin{lemma}\label{lem:small-distance-transfer}
Let $n\ge 1$ be an integer. For $1\le p<\infty$, the following statements hold:
	
	\begin{enumerate}
		\item[\textup{(1)}] If \(A\subseteq(\T^n,d_p)\) is an equilateral set with common distance
		\(\lambda<1/3\), then \(A\) is isometric to an equilateral set in
		\((\mathbb{R}^n,\|\cdot\|_p)\).
		
		\item[\textup{(2)}] If \(B\subseteq(\mathbb{R}^n,\|\cdot\|_p)\) is an equilateral set with
		common distance \(\lambda\le 1/2\), then \(B\) embeds isometrically into
		\((\T^n,d_p)\).
	\end{enumerate}
\end{lemma}

\begin{proof}
	For (1), translate \(A\) so that
	$\boldsymbol{x}^\ast=(1/2,\dots,1/2)\in A.$
	Let \(\boldsymbol{x}=(x_1,\dots,x_n)\in A\setminus\{\boldsymbol{x}^\ast\}\). 
	Since
	$d_{\T}(x_i,1/2)\le d_p(\boldsymbol{x},\boldsymbol{x}^\ast)= \lambda<1/3$
	for every \(i\in[n]\), we have \(x_i\in(1/6,5/6)\). Now let \(\boldsymbol{x},\boldsymbol{y}\in A\) be distinct. Then for every \(i\in[n]\),
	$1-|x_i-y_i|>1/3$,
	whereas
	$	d_{\T}(x_i,y_i)\le d_p(\boldsymbol{x},\boldsymbol{y})<1/3.$
	Thus the second branch in the definition of \(d_{\T}\) cannot occur, and so
	$d_{\T}(x_i,y_i)=|x_i-y_i|$
	for every \(i\in[n]\). Therefore \(A\), viewed inside the cube
	\((1/6,5/6)^n\subseteq\mathbb{R}^n\), has the same pairwise distances as in
	\((\T^n,d_p)\). This proves (1).
	
	For (2), let
\(
\pi:\mathbb R^n\to\T^n
\)
be the coordinatewise quotient map. Let \(\boldsymbol{x},\boldsymbol{y}\in B\) be distinct. Since
	$|x_i-y_i|\le\|\boldsymbol{x}-\boldsymbol{y}\|_p= \lambda\le 1/2$
	for every \(i\in[n]\), we have
	$d_{\T}(\pi(x_i),\pi(y_i))=|x_i-y_i|.$
	It follows that
	$d_p(\pi(\boldsymbol{x}),\pi(\boldsymbol{y}))
	=
	\|\boldsymbol{x}-\boldsymbol{y}\|_p.$
	Thus \(\pi|_B\) is an isometric embedding of \(B\) into \((\T^n,d_p)\).
\end{proof}

As an immediate consequence of Lemma~\ref{lem:small-distance-transfer}~(2),
every equilateral set in \((\mathbb{R}^n,\|\cdot\|_p)\) can be rescaled and
embedded isometrically into \((\T^n,d_p)\) with the same cardinality. Hence
\[
e(\ell^n_p)\le e_p(\T^n).
\]
	
	Lemma~\ref{lem:small-distance-transfer} also shows that the case of small common distance can be reduced
	to the case where the common distance is \(1/4\). Indeed, let
	\(A\subseteq(\T^n,d_p)\) be an equilateral set with common distance
	\(0<\lambda\le 1/4\). By Lemma~\ref{lem:small-distance-transfer}~(1), \(A\)
	is isometric to an equilateral set \(B\subseteq(\mathbb{R}^n,\|\cdot\|_p)\)
	with the same common distance \(\lambda\). Rescaling \(B\) by the factor
	\((4\lambda)^{-1}\), we obtain an equilateral set in
	\((\mathbb{R}^n,\|\cdot\|_p)\) with common distance \(1/4\). By
	Lemma~\ref{lem:small-distance-transfer}~(2), this rescaled set embeds
	isometrically into \((\T^n,d_p)\). Thus \(A\) has the same cardinality as an
	equilateral set in \((\T^n,d_p)\) with common distance \(1/4\).
	Consequently, in proving upper bounds for equilateral sets in
	\((\T^n,d_p)\), we may and will assume that
	$\lambda\ge \frac14.$

\subsection{Reduction to a one-dimensional statement}

We write \(\T=\R/\Z\), equipped with the circular distance
\[
d_{\T}(x,y)=\min\{|x-y|,\ 1-|x-y|\},
\]
where \(x,y\) are represented in \([0,1)\). Equivalently, identifying
\(\T\) with \([-1/2,1/2)\), we write
\[
\rho(x):=d_{\T}(x,0)=|x|,
\]
where \(x\) denotes its representative in \([-1/2,1/2)\). Thus
\(d_{\T}(x,y)=\rho(x-y)\), with \(x-y\) taken modulo \(1\). For
\(\boldsymbol{x},\boldsymbol{y}\in\T^n\), set
\[
d_p(\boldsymbol{x},\boldsymbol{y})
=
\left(\sum_{t=1}^n \rho(x_t-y_t)^p\right)^{1/p}.
\]

Throughout this section, for a real symmetric matrix \(M\), we write
\[
\operatorname{tr}_+(M):=\sum_{\lambda_i(M)>0}\lambda_i(M),
\]
where eigenvalues are counted with multiplicity. Define
\[
\vartheta_p :=
\begin{cases}
0,&1\le p\le 2,\\[1mm]
\dfrac{p-2}{3p-2},&p>2.
\end{cases}
\]
Notice that, for \(p>2\),
\(
\frac1{1-\vartheta_p}
=
\frac32-\frac1p.
\)

For \(x,y\in\T\), define the anchored kernel
\[
B_p(x,y):=
\frac12\bigl(\rho(x)^p+\rho(y)^p-\rho(x-y)^p\bigr).
\]

The key auxiliary tool is the following one-dimensional proposition.

\begin{prop}\label{prop:one-dim-ple2}
Let \(p\ge1\). Let
\(
X=\{x_1,\dots,x_m\}\subseteq\T
\)
and let \(N\ge2\). Define
\(
M_{ij}:=B_p(x_i,x_j)\) for
\(1\le i,j\le m.
\)
For every \(\varepsilon>0\), there exists \(C_{p,\varepsilon}>0\) such that there
are real symmetric matrices \(L,R\in\R^{m\times m}\) satisfying
\begin{enumerate}
\item[\textup{(1)}]
\(
M\preceq L+R.
\)

\item[\textup{(2)}]
\(
\operatorname{rank}L
\le
C_{p,\varepsilon}\bigl(m^{\vartheta_p}+\log(2m)+\log(2N)\bigr).
\)

\item[\textup{(3)}]
\(
\operatorname{tr}_+(R)
\le
\varepsilon\sum_{i=1}^m \rho(x_i)^p
+
C_{p,\varepsilon}\frac{m}{N}.
\)
\end{enumerate}
\end{prop}

We first show how Proposition~\ref{prop:one-dim-ple2} implies
Theorem~\ref{thm:torus-nlogn-ple2}.

\begin{proof}[Proof of Theorem~\ref{thm:torus-nlogn-ple2} assuming Proposition~\ref{prop:one-dim-ple2}]
Let
\[
A=\{\boldsymbol{x}_0,\boldsymbol{x}_1,\dots,\boldsymbol{x}_m\}
\subseteq(\T^n,d_p)
\]
be an equilateral set with common distance \(\lambda\). By translation, assume
\(\boldsymbol{x}_0=\boldsymbol{0}\).

As explained at the end of Subsection~\ref{sec: equil R T},
we may
assume that
\(
\lambda\ge \frac14.
\)

For each coordinate \(t\in[n]\), put
\(
(B_t)_{ij}:=B_p(x_{i,t},x_{j,t})\) for \(1\le i,j\le m.
\)
Since
\(
d_p(\boldsymbol{x}_i,\boldsymbol{0})=\lambda\) and
\(
d_p(\boldsymbol{x}_i,\boldsymbol{x}_j)=\lambda\) for \(i\neq j,\)
we get
\begin{equation}\label{eq:torus-anchored-id-ple2}
\sum_{t=1}^n B_t
=
\frac{\lambda^p}{2}(I_m+J_m).
\end{equation}

Fix \(\varepsilon_*>0\), to be chosen later. Choose
\(
N:=\lceil K n\rceil+2,
\)
where \(K=K(p,\varepsilon_*)\) is sufficiently large that
\(
C_{p,\varepsilon_*}\frac{n}{N}\le \varepsilon_*4^{-p}.
\)
Apply Proposition~\ref{prop:one-dim-ple2} to each \(B_t\). Thus there are
symmetric matrices \(L_t,R_t\in\R^{m\times m}\) such that
\(
B_t\preceq L_t+R_t,
\)
\(
\operatorname{rank}L_t
\le
C_{p,\varepsilon_*}\bigl(m^{\vartheta_p}+\log(2m)+\log(2N)\bigr),
\)
and
\[
\operatorname{tr}_+(R_t)
\le
\varepsilon_*\sum_{i=1}^m \rho(x_{i,t})^p
+
C_{p,\varepsilon_*}\frac{m}{N}.
\]
Set
\(
L_B:=\sum_{t=1}^n L_t\) and
\(
R_B:=\sum_{t=1}^n R_t.
\)
Then, by \eqref{eq:torus-anchored-id-ple2},
\[
\frac{\lambda^p}{2}(I_m+J_m)\preceq L_B+R_B.
\]
Moreover, by subadditivity of \(\operatorname{tr}_+\),
\[
\operatorname{tr}_+(R_B)
\le
\varepsilon_*
\sum_{i=1}^m\sum_{t=1}^n \rho(x_{i,t})^p
+
C_{p,\varepsilon_*}\frac{mn}{N}
\le
2\varepsilon_*m\lambda^p,
\]
where we used \(\lambda\ge1/4\) and the choice of \(N\).

Let \(s=\operatorname{rank}L_B\), and let \(P\) be the orthogonal projection onto
\(\ker L_B\). Then \(\operatorname{rank}P\ge m-s\), and since \(PL_BP=0\),
\[
PR_BP
=
P(L_B+R_B)P
\succeq
P\frac{\lambda^p}{2}(I_m+J_m)P
\succeq
\frac{\lambda^p}{2}P.
\]
For every real symmetric matrix \(R\) and every orthogonal projection \(P\), one has
\[
\operatorname{tr}_+(R)\ge \operatorname{tr}(PRP).
\]
Therefore
\[
2\varepsilon_*m\lambda^p
\ge
\operatorname{tr}_+(R_B)
\ge
\operatorname{tr}(PR_BP)
\ge
\frac{\lambda^p}{2}(m-s).
\]
Taking \(\varepsilon_*=1/16\), we obtain
\(
m-s\le \frac{m}{4},
\)
and hence
\(
m\le 2s.
\)

On the other hand,
\[
s=\operatorname{rank}L_B
\le nC_{p,\varepsilon_*}\bigl(m^{\vartheta_p}+\log(2m)+\log(2N)\bigr).
\]
Thus
\begin{equation}\label{eq:torus-main-bootstrap}
m
\le
C_p n\bigl(m^{\vartheta_p}+\log(2m)+\log(2n)\bigr)
\end{equation}
for some constant \(C_{p}>0.\)

If \(1\le p\le2\), then \(\vartheta_p=0\). Hence
\(
m\le C_p n\log(2m)+C_p n\log(2n).
\)
The elementary bootstrap estimate
gives
\[
m\le C_p n\log(2n).
\]

Now suppose \(p>2\). If \(m\le n^{1/(1-\vartheta_p)}\), then the desired bound
already follows. Otherwise \(m>n^{1/(1-\vartheta_p)}\). Since \(\vartheta_p>0\),
we have
\[
\log(2m)+\log(2n)\le C_p m^{\vartheta_p}.
\]
Thus \eqref{eq:torus-main-bootstrap} gives
\(
m\le C_p n m^{\vartheta_p}.
\)
Consequently,
\[
m\le C_p n^{1/(1-\vartheta_p)}
=
C_p n^{3/2-1/p}.
\]
Since \(|A|=m+1\), the theorem follows after adjusting the constant \(C_p\).
\end{proof}

\subsection{Preliminaries for the one-dimensional proposition}

We first record a linear algebra observation.

\begin{lemma}\label{lem:centered-domination-to-domination}
Let \(A,D\in\R^{m\times m}\) be symmetric matrices and let
\(
C_m:=I_m-\frac1mJ_m.
\)
If
\(
C_mAC_m\preceq C_mDC_m,
\)
then there exists \(\boldsymbol{b}\in\R^m\) such that
\[
A\preceq D+\boldsymbol{1}\boldsymbol{b}^{\mathsf T}+\boldsymbol{b}\boldsymbol{1}^{\mathsf T}.
\]
Moreover, if \(C_mAC_m=C_mDC_m\), then \(\boldsymbol{b}\) may be chosen so that equality holds.
\end{lemma}

\begin{proof}[Proof of Lemma~\ref{lem:centered-domination-to-domination}]
Put \(S:=D-A\). The assumption says that \(C_mSC_m\succeq0\) on
\(\boldsymbol{1}^{\perp}\). Decompose
\[
\R^m=\boldsymbol{1}^{\perp}\oplus \operatorname{span}\{\boldsymbol{1}\}.
\]
The matrix \(\boldsymbol{1}\boldsymbol{b}^{\mathsf T}+\boldsymbol{b}\boldsymbol{1}^{\mathsf T}\) does not
change the restriction to \(\boldsymbol{1}^{\perp}\). By choosing the component
of \(\boldsymbol{b}\) inside \(\boldsymbol{1}^{\perp}\), we can cancel the off-diagonal block
between \(\boldsymbol{1}^{\perp}\) and \(\operatorname{span}\{\boldsymbol{1}\}\).
Then choosing the component of \(\boldsymbol{b}\) in the direction of \(\boldsymbol{1}\)
sufficiently large makes the remaining scalar block nonnegative. Hence
\[
S+\boldsymbol{1}\boldsymbol{b}^{\mathsf T}+\boldsymbol{b}\boldsymbol{1}^{\mathsf T}\succeq0,
\]
which is the desired domination. If \(C_mSC_m=0\), the same block computation
allows us to choose \(\boldsymbol{b}\) so that the resulting matrix is identically zero.
\end{proof}

\begin{lemma}\label{lem:real-line-part-ple2}
Let \(p\ge1\). For every \(\varepsilon>0\), there exists
\(C_{p,\varepsilon}>0\) such that the following holds. Let
\(x_1,\dots,x_m\in[-1/2,1/2]\), and let \(N\ge2\). Define
\[
(B_{\R})_{ij}
:=
\frac12\bigl(|x_i|^p+|x_j|^p-|x_i-x_j|^p\bigr).
\]
Then there are symmetric matrices \(L_{\R},R_{\R}\in\R^{m\times m}\) such that
\begin{enumerate}
\item[\textup{(1)}]
\(
B_{\R}\preceq L_{\R}+R_{\R}.
\)

\item[\textup{(2)}]
\(
\operatorname{rank}L_{\R}\le C_{p,\varepsilon}\log(2N).
\)

\item[\textup{(3)}]
\(
R_{\R}\succeq0
\)
and
\(
\operatorname{tr}(R_{\R})
\le
\varepsilon\sum_{i=1}^m |x_i|^p
+
C_{p,\varepsilon}\frac{m}{N}.
\)
\end{enumerate}
\end{lemma}

\begin{proof}[Proof of Lemma~\ref{lem:real-line-part-ple2}]
 If \(p=2h\) is an even integer, then
\[
\frac12\bigl(|x|^{2h}+|y|^{2h}-|x-y|^{2h}\bigr)
\]
is a polynomial in \(x\) and \(y\) of degree bounded in terms of \(p\). Hence
\(B_{\R}\) has rank at most \(C_p\). Taking
\(
L_{\R}:=B_{\R}\) and
\(
R_{\R}:=0
\)
gives the desired conclusion.
Therefore, it suffices to consider the case $p\in(2k,2k+2)$.	

First suppose
\(
p\in(4q+2,4q+4)
\)
for some integer \(q\ge0\). Set
\(
d:=\left\lfloor\frac p2\right\rfloor=2q+1.
\)
Let \(A_0:=\{x_1,\dots,x_m\}\subseteq\R\), where repetitions ignored. By
Lemma~\ref{lem cnd}, there are real constants
$C_2(A_0,p), \dots, C_{2d}(A_0,p)$ such that
\[
K(x,y):=\sum_{r=1}^d C_{2r}(A_0,p)(x-y)^{2r}-|x-y|^p
\]
is conditionally negative definite on \(A_0\). Define
\(
K_{ij}:=K(x_i,x_j),\) and
\(
P_{ij}:=-\frac12\sum_{r=1}^d C_{2r}(A_0,p)(x_i-x_j)^{2r}.
\)
Since \(C_m \mathbf{1}=0\), we have
\[
C_mB_{\R}C_m
=
\frac12 C_mKC_m+C_mPC_m.
\]
By Lemma~\ref{lem:cnd posit}, \(-C_mKC_m\succeq0\), and hence
\[
C_mB_{\R}C_m\preceq C_mPC_m.
\]
Lemma~\ref{lem:centered-domination-to-domination} therefore gives a vector
\(\boldsymbol{b}\in\R^m\) such that
\[
B_{\R}\preceq P+\boldsymbol{1}\boldsymbol{b}^{\mathsf T}+\boldsymbol{b}\boldsymbol{1}^{\mathsf T}.
\]
The matrix \(P\) has rank at most \(2d+1\), and consequently
\[
\operatorname{rank}\bigl(P+\boldsymbol{1}\boldsymbol{b}^{\mathsf T}
+\boldsymbol{b}\boldsymbol{1}^{\mathsf T}\bigr)\le C_p.
\]
Thus we may take
\(
L_{\R}:=P+\boldsymbol{1}\boldsymbol{b}^{\mathsf T}+\boldsymbol{b}\boldsymbol{1}^{\mathsf T}\)
and \(
R_{\R}:=0.
\)
This proves the lemma in this case.

It remains to consider the case \(p\in(4q,4q+2)\) for some integer \(q\ge0\), with \(p\ge1\).
Let
\(\Gamma_0,\boldsymbol w_0,M_{+,0},M_{-,0}\) be the objects supplied by
Proposition~\ref{prop:zero-anchored-exact-bridge-log}, and put
\[
W:=
\begin{bmatrix}
\boldsymbol w_0(x_1)&\boldsymbol w_0(x_2)&\cdots&\boldsymbol w_0(x_m)
\end{bmatrix}
\in\R^{4q\times m}.
\]
For every \(\boldsymbol\lambda\in\boldsymbol{1}^{\perp}\), Proposition~\ref{prop:zero-anchored-exact-bridge-log}~(i) gives
\[
\boldsymbol\lambda^{\mathsf T}B_{\R}\boldsymbol\lambda
=
c_{p,q}\left\|\sum_{i=1}^m\lambda_i\Gamma_0(x_i)\right\|_{\mathcal H_\beta}^2
+
\boldsymbol\lambda^{\mathsf T}W^{\mathsf T}(M_{+,0}-M_{-,0})W\boldsymbol\lambda.
\]
By Lemma~\ref{lem:centered-domination-to-domination}, in the equality case, there
exists \(\boldsymbol{b}\in\R^m\) such that
\[
B_{\R}
=
c_{p,q}\bigl(\langle\Gamma_0(x_i),\Gamma_0(x_j)\rangle_{\mathcal H_\beta}\bigr)_{i,j=1}^m
+
W^{\mathsf T}(M_{+,0}-M_{-,0})W
+
\boldsymbol{1}\boldsymbol{b}^{\mathsf T}+\boldsymbol{b}\boldsymbol{1}^{\mathsf T}.
\]

We approximate the Hilbert part by dyadic shells. By compactness of \( \Gamma_0([1/4,1/2]\cup[-1/2,-1/4]),\) there is a finite-dimensional subspace \(V_0\subseteq\mathcal H_\beta\) such that \[ \operatorname{dist}(\Gamma_0(u),V_0)^2\le 4^{-p}\varepsilon \] for every \(u\) with \(1/4\le |u|\le1/2\). Let \(L_0:=\left\lceil\frac1p\log_2N\right\rceil+3\) and \(V_N:=\sum_{\ell=0}^{L_0}T_{2^{-\ell}}V_0.\) Then \[ \dim V_N\le C_{p,\varepsilon}\log(2N). \]
\begin{claim}
    \(\operatorname{dist}(\Gamma_0(x),V_N)^2 \le \varepsilon |x|^p+C_{p,\varepsilon}N^{-1}\) for \(|x|\le\frac{1}{2}.\)
\end{claim}
\begin{poc}
    If \(2^{-\ell-2}<|x|\le 2^{-\ell-1}\) for some \(0\le \ell\le L_0\), then \(x=2^{-\ell}u\) for some \(u\) with \(1/4<|u|\le1/2\). By Proposition~\ref{prop:zero-anchored-exact-bridge-log}~(ii), \[ \operatorname{dist}(\Gamma_0(x),V_N)^2 \le 2^{-\ell p}\operatorname{dist}(\Gamma_0(u),V_0)^2 \le 4^{-p}\varepsilon\,2^{-\ell p} \le \varepsilon |x|^p. \] If \(|x|\le 2^{-L_0-2}\), then \(0\in V_N\), and By Proposition~\ref{prop:zero-anchored-exact-bridge-log}~(ii) also gives \[ \operatorname{dist}(\Gamma_0(x),V_N)^2 \le \|\Gamma_0(x)\|_{\mathcal H_\beta}^2 \le C_p |x|^p \le C_{p}2^{-L_0p} \le C_{p}N^{-1}. \] This proves the claim.
\end{poc}

Let \(P_N\) be the orthogonal projection onto \(V_N\). Define
\[
(L_{\R})_{ij}
:=
c_{p,q}\langle P_N\Gamma_0(x_i),P_N\Gamma_0(x_j)\rangle_{\mathcal H_\beta}
+
\bigl(W^{\mathsf T}(M_{+,0}-M_{-,0})W\bigr)_{ij}
+
b_i+b_j,
\]
and set
\(
R_{\R}:=B_{\R}-L_{\R}.
\)
Then
\[
(R_{\R})_{ij}
=
c_{p,q}
\langle (I-P_N)\Gamma_0(x_i),(I-P_N)\Gamma_0(x_j)\rangle_{\mathcal H_\beta},
\]
so \(R_{\R}\succeq0\). Moreover,
\[
\operatorname{rank}L_{\R}
\le
\dim V_N+C_p
\le
C_{p,\varepsilon}\log(2N),
\]
and
\[
\operatorname{tr}(R_{\R})
=
c_{p,q}\sum_{i=1}^m
\|(I-P_N)\Gamma_0(x_i)\|_{\mathcal H_\beta}^2
\le
\varepsilon\sum_{i=1}^m |x_i|^p
+
C_{p,\varepsilon}\frac{m}{N}.
\]
This completes the proof.
\end{proof}

\subsection{The toroidal correction}

We next treat the genuinely toroidal correction.

\begin{lemma}\label{lem:volterra}
Let \(\psi:[0,1]\to\R\) be absolutely continuous, let \(\psi(0)=0\), and assume
that \(\psi'\) has bounded variation on \([0,1]\). Then there exists a constant
\(A_\psi<\infty\) such that the following holds. For every \(0<\tau\le1\),
every
\(
u_1,\dots,u_r, v_1,\dots,v_s\in[0,\tau],
\)
and every integer \(N\ge1\), the matrix
\(
M_{ij}:=\psi((v_j-u_i)_+)
\)
admits a decomposition
\(
M=E+H
\)
such that
\(
\|E\|_{S_1}\le A_\psi\frac{\tau\sqrt{rs}}{N},
\)
and
\(
\operatorname{rank}H\le N.
\)
Moreover, after increasing \(A_\psi\) if necessary, one has
\(
|\psi(t)|\le A_\psi t\) for
\( 0\le t\le 1.
\)
\end{lemma}

\begin{proof}[Proof of Lemma~\ref{lem:volterra}]
Choose a representative of \(\psi'\) of bounded variation, and put
\(
L_\psi:=\|\psi'\|_{L^\infty(0,1)}\) and
\(
V_\psi:=\operatorname{Var}_{[0,1]}(\psi').
\)
Since \(\psi(0)=0\), the Lipschitz bound follows from absolute continuity once
\(A_\psi\ge L_\psi\).

Fix \(0<\tau\le1\). Define a function \(g_\tau\) on \([-2,2]\) by
\[
g_\tau(t)=
\begin{cases}
0,& -2\le t\le0,\\
\psi(\tau t),&0\le t\le1,\\
(2-t)\psi(\tau),&1\le t\le2.
\end{cases}
\]
Extend \(g_\tau\) to a continuous \(4\)-periodic function on \(\R\). On
\([-1,1]\), we have
\(
g_\tau(t)=\psi((\tau t)_+).
\)
The distributional second derivative of this periodic function is a finite
signed periodic measure, and
\[
\|g_\tau''\|_{\mathrm{TV}(\R/4\Z)}
\le
C\tau(L_\psi+V_\psi).
\]
Also,
\(
\|g_\tau\|_{L^\infty([-2,2])}\le L_\psi\tau.
\)
Let
\(
g_\tau(t)=\sum_{\ell\in\Z}\widehat g_\tau(\ell)e^{\pi \ii\ell t/2}
\)
be the Fourier series on the period-four circle. For \(\ell\ne0\), two
integrations by parts in the sense of distributions give
\[
|\widehat g_\tau(\ell)|
\le
\frac{C\tau(L_\psi+V_\psi)}{\ell^2}.
\]
Moreover, \(|\widehat g_\tau(0)|\le L_\psi\tau\). Hence, for all \(K\ge0\),
\[
\sum_{|\ell|>K}|\widehat g_\tau(\ell)|
\le
\frac{B_\psi\tau}{K+1},
\]
where \(B_\psi\) depends only on \(\psi\).

For \(u,v\in[0,\tau]\), the number \((v-u)/\tau\) belongs to \([-1,1]\). Hence
\[
\psi((v-u)_+)
=
g_\tau\left(\frac{v-u}{\tau}\right)
=
\sum_{\ell\in\Z}
\widehat g_\tau(\ell)
e^{-\pi \ii\ell u/(2\tau)}
e^{\pi \ii\ell v/(2\tau)}.
\]
Let
\(
K:=\left\lfloor\frac{N-1}{2}\right\rfloor,
\)
and let \(H\) be the real matrix obtained by keeping only the modes
\(|\ell|\le K\). The term \(\ell=0\) has rank at most \(1\), and the modes
\(\ell\) and \(-\ell\) together have real rank at most \(2\). Thus
\[
\operatorname{rank}H\le2K+1\le N.
\]
The remainder is a sum of separated rank-one matrices. Since each separated
vector has Euclidean norm \(\sqrt r\) or \(\sqrt s\), we get
\[
\|E\|_{S_1}
\le
\sum_{|\ell|>K}|\widehat g_\tau(\ell)|\sqrt{rs}
\le
A_\psi\frac{\tau\sqrt{rs}}{N}.
\]
This proves the lemma.
\end{proof}

For \( 0\le t\le 1,\) define
\[
\rho_p(t):=2^{-p}\bigl((1+t)^p-(1-t)^p\bigr).
\]

\begin{lemma}\label{lem:rectangular-corner-rank}
Let \(m\ge1\), let \(r+s\le m\), and let
\(
u_1,\dots,u_r,v_1,\dots,v_s\in[0,1].
\)
Define
\(
M_{ij}:=\rho_p\bigl((v_j-u_i)_+\bigr).
\)
For every \(\eta\in(0,1]\), there is a decomposition
\(
M=E+H
\)
such that
\(
\|E\|_{S_1}
\le
\eta\left(r+\sum_{j=1}^s v_j^p\right),
\)
and
\(
\operatorname{rank}H
\le
C_{p,\eta}\bigl(m^{\vartheta_p}+\log(2m)\bigr).
\)
\end{lemma}

\begin{proof}[Proof of Lemma~\ref{lem:rectangular-corner-rank}]
The function \(\rho_p\) satisfies the assumptions of Lemma~\ref{lem:volterra}.
Indeed, \(\rho_p(0)=0\), \(\rho_p\) is absolutely continuous, and
\[
\rho_p'(t)
=
p2^{-p}\bigl((1+t)^{p-1}+(1-t)^{p-1}\bigr),
\]
with the obvious interpretation when \(p=1\). Hence \(\rho_p'\) has bounded
variation on \([0,1]\). Let \(A_p:=A_{\rho_p}\) be the constant from
Lemma~\ref{lem:volterra}. We then split the proof into the three regimes \(1\le p<2\), \(p=2\), and \(p>2\).

First assume \(1\le p<2\). Put
\[
\tau_*:=\min\left\{1,\frac{\eta}{16A_pm}\right\}.
\]
All columns with \(v_j\le\tau_*\) are put into the error. Since
\(
|\rho_p(t)|\le A_pt,
\)
we get
\[
\|E_{\mathrm{bot}}\|_{S_1}
\le
A_p\tau_*rs
\le
\frac{\eta r}{16}.
\]
For the remaining columns, use dyadic column scales \(C_\ell\subseteq(\tau_*,1]\). For each scale, denote its upper endpoint by \(\tau_\ell\), so that \( \frac{\tau_\ell}{2}<v_j\le \tau_\ell\) for every \(j\in C_\ell, \) except possibly for the first or last truncated scale, which only changes the absolute constants. Set \(R_\ell:=\{i:u_i<\tau_\ell\},\) \(r_\ell:=|R_\ell|,\) \(s_\ell:=|C_\ell|,\) and \( B_\ell:=\sum_{j\in C_\ell}v_j^p.\) The block \(R_\ell\times C_\ell\) contains all nonzero entries from this scale. Moreover, \[ s_\ell\tau_\ell^p \le 2^p\sum_{j\in C_\ell}v_j^p = 2^pB_\ell. \]
Apply Lemma~\ref{lem:volterra} to this block with
\(
N_0:=\left\lceil C_pA_p\eta^{-1}\right\rceil.
\)
Then
\[
\|E_\ell\|_{S_1}
\le
A_p\frac{\tau_\ell\sqrt{r_\ell s_\ell}}{N_0}.
\]
Now
\[
\tau_\ell\sqrt{r_\ell s_\ell}
=
\sqrt{\tau_\ell^{2-p}r_\ell}\sqrt{s_\ell\tau_\ell^p}
\le
C_p\bigl(\tau_\ell^{2-p}r_\ell+B_\ell\bigr).
\]
Since \(2-p>0\), the dyadic sum is summable:
\[
\sum_\ell \tau_\ell^{2-p}r_\ell
=
\sum_{i=1}^r\sum_{\ell:u_i<\tau_\ell}\tau_\ell^{2-p}
\le
C_pr.
\]
Also,
\(
\sum_\ell B_\ell\le\sum_{j=1}^s v_j^p.
\)
Thus, after choosing the constant in \(N_0\) sufficiently large,
\[
\sum_\ell\|E_\ell\|_{S_1}
\le
\frac{\eta}{2}\left(r+\sum_{j=1}^s v_j^p\right).
\]
Together with the bottom scale, this gives the required error estimate. The rank
is at most a constant depending on \(p,\eta\) times the number of scales. Since
\(\tau_*^{-1}\le C_{p,\eta}m\), we get
\[
\operatorname{rank}H\le C_{p,\eta}\log(2m).
\]

Next assume \(p=2\). Then \(\rho_2(t)=t\). Put
\(
\tau_*:=\min\left\{1,\frac{\eta}{16m}\right\}.
\)
The columns with \(v_j\le\tau_*\) are put into the error, giving
\[
\|E_{\mathrm{bot}}\|_{S_1}
\le
\tau_*rs
\le
\frac{\eta r}{16}.
\]
For each non-bottom dyadic column scale \(C_\ell\), denote its upper and lower endpoints by \(\tau_\ell\) and \(\underline{\tau}_\ell\), and put \( s_\ell:=|C_\ell|\) and \(B_\ell:=\sum_{j\in C_\ell}v_j^2.\) Split the relevant rows as
\(
R_\ell^-:=\{i:u_i\le\underline{\tau}_\ell\}\)
and \(
R_\ell^+:=\{i:\underline{\tau}_\ell<u_i<\tau_\ell\}.
\)
Rows with \(u_i\ge\tau_\ell\) give zero entries on \(C_\ell\). If
\(i\in R_\ell^-\) and \(j\in C_\ell\), then \(u_i<v_j\), and hence
\(
(v_j-u_i)_+=v_j-u_i.
\)
Thus the block \(R_\ell^-\times C_\ell\) has rank at most \(2\), and is put
entirely into the low-rank part.

On \(R_\ell^+\times C_\ell\), apply Lemma~\ref{lem:volterra} to \(\psi(t)=t\) with
\(
N_0:=\left\lceil C\eta^{-1}\right\rceil.
\)
Writing \(r_\ell^+:=|R_\ell^+|\), we obtain
\[
\|E_\ell\|_{S_1}
\le
C\frac{\tau_\ell\sqrt{r_\ell^+s_\ell}}{N_0}.
\]
Since
\[
\tau_\ell\sqrt{r_\ell^+s_\ell}
\le
\frac12r_\ell^+ + \frac12s_\ell\tau_\ell^2
\le
\frac12r_\ell^+ + C B_\ell,
\]
and since the row bands \(R_\ell^+\) are disjoint, we have
\(
\sum_\ell r_\ell^+\le r.
\)
Also \(\sum_\ell B_\ell\le\sum_jv_j^2\). Hence the total error is bounded by
\(
\eta\left(r+\sum_jv_j^2\right),
\)
after increasing the constant in \(N_0\). The rank contribution per scale is at
most \(2+N_0\), so
\(
\operatorname{rank}H\le C_\eta\log(2m).
\)

Finally assume \(p>2\). Put
\(
\theta:=\vartheta_p=\frac{p-2}{3p-2},\)
\(
M_0:=m^\theta\) and
\(
\tau_0:=m^{-2/(3p-2)}.
\)
Then
\(
\tau_0^{1-p/2}=M_0\) and
\(
\tau_0\sqrt{\frac{m}{M_0}}=M_0.
\)
Let
\[
C_0:=\{j:0\le v_j\le\tau_0\}
\]
be the bottom scale. For the remaining columns use dyadic scales
\[
C_\ell:=\{j:2^{\ell-1}\tau_0<v_j\le2^\ell\tau_0\},
\]
for \(\ell\ge1\), with the last scale truncated at \(1\). The number of non-bottom scales is at
most \(C_p\log(2m)\).

First consider the bottom scale. Put
\(
R_0:=\{i:u_i<\tau_0\},\)
\(
r_0:=|R_0|\)
and
\(
s_0:=|C_0|.
\)
If \(r_0\le M_0\), put the whole bottom block into the low-rank part, and its rank is
at most \(M_0\). If \(r_0>M_0\), apply Lemma~\ref{lem:volterra} with
\(
N_0:=\left\lceil C_pA_p\eta^{-1}M_0\right\rceil.
\)
Since \(s_0\le m\) and \(r_0>M_0\),
\[
\tau_0\sqrt{r_0s_0}
\le
\tau_0r_0\sqrt{\frac{m}{r_0}}
\le
\tau_0r_0\sqrt{\frac{m}{M_0}}
=
M_0r_0.
\]
Therefore the bottom error is at most \(\eta r_0/4\), after choosing the constant
in \(N_0\) large enough, and the bottom rank is \(O_{p,\eta}(M_0)\).

Now consider non-bottom scales \(\ell\ge1\). Define
\(
R_\ell:=\{i:u_i<\tau_\ell\},\)
\(
r_\ell:=|R_\ell|,\)
\(
s_\ell:=|C_\ell|\)
and
\(
B_\ell:=\sum_{j\in C_\ell}v_j^p.
\)
As before,
\(
s_\ell\tau_\ell^p\le C_pB_\ell.
\)
Set
\(
\gamma:=\frac{p-2}{3},\)
\(
S:=\sum_{\ell\ge1}\tau_\ell^{-\gamma}\)
and
\(
\alpha_\ell:=\frac{\tau_\ell^{-\gamma}}{S}.
\)
If there are no non-bottom scales, there is nothing to prove. Otherwise
\(\sum_\ell\alpha_\ell=1\). Apply Lemma~\ref{lem:volterra} on the \(\ell\)-th
block with
\(
N_\ell:=
\left\lceil
C_pA_p\eta^{-1}
\frac{\tau_\ell^{1-p/2}}{\sqrt{\alpha_\ell}}
\right\rceil.
\)
Then
\[
\|E_\ell\|_{S_1}
\le
A_p\frac{\tau_\ell\sqrt{r_\ell s_\ell}}{N_\ell}
\le
C_p\eta\sqrt{\alpha_\ell r_\ell}\sqrt{s_\ell\tau_\ell^p}.
\]
Using \(2ab\le a^2+b^2\), we obtain
\[
\|E_\ell\|_{S_1}
\le
\frac{\eta}{4}\alpha_\ell r_\ell
+
C_p\eta B_\ell.
\]
Summing in \(\ell\), and using
\(
\sum_\ell \alpha_\ell r_\ell\le r\) and
\(
\sum_\ell B_\ell\le\sum_jv_j^p,
\)
the total non-bottom error is bounded by
\(
\frac{\eta}{2}\left(r+\sum_jv_j^p\right),
\)
after adjusting constants.

It remains to estimate the rank. We have
\[
\sum_\ell N_\ell
\le
C_p\log(2m)
+
C_{p,\eta}
\sum_\ell
\frac{\tau_\ell^{1-p/2}}{\sqrt{\alpha_\ell}}.
\]
Since
\(
\alpha_\ell^{-1/2}=S^{1/2}\tau_\ell^{\gamma/2}
\)
and
\(
1-\frac p2+\frac{\gamma}{2}=-\gamma,
\)
we get
\[
\sum_\ell
\frac{\tau_\ell^{1-p/2}}{\sqrt{\alpha_\ell}}
=
S^{1/2}\sum_\ell\tau_\ell^{-\gamma}
=
S^{3/2}.
\]
The dyadic sum is dominated by its first term:
\(
S\le C_p\tau_0^{-\gamma}.
\)
Therefore
\[
S^{3/2}
\le
C_p\tau_0^{-3\gamma/2}
=
C_p\tau_0^{-(p-2)/2}
=
C_pm^{(p-2)/(3p-2)}
=
C_pM_0.
\]
Thus the non-bottom rank is at most
\(
C_{p,\eta}\bigl(M_0+\log(2m)\bigr).
\)
Including the bottom scale gives the same bound. This completes the proof.
\end{proof}

\begin{prop}\label{prop:strong-torus-corner}
Let \(X=(z_1,\dots,z_m)\subseteq [-1/2,1/2]\), and define
\(
W_p(x,y):=|x-y|^p-\rho(x-y)^p.
\)
Let
\(
W_p(X):=\bigl(W_p(z_\alpha,z_\beta)\bigr)_{\alpha,\beta=1}^m.
\)
For every \(\varepsilon\in(0,1]\), there is a decomposition
\(
W_p(X)=F+G
\)
into real symmetric matrices such that
\(
\operatorname{tr}_+(F)
\le
\varepsilon\sum_{\ell=1}^m |z_\ell|^p
\)
and
\(
\operatorname{rank}G
\le
C_{p,\varepsilon}\bigl(m^{\vartheta_p}+\log(2m)\bigr).
\)
\end{prop}

\begin{proof}[Proof of Proposition~\ref{prop:strong-torus-corner}]
Reorder the sample so that the positive points come first and the negative
points second. Points equal to \(0\) give zero rows and columns in \(W_p(X)\), and
may be ignored.

Write the positive points as
\(
x_i=\frac{a_i}{2},\) with \(a_i\in(0,1] \),
and the negative points as
\(
y_j=-\frac{b_j}{2},\) with \(b_j\in(0,1].\)

Same-sign pairs give no contribution, since then \(|x-y|\le1/2\) and hence
\(\rho(x-y)=|x-y|\).

For an opposite-sign pair \(x=a/2>0\), \(y=-b/2<0\), one has
\[
|x-y|=\frac{a+b}{2}.
\]
If \(a+b\le1\), then again \(\rho(x-y)=|x-y|\), so \(W_p(x,y)=0\). If
\(a+b>1\), then
\[
\rho(x-y)=1-\frac{a+b}{2}=\frac{2-a-b}{2}.
\]
Therefore, in all cases,
\[
W_p(x,y)
=
\rho_p\bigl((a+b-1)_+\bigr),
\]
where
\(
\rho_p(t)=2^{-p}\bigl((1+t)^p-(1-t)^p\bigr).
\) Thus, after the above ordering,
\[
W_p(X)=
\begin{pmatrix}
0&D\\
D^{\mathsf T}&0
\end{pmatrix},
\]
where
\(
D_{ij}=\rho_p\bigl((a_i+b_j-1)_+\bigr).
\)

Split the positive and negative indices into
\[
I_{\mathrm{hi}}:=\{i:a_i\ge1/2\},
\qquad
I_{\mathrm{lo}}:=\{i:a_i<1/2\},
\]
and
\[
J_{\mathrm{hi}}:=\{j:b_j\ge1/2\},
\qquad
J_{\mathrm{lo}}:=\{j:b_j<1/2\}.
\]
On \(I_{\mathrm{lo}}\times J_{\mathrm{lo}}\), one has \(a_i+b_j<1\), and hence
the corresponding block of \(D\) vanishes.

First consider the block \(I_{\mathrm{hi}}\times\{1,\dots,s\}\). Put
\(
u_i:=1-a_i\) and
\(
v_j:=b_j.
\)
Then
\(
a_i+b_j-1=v_j-u_i,
\)
and the block has the form
\(
\rho_p\bigl((v_j-u_i)_+\bigr).
\)
Applying Lemma~\ref{lem:rectangular-corner-rank} with parameter \(\eta\in(0,1]\),
we get a decomposition of this block into \(E_1+H_1\), with
\[
\|E_1\|_{S_1}
\le
\eta\left(|I_{\mathrm{hi}}|+\sum_jb_j^p\right)
\]
and
\(
\operatorname{rank}H_1
\le
C_{p,\eta}\bigl(m^{\vartheta_p}+\log(2m)\bigr).
\)
Since \(a_i\ge1/2\) on \(I_{\mathrm{hi}}\),
\(
|I_{\mathrm{hi}}|
\le
2^p\sum_{i\in I_{\mathrm{hi}}}a_i^p.
\)
Thus
\[
\|E_1\|_{S_1}
\le
C_p\eta\left(\sum_i a_i^p+\sum_j b_j^p\right).
\]

Next consider the block \(I_{\mathrm{lo}}\times J_{\mathrm{hi}}\). Transpose it
and put
\(
u_j:=1-b_j\) and
\(
v_i:=a_i.
\)
Then
\(
a_i+b_j-1=v_i-u_j,
\)
so the transposed block again has the form
\(
\rho_p\bigl((v_i-u_j)_+\bigr).
\)
Lemma~\ref{lem:rectangular-corner-rank} gives a decomposition \(E_2+H_2\), with
\[
\|E_2\|_{S_1}
\le
\eta\left(|J_{\mathrm{hi}}|+\sum_{i\in I_{\mathrm{lo}}}a_i^p\right)
\]
and
\(
\operatorname{rank}H_2
\le
C_{p,\eta}\bigl(m^{\vartheta_p}+\log(2m)\bigr).
\)
Since \(b_j\ge1/2\) on \(J_{\mathrm{hi}}\),
\(
|J_{\mathrm{hi}}|
\le
2^p\sum_{j\in J_{\mathrm{hi}}}b_j^p,
\)
and therefore
\[
\|E_2\|_{S_1}
\le
C_p\eta\left(\sum_i a_i^p+\sum_j b_j^p\right).
\]

Combining the two rectangular pieces, we obtain
\(
D=E_D+H_D,
\)
with
\(
\|E_D\|_{S_1}
\le
C_p\eta\left(\sum_i a_i^p+\sum_j b_j^p\right),
\)
and
\(
\operatorname{rank}H_D
\le
C_{p,\eta}\bigl(m^{\vartheta_p}+\log(2m)\bigr).
\)

Now embed this rectangular decomposition into the full symmetric matrix:
\[
F:=
\begin{pmatrix}
0&E_D\\
E_D^{\mathsf T}&0
\end{pmatrix},
\qquad
G:=
\begin{pmatrix}
0&H_D\\
H_D^{\mathsf T}&0
\end{pmatrix}.
\]
Then
\[
W_p(X)=F+G.
\]
The nonzero eigenvalues of \(F\) are precisely the singular values of \(E_D\),
with both signs. Hence
\[
\operatorname{tr}_+(F)=\|E_D\|_{S_1}.
\]
Also,
\(
\operatorname{rank}G\le2\operatorname{rank}H_D.
\)
Consequently,
\[
\operatorname{tr}_+(F)
\le
C_p\eta\left(\sum_i a_i^p+\sum_j b_j^p\right),
\]
and
\(
\operatorname{rank}G
\le
C_{p,\eta}\bigl(m^{\vartheta_p}+\log(2m)\bigr).
\)

Finally, since \(a_i=2|x_i|\) and \(b_j=2|y_j|\), we have
\[
\sum_i a_i^p+\sum_j b_j^p
=
2^p\sum_{\ell=1}^m |z_\ell|^p.
\]
Choosing \(\eta=\varepsilon/(C_p2^p)\), we obtain
\[
\operatorname{tr}_+(F)
\le
\varepsilon\sum_{\ell=1}^m |z_\ell|^p.
\]
The rank estimate becomes
\[
\operatorname{rank}G
\le
C_{p,\varepsilon}\bigl(m^{\vartheta_p}+\log(2m)\bigr).
\]
This completes the proof.
\end{proof}

\subsection{Proof of the one-dimensional proposition}

\begin{proof}[Proof of Proposition~\ref{prop:one-dim-ple2}]
Let
\(
\varepsilon_0:=\frac{\varepsilon}{2}.
\)
We identify \(\T\) with \([-1/2,1/2)\), and write \(x_i\) also for the
representative of \(x_i\) in this interval. Thus
\(
\rho(x_i)=|x_i|.
\)
For \(x,y\in[-1/2,1/2)\), define
\[
B_{\R}(x,y):=
\frac12\bigl(|x|^p+|y|^p-|x-y|^p\bigr),
\]
and
\[
W(x,y):=
\frac12\bigl(|x-y|^p-\rho(x-y)^p\bigr).
\]
Then
\[
B_p(x,y)=B_{\R}(x,y)+W(x,y),
\]
and hence the matrix \(M\) decomposes as
\(
M=B_{\R}+W.
\)

By Lemma~\ref{lem:real-line-part-ple2}, applied with \(\varepsilon_0\), there are
real symmetric matrices \(L_{\R},R_{\R}\) such that
\(
B_{\R}\preceq L_{\R}+R_{\R},
\)
\(
\operatorname{rank}L_{\R}\le C_{p,\varepsilon}\log(2N),
\)
\(
R_{\R}\succeq0,
\)
and
\[
\operatorname{tr}_+(R_{\R})
=
\operatorname{tr}(R_{\R})
\le
\varepsilon_0\sum_{i=1}^m |x_i|^p
+
C_{p,\varepsilon}\frac{m}{N}.
\]

By Proposition~\ref{prop:strong-torus-corner}, applied with \(\varepsilon_0\), the
matrix
\[
W_p(X)=\bigl(|x_i-x_j|^p-\rho(x_i-x_j)^p\bigr)_{i,j=1}^m
\]
admits a decomposition
\(
W_p(X)=F_0+G_0
\)
such that
\(
\operatorname{tr}_+(F_0)
\le
\varepsilon_0\sum_{i=1}^m |x_i|^p,
\)
and
\[
\operatorname{rank}G_0
\le
C_{p,\varepsilon}\bigl(m^{\vartheta_p}+\log(2m)\bigr).
\]
Since \(W=\frac12 W_p(X)\), we may write
\(
W=F+G
\)
with
\(
\operatorname{tr}_+(F)
\le
\varepsilon_0\sum_{i=1}^m |x_i|^p
\)
and
\[
\operatorname{rank}G
\le
C_{p,\varepsilon}\bigl(m^{\vartheta_p}+\log(2m)\bigr).
\]

Finally set
\(
L:=L_{\R}+G\) and
\(
R:=R_{\R}+F.
\)
Then
\[
M=B_{\R}+W
\preceq
L_{\R}+R_{\R}+F+G
=
L+R.
\]
Moreover,
\[
\operatorname{rank}L
\le
\operatorname{rank}L_{\R}+\operatorname{rank}G
\le
C_{p,\varepsilon}\bigl(m^{\vartheta_p}+\log(2m)+\log(2N)\bigr).
\]
By subadditivity of \(\operatorname{tr}_+\),
\[
\operatorname{tr}_+(R)
\le
\operatorname{tr}_+(R_{\R})+\operatorname{tr}_+(F)
\le
\varepsilon\sum_{i=1}^m \rho(x_i)^p
+
C_{p,\varepsilon}\frac{m}{N}.
\]
This proves the proposition.
\end{proof}

\section{Lower bounds for \(e_{p}(\mathbb{T}^{n})\)}
\subsection{Proof of Theorem~\ref{thm:odd-prime-power-lower-bound}}
Let \( V:=\mathbb{F}_r^m\) be an $m$-dimensional vector space over $\mathbb{F}_r$, and then \( |V|=r^m=q. \)
We write
\[
V^*:=\operatorname{Hom}_{\mathbb F_r}(V,\mathbb F_r)
\]
for the dual space of $V$, namely the set of all $\mathbb F_r$-linear maps
\(
\varphi:V\to \mathbb F_r.
\)
Elements of $V^*$ are called \emph{linear functionals} on $V$. It is a standard fact that $V^*$ is again an $m$-dimensional vector space over $\mathbb F_r$.
In particular,
\(
|V^*|=r^m=q.
\)
For these basic facts on dual spaces, see for example~\cite{Roman2008}.

Since $r$ is odd, every nonzero functional $\varphi\in V^*$ is distinct from $-\varphi$.
Therefore the set $V^*\setminus\{0\}$ is partitioned into disjoint pairs \( \{\varphi,-\varphi\}.\)
Notice that \( |V^*\setminus\{0\}|=q-1=2n, \)
there are exactly $n$ such pairs.
Choose one representative from each pair and denote them by
\[
\varphi_1,\dots,\varphi_n\in V^*\setminus\{0\}.
\]
We also identify $\mathbb F_r$ with the set of integers
\(
\{0,1,\dots,r-1\},
\)
and then with the corresponding subset of $\T=\mathbb R/\mathbb Z$ via
\[
a\rightarrow \frac{a}{r}.
\]
Now define a map
\(
\eta:V\to \T^n
\)
by
\[
\eta(\boldsymbol{x})
:=
\left(
\frac{\varphi_1(\boldsymbol{x})}{r},
\frac{\varphi_2(\boldsymbol{x})}{r},
\dots,
\frac{\varphi_n(\boldsymbol{x})}{r}
\right).
\]
We claim that $\eta(V)$ is an equilateral set in $(\T^n,d_p)$.
Let $\boldsymbol{x},\boldsymbol{y}\in V$ be distinct, and write
\(
\boldsymbol{u}:=\boldsymbol{x}-\boldsymbol{y}\neq \boldsymbol{0}.
\)
Consider the evaluation map
\(
E_{\boldsymbol{u}}:V^*\to \mathbb F_r,
\)
with
\(
E_{\boldsymbol{u}}(\varphi):=\varphi(\boldsymbol{u}).
\)
Since $\boldsymbol{u}\neq \boldsymbol{0}$, this is a nonzero linear functional on the $m$-dimensional space $V^*$.
Hence $E_{\boldsymbol{u}}$ is surjective, and every fiber has cardinality
\(
|E_{\boldsymbol{u}}^{-1}(a)|=r^{m-1}
\)
for each \(a\in\mathbb F_r.\) In particular, among all functionals in $V^*$,
\begin{itemize}
\item exactly $r^{m-1}$ satisfy $\varphi(\boldsymbol{u})=a$ for each $a\in\mathbb F_r$;
\item hence, among the nonzero functionals in $V^*\setminus\{0\}$, exactly $r^{m-1}-1$ satisfy $\varphi(\boldsymbol{u})=0$, and for each nonzero $a\in\mathbb F_r^\times$, exactly $r^{m-1}$ satisfy $\varphi(\boldsymbol{u})=a$.
\end{itemize}
Now pass from nonzero functionals to the chosen representatives
\(
\varphi_1,\dots,\varphi_n,
\)
that is, to the quotient by the involution $\varphi\rightarrow -\varphi$. For each integer
\(
t\in\left\{1,2,\dots,\frac{r-1}{2}\right\},
\)
the set of nonzero field elements $\{\pm t\}$ consists of two distinct elements of $\mathbb F_r$.
Since each of the values $t$ and $-t$ occurs exactly $r^{m-1}$ times among all nonzero functionals, it follows that among the chosen representatives $\varphi_1,\dots,\varphi_n$, exactly
\(
r^{m-1}
\)
of them satisfy
\(
\varphi_j(\boldsymbol{u})\in\{t,-t\}.
\)
Also, exactly
\(
\frac{r^{m-1}-1}{2}
\)
of the chosen representatives satisfy
\(
\varphi_j(\boldsymbol{u})=0.
\)

By the definitions, for a coordinate $j$ with $\varphi_j(\boldsymbol{u})\in\{t,-t\}$, the $j$-th coordinate difference is either $t/r$ or $1-t/r$, and hence the circular distance is exactly
\(
\frac{t}{r}.
\)
Therefore
\[
d_p\bigl(\eta(\boldsymbol{x}),\eta(\boldsymbol{y})\bigr)^p
=
r^{m-1}\sum_{t=1}^{(r-1)/2}\left(\frac{t}{r}\right)^p.
\]
Thus $\eta(V)$ is indeed an equilateral set in $(\T^n,d_p)$.

It remains only to note that $\eta$ is injective. Indeed, if $\boldsymbol{x}\neq \boldsymbol{y}$, then $\boldsymbol{u}=\boldsymbol{x}-\boldsymbol{y}\neq \boldsymbol{0}$, and hence there exists some $\varphi\in V^*$ with $\varphi(\boldsymbol{u})\neq 0$.
Therefore the pair $\{\varphi,-\varphi\}$ contributes one chosen representative $\varphi_j$ with
\(
\varphi_j(\boldsymbol{u})\neq 0,
\)
so the $j$-th coordinates of $\eta(\boldsymbol{x})$ and $\eta(\boldsymbol{y})$ are different.
Thus
\(
|\eta(V)|=|V|=r^m=q=2n+1.
\)
This finishes the proof.

\subsection{Proof of Theorem~\ref{thm:lower-bound-4n-plus-1}}
For an odd prime \(q\) and an integer \(u\), write
\(
\|u\|_q\in\left\{0,1,\dots,\frac{q-1}{2}\right\}
\)
for the unique integer such that
\(
\|u\|_q\equiv \pm u \pmod q.
\)
We also write \(\chi_q\) for the quadratic character modulo \(q\), so that
\(\chi_q(0)=0\), \(\chi_q(a)=1\) if \(a\) is a nonzero quadratic residue
modulo \(q\), and \(\chi_q(a)=-1\) otherwise.

We first construct a \(q\)-point
equilateral seed from quadratic residues.
\begin{lemma}\label{lem:quadratic-seed}
Let \(q\) be a prime with \(q\equiv 1\pmod 4\). Put
\(
M:=\frac{q-1}{2}
\)
and define
\(
R_q:=\{a\in\{1,\dots,M\}:\chi_q(a)=1\}.
\)
For \(s>0\), let
\(
F_q(s):=\sum_{a=1}^{M}\chi_q(a)a^s.
\)
If \(F_q(p)=0\) for some \(p>0\), then
\[
B_q:=
\left\{
\boldsymbol{
b}_j=\left(\frac{aj}{q}\right)_{a\in R_q}:j\in\F_q
\right\}
\subseteq \T^{(q-1)/4}
\]
is a \(q\)-point equilateral set in \((\T^{(q-1)/4},d_p)\).
\end{lemma}
\begin{proof}[Proof of Lemma~\ref{lem:quadratic-seed}]
    Since \(q\equiv1\pmod4\), we have \(\chi_q(-1)=1\). Hence the two elements
\(a\) and \(-a\) have the same quadratic character. Therefore, among the
representatives
\(
1,2,\dots,M,
\)
exactly half are quadratic residues. Thus
\[
|R_q|=\frac{M}{2}=\frac{q-1}{4}.
\]
Let
\(
N_q:=\{1,\dots,M\}\setminus R_q.
\)
Notice that the points \(\boldsymbol{b}_j\) are distinct. Indeed, if \(\boldsymbol{b}_j=\boldsymbol{b}_k\), then
for every \(a\in R_q\) we have \(a(j-k)\equiv0\pmod q\). Since \(a\not\equiv0\pmod q\),
this implies \(j=k\). Hence \(|B_q|=q\). Now take \(j\neq k\), and put
\(
h:=j-k\in\F_q^\times.
\)
Then
\[
d_p(\boldsymbol{b}_j,\boldsymbol{b}_k)^p
=
q^{-p}\sum_{a\in R_q}\|ah\|_q^p.
\]
If \(\chi_q(h)=1\), multiplication by \(h\) preserves quadratic residues. Since
\(\chi_q(-1)=1\), it also preserves the quadratic character on each pair
\(\{\pm u\}\). Thus
\(
\{\|ah\|_q:a\in R_q\}=R_q.
\)
If \(\chi_q(h)=-1\), multiplication by \(h\) swaps residues and nonresidues on
the pairs \(\{\pm u\}\), and therefore
\(
\{\|ah\|_q:a\in R_q\}=N_q.
\)
Consequently,
\(
d_p(\boldsymbol{b}_j,\boldsymbol{b}_k)^p\) equals either \(
q^{-p}\sum_{u\in R_q}u^p\) or
\(
q^{-p}\sum_{u\in N_q}u^p.
\)
Notice that
\[
F_q(p)
=
\sum_{u\in R_q}u^p-\sum_{u\in N_q}u^p=0,
\]
which yields \(B_{q}\) is equilateral.
\end{proof}

The next lemma lifts a \(q\)-point seed to a larger equilateral set using the
standard projective simplex code.
\begin{lemma}\label{lem:simplex-lift-4n}
Let \(q\) be a prime, let \(p>0\), and let
\(
B=\{b_\alpha:\alpha\in\F_q\}\subseteq \T^r
\)
be a \(q\)-point equilateral set in \((\T^r,d_p)\), with common nonzero distance
\(\delta\). Then for every integer \(m\ge1\), the space
\(
\left(\T^{\,r(q^m-1)/(q-1)},d_p\right)
\)
contains an equilateral set of cardinality \(q^m\).
\end{lemma}
\begin{proof}[Proof of Lemma~\ref{lem:simplex-lift-4n}]
Let
\(
L:=\frac{q^m-1}{q-1}.
\)
Choose representatives
\(
\boldsymbol{v}_1,\dots,\boldsymbol{v}_L
\)
of the one-dimensional subspaces of \(\F_q^m\). For
\(\boldsymbol{u}\in\F_q^m\), define
\[
c(\boldsymbol{u})
:=
(\boldsymbol{u}\cdot\boldsymbol{v}_1,\dots,\boldsymbol{u}\cdot\boldsymbol{v}_L)
\in\F_q^L.
\]
We first recall the standard distance property of this simplex code. If
\(\boldsymbol{u}\ne\boldsymbol{u}'\), put
\(
\boldsymbol{w}:=\boldsymbol{u}-\boldsymbol{u}'\ne\boldsymbol{0}.
\)
The coordinates \(i\) for which
\(
\boldsymbol{u}\cdot\boldsymbol{v}_i
=
\boldsymbol{u}'\cdot\boldsymbol{v}_i
\)
are exactly those for which
\[
\boldsymbol{v}_i\in \boldsymbol{w}^{\perp}:=
\{\boldsymbol{x}\in\F_q^m:\boldsymbol{w}\cdot\boldsymbol{x}=0\}.
\]
Since \(\boldsymbol{w}\ne\boldsymbol{0}\), the hyperplane
\(\boldsymbol{w}^{\perp}\) has dimension \(m-1\), and hence contains
\(
\frac{q^{m-1}-1}{q-1}
\)
one-dimensional subspaces. Therefore the number of coordinates on which
\(c(\boldsymbol{u})\) and \(c(\boldsymbol{u}')\) differ is
\[
L-\frac{q^{m-1}-1}{q-1}=q^{m-1}.
\]

Now define
\[
\Phi(\boldsymbol{u})
:=
\bigl(
\boldsymbol{b}_{\boldsymbol{u}\cdot\boldsymbol{v}_1},
\dots,
\boldsymbol{b}_{\boldsymbol{u}\cdot\boldsymbol{v}_L}
\bigr)
\in(\T^r)^L=\T^{rL}.
\]
If \(\boldsymbol{u}\ne\boldsymbol{u}'\), then the previous paragraph shows that
exactly \(q^{m-1}\) blocks of \(\Phi(\boldsymbol{u})\) and
\(\Phi(\boldsymbol{u}')\) are different. In each different block, the two
entries are distinct points of the seed \(B\), and therefore contribute
\(\delta^p\) to the \(p\)-th power of the distance. The remaining blocks
contribute \(0\). Hence
\[
d_p(\Phi(\boldsymbol{u}),\Phi(\boldsymbol{u}'))^p
=
q^{m-1}\delta^p,
\]
independently of the choice of distinct
\(\boldsymbol{u},\boldsymbol{u}'\in\F_q^m\). Thus
\(
\{\Phi(\boldsymbol{u}):\boldsymbol{u}\in\F_q^m\}
\)
is equilateral in \(\T^{rL}\).

Moreover, if \(\boldsymbol{u}\ne\boldsymbol{u}'\), then
\(c(\boldsymbol{u})\) and \(c(\boldsymbol{u}')\) differ in \(q^{m-1}\) coordinates.
In particular, \(\Phi(\boldsymbol{u})\ne\Phi(\boldsymbol{u}')\). Therefore the
constructed equilateral set has cardinality
\(
|\F_q^m|=q^m.
\)
This completes the proof.
\end{proof}

We now show that infinitely many primes provide a seed exponent \(p_q\).

\begin{prop}\label{prop:four-legendre-large-root}
There exists a constant \(Q_0\) with the following property. Let \(q\ge Q_0\) be
a prime satisfying
\(
q\equiv1\pmod4\) and
\(
\chi_q(3)=\chi_q(5)=\chi_q(7)=\chi_q(11)=-1.
\)
Put
\(
M:=\frac{q-1}{2}
\)
and
\(
F_q(s):=\sum_{a=1}^{M}\chi_q(a)a^s.
\)
Then there exists
\(
p_q>\frac{q-1}{4}
\)
such that
\(
F_q(p_q)=0.
\)
Consequently, for every integer \(m\ge1\), if
\(
n=\frac{q^m-1}{4},
\)
then
\(
e_{p_q}(\T^n)\ge4n+1.
\)
\end{prop}
\begin{proof}[Proof of Proposition~\ref{prop:four-legendre-large-root}]
For \(x>0\), define
\(
H_q(x):=M^{-Mx}F_q(Mx).
\)
Then
\[
H_q(x)=\sum_{a=1}^{M}\chi_q(a)\left(\frac{a}{M}\right)^{Mx}.
\]
Reindex by writing \(a=M-j\), where \(0\le j\le M-1\). Since
\[
M-j=\frac{q-1}{2}-j\equiv-\frac{2j+1}{2}\pmod q,
\]
and since \(q\equiv1\pmod4\), we have \(\chi_q(-1)=1\). Put
\(
\varepsilon_q:=\chi_q(2)\in\{\pm1\}.
\)
Because \(\chi_q(2)^{-1}=\chi_q(2)\), we get
\[
\chi_q(M-j)
=
\chi_q(-1)\chi_q(2)^{-1}\chi_q(2j+1)
=
\varepsilon_q\chi_q(2j+1).
\]
Therefore
\[
\varepsilon_q H_q(x)
=
\sum_{j=0}^{M-1}\chi_q(2j+1)
\left(1-\frac{j}{M}\right)^{Mx}.
\]

We evaluate this expression at \(x=\frac12\). The assumptions give
\(
\chi_q(1)=1,\)
\(
\chi_q(3)=\chi_q(5)=\chi_q(7)=\chi_q(11)=-1\)
and 
\(
\chi_q(9)=\chi_q(3)^2=1.
\)
Hence the first six coefficients \(\chi_q(2j+1)\), for \(j=0,1,\dots,5\), are
\(1,-1,-1,-1,1,-1.\)
Thus
\[
\varepsilon_q H_q\!\left(\frac12\right)
\le
\sum_{j=0}^{5}s_j
\left(1-\frac{j}{M}\right)^{M/2}
+
\sum_{j=6}^{M-1}
\left(1-\frac{j}{M}\right)^{M/2},
\]
where
\(
(s_0,s_1,s_2,s_3,s_4,s_5)=(1,-1,-1,-1,1,-1).
\)
For each fixed \(j\),
\(
\left(1-\frac{j}{M}\right)^{M/2}\rightarrow e^{-j/2}\)
when \(M\to\infty\),
and for all \(0\le j\le M-1\),
\(
\left(1-\frac{j}{M}\right)^{M/2}\le e^{-j/2}.
\)
It follows that
\[
\limsup_{q\to\infty}
\varepsilon_q H_q\!\left(\frac12\right)
\le C,
\]
where
\(
C:=
1-e^{-1/2}-e^{-1}-e^{-3/2}+e^{-2}-e^{-5/2}
+
\sum_{j=6}^{\infty}e^{-j/2}.
\)
Since
\(
\sum_{j=6}^{\infty}e^{-j/2}
=
\frac{e^{-3}}{1-e^{-1/2}},
\)
we have
\[
C=
1-e^{-1/2}-e^{-1}-e^{-3/2}+e^{-2}-e^{-5/2}
+
\frac{e^{-3}}{1-e^{-1/2}}.
\]
Set \(t:=e^{-1/2}\). Then
\[
C
=
1-t-t^2-t^3+t^4-t^5+\frac{t^6}{1-t}
=
\frac{N(t)}{1-t},
\]
where
\(
N(t)=1-2t+2t^4-2t^5+2t^6.
\)
Notice that \(3/5<t<61/100\) and we can also compute that
\(
N'(u)=2(6u^5-5u^4+4u^3-1).
\)
Then one can check that \(N(t)\) is strictly decreasing on
\([0,61/100]\), and so
\[
N(t)<N\!\left(\frac35\right)
=
-\frac{47}{15625}<0.
\]
Since \(1-t>0\), we obtain
\(
C<0.
\)
Consequently, there exists \(Q_0\) such that for every prime \(q\ge Q_0\)
satisfying the hypotheses,
\(
\varepsilon_qH_q\!\left(\frac12\right)<0.
\)
On the other hand, as \(x\to\infty\), every term in the expression for
\(\varepsilon_qH_q(x)\) with \(j\ge1\) tends to \(0\), while the \(j=0\) term is
equal to \(1\). Thus
\[
\lim_{x\to\infty}\varepsilon_qH_q(x)=1>0.
\]
Since \(\varepsilon_qH_q(x)\) is continuous on \((0,\infty)\), the intermediate
value theorem gives some
\(
x_q>\frac12
\)
such that
\(
H_q(x_q)=0.
\)
Now set
\(
p_q:=Mx_q.
\)
Then
\(
F_q(p_q)=0
\)
and
\(
p_q>M\cdot\frac12=\frac{q-1}{4}.
\)
Finally, Lemma~\ref{lem:quadratic-seed} gives a \(q\)-point equilateral set in
\(
\T^{\frac{q-1}{4}}
\)
with respect to \(d_{p_q}\). Applying Lemma~\ref{lem:simplex-lift-4n} with
\(
r=\frac{q-1}{4}
\)
gives, for every \(m\ge1\), an equilateral set of cardinality \(q^m\) in
\(
\T^{\,\frac{q-1}{4}\cdot\frac{q^m-1}{q-1}}
=
\T^{\frac{q^{m}-1}{4}}.
\)
If \(n=\frac{q^{m}-1}{4}\), then \(q^m=4n+1\). Hence
\(
e_{p_q}(\T^n)\ge4n+1.
\)
    
\end{proof}

It remains to show that infinitely many primes satisfy the Legendre-symbol
conditions used above.

\begin{lemma}\label{lem:infinitely-many-q-4n}
There exist infinitely many primes \(q\) such that
\(
q\equiv1\pmod4\)
and
\(
\chi_q(3)=\chi_q(5)=\chi_q(7)=\chi_q(11)=-1.
\)
\end{lemma}
\begin{proof}[Proof of Lemma~\ref{lem:infinitely-many-q-4n}]
    For each \(\ell\in\{3,5,7,11\}\), choose a quadratic nonresidue
\(
r_\ell\in(\mathbb Z/\ell\mathbb Z)^\times.
\)
By the Chinese remainder theorem, there exists an integer \(r\) such that
\(
r\equiv1\pmod4\)
and
\(
r\equiv r_\ell\pmod \ell\)
for each \(\ell=3,5,7,11.\)
In particular,
\(
\gcd(r,4\cdot3\cdot5\cdot7\cdot11)=1.
\)
By Dirichlet's theorem on primes in arithmetic progressions, there are
infinitely many primes
\(
q\equiv r\pmod{4620}.
\)
For every such prime \(q\), we have \(q\equiv1\pmod4\). Hence quadratic
reciprocity gives
\[
\chi_q(\ell)
=\chi_{\ell}(q)=\chi_{\ell}(r_\ell)=
-1
\]
for \(\ell=3,5,7,11.\) This proves the lemma.
\end{proof}
Proposition~\ref{prop:four-legendre-large-root} and Lemma~\ref{lem:infinitely-many-q-4n} together yields Theorem~\ref{thm:lower-bound-4n-plus-1}.

\begin{rmk}
	For the sake of keeping the proof concise, we have restricted ourselves to a convenient infinite family of sufficiently large primes \(q\), which already suffices to show that the values of \(p\) for which
	$e_p(\T^n)\ge 4n+1$
	holds for infinitely many \(n\) can be chosen arbitrarily large.
	In fact, if one only asks for the existence of some \(p_q>1\), rather than the stronger lower bound \(p_q>(q-1)/4\), then a similar argument shows that, for every prime
	$q\equiv 1\pmod 8,$
	there exists \(p_q>1\) such that
	$e_{p_q}(\T^n)\ge 4n+1$
	for every \(m\ge 1\), where
	$n=\frac{q^m-1}{4}.$
\end{rmk}

\section{Concluding remarks}

We have proved that Kusner's conjecture holds for every \(2\le p\le4\) and every
dimension \(n\). More generally, for \(p\in[4k+2,4k+4]\), we obtain the linear
bound
\(
e(\ell_p^n)\le (2k+1)n+1,
\)
while for the complementary intervals \(p\in(4k,4k+2)\),  we prove the
almost linear estimate
\(
e(\ell_p^n)\le C_{p,k}n\log(2n).
\)

The main difficulty in Kusner's conjecture is that the usual finite-dimensional
algebraic methods are naturally adapted to even exponents. When \(p\) is an even
integer, the kernel \(|x-y|^p\) is polynomial, and one can hope to control the
corresponding distance matrix by rank considerations. Away from the even
integers, however, the kernel is no longer finite rank in any direct sense. This
is the central obstruction behind the problem: one needs a way to extract enough
finite-dimensional structure from a non-polynomial kernel.

Our approach shows that this obstruction can be overcome on the intervals
\([4k+2,4k+4]\). The key point is not simply to approximate \(|x-y|^p\), but to
decompose it with the correct signs. More precisely, after adding a suitable
one-dimensional correction, the remaining kernel is conditionally negative
definite, while the positive oscillatory part has controlled rank. Once the
distance matrix is double-centered, the conditionally negative definite part
becomes positive semidefinite and can be absorbed into the favorable side of the
rank comparison. This sign-sensitive decomposition is what makes a linear bound
possible beyond the isolated even cases.

The complementary intervals \((4k,4k+2)\) appear to be genuinely harder. In this
range, the same finite rank correction no longer has the required sign pattern.
Our proof instead keeps an infinite-dimensional Hilbert space component and
then approximates it on dyadic shells. The logarithmic factor in
\(C_{p,k}n\log(2n)\) comes exactly from this multiscale approximation. Thus the
remaining problem is not merely a matter of sharpening constants, and it asks for a
new way to control all scales at once, or for a stronger use of the global
equilateral structure.

\begin{ques}
Let  \(p\in(4k,4k+2)\) and \(p\geq 1\). Is it true that
\(
e(\ell_p^n)\le C_{p}n
\)
for all \(n\ge1\)?
\end{ques}

A positive answer would show that the equilateral number is linear throughout
the whole super-Euclidean range \(p>2\). At
present, the main obstacle is the lack of a finite-dimensional decomposition
with the correct sign pattern on these intervals.

We also prove new upper bounds for the toroidal problem, namely
\(
e_p(\mathbb T^n)\le C_p n\log(2n)
\)
for \(p\in[1,2]\), and
\(
e_p(\mathbb T^n)\le C_p n^{3/2-1/p}
\)
for every fixed \(p>2\).
The toroidal problem introduces a different difficulty. Even after anchoring at
the origin, the cyclic distance creates a wrap-around correction which has no
analogue in \(\mathbb R^n\). This correction is supported near the opposite
corners of the circle and is not controlled by the Euclidean kernel alone. Our
proof handles it by reducing the problem to one-sided rectangular kernels and
then decomposing these kernels into a low-rank part plus a small positive-trace
error. This mechanism is strong enough to improve Alon's polynomial bounds for
all finite \(p\), but it still falls short of the expected linear estimate.

\begin{ques}
For every fixed \(p\ge 1\), is there a constant \(C_p\) such that
\(
e_p(\mathbb T^n)\le C_p n
\)
for all \(n\ge1\)?
\end{ques}

Even obtaining such a linear bound in the endpoint cases \(p=1\) and \(p=2\)
seems to require new ideas.  Our bound
\(
e_p(\mathbb T^n)\le C n\log(2n)
\)
for \(p\in[1,2]\) is within a logarithmic factor of the linear bound asked for
above.  It remains unclear, however, whether this logarithmic loss is merely an
artifact of the method or a genuine feature of the toroidal setting.

Finally, the methods developed here suggest a broader principle. In several
places, a global equilateral condition is transformed, after centering, into a
rank or positive-trace comparison through a carefully chosen kernel
decomposition. It would be interesting to see whether similar sign-sensitive
kernel decompositions can be used in other extremal problems where polynomial
methods capture only the even exponent or finite rank cases.

\bibliographystyle{abbrv}
\bibliography{Torus}
\end{document}